\documentclass[11pt]{article}
\usepackage[utf8]{inputenc}  
\usepackage[T1]{fontenc}
\usepackage{amsmath}
\usepackage{amsfonts,amssymb,verbatim}
\usepackage{amsthm}
\usepackage{mathrsfs}
\usepackage{array}
\usepackage{multicol}
\usepackage{moreverb}
\usepackage{pdfpages}
\usepackage{xcolor}
\usepackage{indentfirst}
\usepackage{caption}
\usepackage{enumerate}
\usepackage{hyperref}  
\usepackage{pgf,tikz}
\usepackage{mathrsfs}
\usepackage{stmaryrd}
\usetikzlibrary{arrows}               
 \hypersetup{
    hyperfigures = true,
    colorlinks = true,
    linkcolor=blue
    }
 
 \usepackage[hmargin={2.5cm,2.5cm},top=0.8cm,bottom=3.8cm]{geometry}
\usepackage{charter}
\usepackage{graphicx}
\usepackage{aeguill,lmodern}\usepackage{bclogo}

\usepackage{fancyhdr}

\theoremstyle{theorem}
\newtheorem{thmt}{Theorem}[section]
\newtheorem{lem}[thmt]{Lemma}
\newtheorem{prop}[thmt]{Proposition}

\newtheorem{rem}{Remark}[section]
\newtheorem*{lem*}{Lemme}
\newtheorem{cor}[thmt]{Corollary}
\newtheorem{thm}{Theorem}

\makeatletter
\DeclareRobustCommand{\cev}[1]{%
  \mathpalette\do@cev{#1}%
}
\newcommand{\do@cev}[2]{%
  \fix@cev{#1}{+}%
  \reflectbox{$\m@th#1\vec{\reflectbox{$\fix@cev{#1}{-}\m@th#1#2\fix@cev{#1}{+}$}}$}%
  \fix@cev{#1}{-}%
}
\newcommand{\fix@cev}[2]{%
  \ifx#1\displaystyle
    \mkern#23mu
  \else
    \ifx#1\textstyle
      \mkern#23mu
    \else
      \ifx#1\scriptstyle
        \mkern#22mu
      \else
        \mkern#22mu
      \fi
    \fi
  \fi
}

\makeatother

\def\N{{\mathbb N}}
\def\R{{\mathbb R}}

\def\P{{\mathbb P}}
\def\E{{\mathbb E}}

\newtheoremstyle{ouech}
{\topsep}
{\topsep}
{\upshape} 
{} 
{\bfseries} 
{} 
{\newline}
{\thmname{#1}\thmnumber{ #2}\thmnote{#3}} 
\theoremstyle{ouech}

\theoremstyle{exemple}

\theoremstyle{ouech}

\theoremstyle{ouech}

\numberwithin{equation}{section}
\title{A continuous random operator associated with the Vertex Reinforced Jump Process on the circle and the real line}
\author{V. Rapenne\footnote{Universite Claude Bernard Lyon 1, Institut Camille Jordan,  rapenne@math.univ-lyon1.fr}, C. Sabot\footnote{Universite Claude Bernard Lyon 1, Institut Camille Jordan et Institut Universitaire de France, sabot@math.univ-lyon1.fr} }
\begin{document}

\maketitle
\begin{abstract}
In this paper, we focus on the scaling-limit of the random potential $\beta$ associated with the Vertex Reinforced Jump Process (VRJP) on one-dimensional graphs. Moreover, we give a few applications of this scaling-limit. By considering a relevant scaling of $\beta$, we contruct a continuous-space version of the random Schrödinger operator $H_{\beta}$ which is associated with the VRJP on circles and on $\R$. We also compute the integrated density of states of this operator on $\R$ which has a remarkably simple form. Moreover, by means of the same scaling, we obtain a new proof of the Matsumoto-Yor properties concerning the geometric Brownian motion which were proved in \cite{matyorpit}. This new proof is based on some fundamental properties of the random potential $\beta$. We use also the scaling-limit of $\beta$ in order to prove new identities in law involving exponential functionals of the Brownian motion which generalize the Dufresne identity. 
\end{abstract}
\section{Introduction}

This paper is concerned with the construction and investigation of scaling limits of the random Schrödinger operator associated with the Vertex Reinforced Jump Process (VRJP) on some one-dimensional sets (the one-dimensional torus and the real line). We start by briefly recalling the definition of the VRJP: let $(V,E)$ be a non-directed locally finite graph and $(W_{i,j})_{i,j\in V}$ be a family of non-negative conductances such that $W_{i,j}>0$ if and only if $\{i,j\}\in E$. The VRJP is the continuous self-reinforced random walk $(Y_s)_{s\geq 0}$ which is defined as follows: the VRJP starts from some vertex $i_0\in V$ and conditionally on the past before time $s$, it jumps from a vertex $i$ to one of its neighbour $j$ at rate $W_{i,j}L_j(s)$ where 
$$L_j(s)=1+\int_0^s\textbf{1}\{Y_u=j\}du.$$

In \cite{sabot_tarres}, it was shown, firstly, that the VRJP is closely related to the Edge Reinforced Random Walk, a famous reinforced process introduced by Diaconis and Coppersmith in the 80's, and secondly, that after some time change the VRJP can be represented as a mixture of Markov Jump Processes with a mixing measure given by a marginal of a supersymmetric sigma-field called the $\mathbb{H}^{2|2}$ model which was introduced by Zirnbauer in \cite{Zirnbauerorigins1,Zirnbauerorigins2} and investigated by Disertori, Spencer and Zirnbauer in \cite{DSZ, Disp}. A complementary, but closely related, representation of the VRJP in terms of a random Schrödinger operator was provided in \cite{SZT} on finite graphs. More precisely, if $(V,E)$ is a non-directed finite graph and $(W_{i,j})_{i,j\in V}$ some conductances on the edges, then for any potential $\beta=(\beta_i)_{i\in V}$ on the vertices, we define the discrete Schrödinger operator $(H_\beta(i,j))_{i,j\in V}$ by
$$H_{\beta}(i,j)=\textbf{1}\{i=j\}2\beta_i-\textbf{1}\{i\sim j\}W_{i,j}.$$
In \cite{SZT}, an explicit probability measure on the set of potentials $\beta$ is defined (see Proposition \ref{prop:resumebeta}  below). This measures lives on the set where $H_\beta$ is positive definite, it has several remarkable properties which we recall in section \ref{subsecbetarecall}, and it gives a representation of the VRJP in the following sense: after some time-change, the VRJP starting at a vertex $i_0$ is a mixture of Markov jump processes with jump rates from $i$ to $j$ given by:
$$\frac{1}{2}W_{i,j}\frac{G_{\beta}(i_0,j)}{G_{\beta}(i_0,i)},$$
where $\beta$ is the random potential and $G_\beta=(H_\beta)^{-1}$.  That representation, and its generalization to infinite graphs (see \cite{sabotzeng}), has played an important role in order to understand the asymptotic behavior of the VRJP (see in particular \cite{sabotzeng,Poudevigne}).

In this paper, we are mainly concerned by that representation and its scaling limits. The question of the scaling limits of the VRJP and its representation is rather natural, but remains still quite mysterious. In \cite{LST}, the scaling limit of the VRJP itself was analysed. With this goal, the scaling limit of its mixing field, i.e. of the limit of the function $(G_{\beta}(0,\lfloor nt\rfloor)/G_{\beta}(0,0))_{t \in \R}$ was described on the real line in terms of the geometric Brownian motion. Here, we investigate the scaling limit of $H_\beta$ and $G_\beta$ as random operators, both on the one-dimensional torus and the real line. More precisely, the main results of the paper are summarized below:
\begin{itemize}
\item By scaling limits, we construct an explicit continuous-space version of the operator $H_{\beta}$ and its inverse $G_{\beta}$ on circles, and on the real line. We describe the domains of these operators on the circle, while we still face some technical difficulties to describe the domains in the case of the real line. However, we hope to solve this problem in the near future. Note that the one-dimensional torus is not a tree, which induces specific difficulties (the representation of the VRJP on trees is considerably simpler), and that in the case of the real line, we get the full description of the operators, while in \cite{LST}, only the scaling limit of $(G_{\beta}(0,\lfloor nt\rfloor)/G_{\beta}(0,0))_{t\geq 0}$ was considered in order to analyse the continuous-space VRJP.
\item A natural problem regarding $H_{\beta}$ consists in investigating its spectral properties as a self-adjoint operator.
In particular, spectral properties of self-adjoint operators are very important if we want to understand the dynamical properties of a quantum particle whose wave-function follows the Schrödinger equation.
In this paper, we compute the exact density of states of the limiting operator $H_{\beta}$ on $\R$, which has a surprisingly simple form. To do so, we follow a method which is very similar with what is done in \cite{Funak} for the continuous Anderson model.
\item In the discrete-space case, many identities in law involving the $\beta$ potential are known. Taking the scaling-limit in these formulas, we prove new identities in law involving exponential functionals of the Brownian motion. These identities generalize the famous Dufresne identity (originally proved in \cite{Duf}) which states that
$$\int_0^{\infty}e^{2\alpha_s-s}ds\overset{law}=\frac{1}{2\gamma}$$
where $\alpha$ is a Brownian motion and $\gamma$ is a Gamma distribution with parameters $(1/2,1)$.
\item Considering the scaling limit of the $\beta$-potential on $\N^*$, we give a new proof of the Matsumoto-Yor properties (see \cite{matyorpit1} and \cite{matyorpit}), which concern exponential functionals of the Brownian motion. More precisely, we give a discrete time version of the Matsumoto-Yor properties which involves natural functions of the $\beta$-potential. Specified to our context, the Matsumoto-Yor properties state that the process $(Z_t)_{t\geq 0}=\left(\frac{T_t}{e_t}\right)_{t\geq 0}$, where $e$ is the geometric Brownian motion and for every $t\geq 0$, $T_t=\int_0^te_s^2ds$, is a Markov process in its own filtration. Moreover the filtration of $Z$ is stricly smaller than the filtration of $e$, and there is an explicit intertwining between $e$ and $Z$ involving Inverse Gaussian distributions. While somehow mysterious at first sight, Matsumoto-Yor properties seem to be rather fundamental and have been generalized in different directions, in particular in relation with properties of Lie groups and solvable polymer models (see \cite{ocoyor,PBO1,PBO2}).
Note that some Mastumoto-Yor properties on graphs also appear in a different way in \cite{GRSZ}.
\end{itemize}

Finally, let us mention some related works concerning a different operator, the continuous Anderson model on $\R$ (where the random potential is given by a white noise). In \cite{Dumazlabbe}, Dumaz and Labbé gave a very accurate description of the spectrum and of the eigenstates for this operator. It is probably possible to apply their ideas in order to give the precise behaviour of the spectrum of the continuous-space version of $H_{\beta}$. However we do not do it in this article. Furthermore, an interesting question would be to find continuous versions of $H_{\beta}$ on topological spaces which are not one-dimensional. In the case of the Anderson model, in \cite{Labbe}, Labbé managed to do it on $(-L,L)^d$ with $d\in\{1,2,3\}$ and $L>0$. It would be interesting to know if it is possible to do the same thing for $H_{\beta}$ but this question remains open for now.
\section{Context and statement of the results}
\subsection{General notation}
For every $n\in\N^*$, $\mathcal{C}_n$ denotes the circle graph with $2n+1$ points. More precisely, the vertex set of $\mathcal{C}_n$ is $\{-n,\cdots,0,\cdots,n\}$ and for every $i\in\{-n,\cdots,0,\cdots,n-1\}$ there is an edge between $i$ and $i+1$ and there is an edge between $n$ and $-n$.
In any graph $(V,E)$ the fact that two vertices $i$ and $j$ are related by an edge is denoted by $i\sim j$.

If $V_1$ and $V_2$ are two finite sets and $H$ is a matrix indexed by $V_1\times V_2$, we denote the coefficient of $H$ at $(i,j)\in V_1\times V_2$ by $H(i,j)$. If $H$ is a squared symmetric matrix on a set $V$, then we write $H>0$ (resp. $H\geq 0$) when $H$ is positive definite (resp. when $H$ is non-negative). If $H$ is a squared matrix, its determinant is denoted by $|H|$. If $V$ is a finite set and $v_1$ and $v_2$ are two vectors of $\R^V$, the standard scalar product between $v_1$ and $v_2$ is denoted by $\langle v_1,v_2\rangle$. If $H$ is a matrix on a finite set $V\times V$ and if $V_1$ and $V_2$ are two subsets of $V$, the restriction of $H$ to $V_1\times V_2$ is denoted by $H_{V_1,V_2}$. Moreover, if $v$ is a vector of $\R^V$ and $V_1$ is a subset of $V$, then the restriction of $v$ to $V_1$ is denoted by $v_{V_1}$.

We denote by $\mathcal{C}^{(\lambda)}$ the continuous circle $\R/2\lambda\R$. For sake of convenience, we will often write $\mathcal{C}^{(\lambda)}$ as $[-\lambda,\lambda]$ (for example when we write integrals) where it is implicit that $-\lambda$ and $\lambda$ are topologically identified. If $t,t'\in\mathcal{C}^{(\lambda)}$, we write sometimes $t\leq t'$ or $t<t'$ refering to the natural order on $[-\lambda, \lambda]$.

Recall that an Inverse Gaussian distribution with parameters $(\mu,\lambda)$ has density
$$\textbf{1}\{x>0\}\frac{\sqrt{\lambda}}{ \sqrt{2\pi x^3}}e^{-\lambda\frac{(x-\mu)^2}{2x\mu^2}}dx.$$
The Inverse Gaussian distribution with parameters $(\mu,\lambda)$ will be denoted by $IG(\mu,\lambda)$. Recall that $\E\left[IG(\mu,\lambda)\right]=\mu$ and $\E\left[IG(\mu,\lambda)^2\right]=\mu^2+\frac{\mu^3}{\lambda}.$
\subsection{The random potential $\beta$}\label{subsecbetarecall}
Let $(V,E)$ be a finite graph with $n$ vertices. Let $W$ be a matrix of symmetric non-negative weights $(W_{i,j})_{i,j\in V\times V}$ such that $W_{i,j}>0$ if and only if $\{i,j\}\in E$. For every $\beta\in\R_+^V$, let us define the matrix $H_{\beta}$ on $V\times V$ such that for every $i,j\in V$,
\begin{align}
H_{\beta}(i,j)=\textbf{1}\{i=j\}2\beta_i-W_{i,j}.\label{defiH5}
\end{align}
For every $\beta\in\R_+^V$ such that $H_{\beta}$ is positive definite, one can define its inverse $G_{\beta}$ which has only positive entries.
In \cite{SZT}, in order to study the VRJP, Sabot, Tarrès and Zeng introduced a probability measure $\nu_V^W$ on $\R_+^V$ which will be crucial in this paper. It is defined by means of the following proposition:
\begin{prop}[Proposition 1 and Theorem 3 in \cite{SZT}]\label{prop:resumebeta}
~\\
\begin{enumerate}[(i)]
\item The function
\begin{equation}
\begin{array}{lll}
\R^V&\longrightarrow&\R_+\\
\beta&\mapsto& \left(\frac{2}{\pi}\right)^{n/2}\textbf{1}\{H_{\beta}>0\}e^{-\frac{1}{2}\langle 1,H_{\beta}1\rangle}\frac{1}{\sqrt{|H_{\beta}|}}
\end{array}\label{formuladens}
\end{equation}
is a probability density. In the formula \eqref{formuladens}, $1$ stands for the vector of $\R^V$ whose entries are all equal to 1. We denote by $\nu_V^W$ the probability measure on $\R_+^V$ which is associated with the density of equation \eqref{formuladens}.
\item One can compute explicitely the Laplace transform of $\nu_V^W$. For every $t\in\R_+^V$,

\begin{equation}
\int e^{-\langle t,\beta\rangle}\nu_V^W(d\beta)=\exp\left(-\frac{1}{2}\sum\limits_{\{i,j\}\in E}W_{i,j}\left(\sqrt{(t_i+1)(t_j+1)}-1\right) \right)\prod\limits_{i\in V}\frac{1}{\sqrt{1+t_i}}.\label{transfolaplace}
\end{equation}
\end{enumerate}
In this article, we will often make a small abuse of notation by using the notation $\beta$ to designate a random vector and a variable inside the density of this random vector. Now, let  $\beta$ be a random vector with distribution $\nu_V^W$.
\begin{enumerate}[(i)]
\setcounter{enumi}{2}
\item  For every $i\in V$, $1/(2\beta_i-W_{i,i})$ is an Inverse Gaussian ditribution with parameters $\left(\frac{1}{\sum\limits_{i\neq j}W_{i,j}},1 \right)$.
\item The random potential $\beta$ is 1-dependent, that is, if $V_1$ and $V_2$ are two subsets of $V$ which are not related by an edge, then $(\beta_i)_{i\in V_1}$ and $(\beta_i)_{i\in V_2}$ are independent.
\item For every $i\in V$,
$$G_{\beta}(i,i)\overset{law}=\frac{1}{2\gamma}$$
where $\gamma$ is a Gamma distribution with parameters $(1/2,1)$.
\end{enumerate}

\end{prop}
In \cite{SZT}, Sabot, Tarrès and Zeng used the $\beta$-potential in order to study the VRJP on any finite graph $V$. They showed that the VRJP (actually a time-changed version of the VRJP) on $V$ starting from $i_0\in V$ with weight matrix $W$ is a mixture of random processes which jumps from $i$ to $j$ at rate
$$\frac{W_{i,j}}{2}\frac{G_{\beta}(i_0,j)}{G_{\beta}(i_0,i)}.$$
In \cite{sabotzeng}, Sabot and Zeng proved that this is possible to extend the measure $\nu_V^W$ and the above representation of the VRJP on an infinite graph which enables us to look in particular at the interesting case of $\mathbb{Z}^d$.

\begin{prop}[Section 4.2 in \cite{sabotzeng}]
Let $(V,E,W)$ be an infinite locally finite graph with conductances where $W$ is a symmetric conductance operator $(W_{i,j})_{i,j\in V}$ such that $W_{i,j}>0$ if and only if $\{i,j\}\in E$. Then, there exists an infinite-volume measure $\nu_V^W$ on $\R_+^V$ such that for every finite subset $V_1$ which is included in $V$, for any $t\in\R_+^{V_1}$,
\begin{align}
\int e^{-\langle t,\beta\rangle}\nu_V^W(d\beta)=e^{-\frac{1}{2}\sum\limits_{\{i,j\}\in E}W_{i,j}\left(\sqrt{(t_i+1)(t_j+1)}-1\right)-\sum\limits_{\{i,j\}\in E, j\notin V_1} W_{i,j}\left(\sqrt{t_i+1}-1 \right)}\prod\limits_{i\in V_1}\frac{1}{\sqrt{1+t_i}}.\label{transfolaplaceinfini}
\end{align}
\end{prop}
Now, let $(V_n)_{n\in\N}$ be an increasing sequence of boxes such that $\bigcup_{n\in\N}V_n=V$. For every $n\in\N$, for every $i,j\in V$, let us define
$$\hat{G}_{\beta}^{(n)}(i,j)=(H_{\beta})_{V_n,V_n}^{-1}(i,j)$$
if $i,j\in V_n$ and $\hat{G}_{\beta}^{(n)}(i,j)=0$ otherwise. Moreover, for every $n\in\N$, let us define $\psi_{\beta}^{(n)}$ as the unique solution of the equation
$$\begin{array}{rl}
(H_{\beta}\psi^{(n)}_{\beta})(i)=0,& \text{ for every }i\in V_n,\\
\psi^{(n)}_{\beta}(i)=1,& \text{ for every }i\in V_n^c.
\end{array}
$$
Note that for every $n\in\N$, and for every $i\in V_n$, there is another useful expression of $\psi_{\beta}^{(n)}(i)$ which is
$$\psi_{\beta}^{(n)}(i)=\sum\limits_{j\in V_n} \hat{G}_{\beta}^{(n)}(i,j)\eta_j^{(n)}$$
where for every $j\in V_n$, $\eta_j^{(n)}=\sum\limits_{k\sim j, j\notin V_n}W_{j,k}.$
These objects were introduced by Sabot and Zeng in \cite{sabotzeng} and were crucial in order to study the VRJP on infinite graphs because of the following proposition:
\begin{prop}[Proposition 9 and Theorem 1 in \cite{sabotzeng}]\label{prop:rappelpsi}
Let $\beta\sim \nu_V^W$. It holds that,
\begin{enumerate}[(i)]
\item For every $i\in V$, $(\psi_{\beta}^{(n)}(i))_{n\in\N}$ is a non-negative martingale. In particular, it has almost surely a limit $\psi_{\beta}(i)$. 
\item The bracket of $(\psi_{\beta}^{(n)})_{n\in\N}$ is $(\hat{G}_{\beta}^{(n)})_{n\in\N}$ in the sense that for every $i,j\in V$, $(\psi_{\beta}^{(n)}(i)\psi_{\beta}^{(n)}(j)-\hat{G}_{\beta}^{(n)}(i,j))_{n\in\N}$ is a martingale.
\item The VRJP is almost surely recurrent if and only if almost surely $\psi_{\beta}(i)=0$ for every $i\in V$ .
\item The VRJP is almost surely transient if and only if almost surely $\psi_{\beta}(i)>0$ for every $i\in V$.
\item Let $\gamma$ be a Gamma random variable with parameters $(1/2,1)$ which is independent of $\beta$. For every $i,j\in V$, let us define
$$G_{\beta,\gamma}(i,j)=\hat{G}_{\beta}(i,j)+\frac{1}{2\gamma}\psi_{\beta}(i)\psi_{\beta}(j).$$
Then, the VRJP on $V$ starting from $i_0\in V$ with weight matrix $W$ is a mixture of random processes which jumps from $i$ to $j$ at rate
$$\frac{W_{i,j}}{2}\frac{G_{\beta,\gamma}(i_0,j)}{G_{\beta,\gamma}(i_0,i)}.$$
\end{enumerate}
\end{prop}
In this article, we want to understand what is going on when we consider a scaling limit for the $\beta$-field on a one-dimensional graph. We will show that it is possible to define a continuous version of  $H_{\beta}$ on $\R$ or continuous circles. But before that, we prove a discrete-time version of the Matsumoto-Yor properties by means of the $\beta$-field on $\N^*$, $(\psi_{\beta}^{(n)})_{n\in\N}$ and  $(\hat{G}_{\beta}^{(n)})_{n\in\N}$. We will take a scaling limit in this discrete-time version in order to recover the continuous-time version of the Matsumoto-Yor properties. This strong connection between the $\beta$-field and the Matsumoto-Yor properties explains the origin of the surprising identities in law in the forthcoming subsection \ref{subsubsec:idenloi}.
\subsection{A new approach of the Matsumoto-Yor properties in relation with the mixing measure of the VRJP}\label{subsubsec:matyor}

First, let us recall the Matsumoto-Yor properties on $\R_+$.
Let $\alpha$ be a standard Brownian motion on $\R_+$. Then we can define the associated geometric Brownian motion $e$ as $(e_t)_{t\geq 0}=(\exp(\alpha_t-t/2))_{t\geq 0}$. Moreover, let us define the related exponential functionals $T$ and $Z$ such that for every $t>0$,
\begin{align} T_t=\int_0^t e_s^2ds\text{ and }Z_t=\frac{T_t}{e_t}.\label{defimatyor}
\end{align}
For every $t\geq 0$, we define two sigma-fields $\mathcal{A}_t=\sigma\left(\alpha_s,s\leq t \right)$ and $\mathcal{Z}_t=\sigma\left(Z_s,s\leq t\right)$. 
Then, Matsumoto and Yor proved the following results:
\begin{thm}[Theorem 1.6 and Proposition 1.7 in \cite{matyorpit}.]\label{thm:matsumotoyor}
~\\\vspace{-0.5 cm}
\begin{enumerate}[(i)]
\item For every $t>0$, $\mathcal{Z}_t\subsetneqq\mathcal{A}_t.$
\item  $Z$ is a diffusion process whose infinitesimal generator is
$$\frac{1}{2}z^2\frac{d^2}{dz^2}+(1+z)\frac{d}{dz}.$$
\item For every $t>0$, the conditional distribution of $e_t$ knowing $\mathcal{Z}_t$ is an Inverse Gaussian distribution with parameters $(1,1/Z_t)$. More precisely, for every $t>0$, the conditional distribution of $e_t$ knowing $\mathcal{Z}_t$ is an Inverse Gaussian distribution with parameter $(1,1/Z_t)$, i.e. it has density
$$\textbf{1}\{x>0\}\frac{1}{\sqrt{2\pi Z_tx^3}}e^{-\frac{1}{Z_t}\frac{(x-1)^2}{2x}}dx.$$
\end{enumerate}
\end{thm}
Theorem \ref{thm:matsumotoyor} will be called "Matsumoto-Yor properties" in the sequel of this paper. Now, let us give a discrete-time counterpart of those Matsumoto-Yor properties. Now, let $m>0$ and let $K_m$ be a weight operator on the line graph $\N^*$ such that for every $i\in \N^*$, $K_m(i,i+1)=K_m(i+1,i)=m$. All other entries of $K_m$ are zero. Then we can define the random operator ${\bf{H}}_{\beta}^{(m)}$ on the discrete half-line $\N^*$ associated with the random field $\beta\sim \nu_{\N^*}^{K_m}$. We write ${\bf{H}}_{\beta}^{(m)}$ in bold letters in order to avoid the confusion with $H_{\beta}^{(\lambda,n)}$ on the discrete circle which shall be introduced later. Now, for every $n\in \N^*$, let us define $\hat{\bf{G}}_{\beta}^{(n,m)}=\left(({\bf{H}}_{\beta}^{(m)})_{\llbracket 1,n\rrbracket,\llbracket 1,n\rrbracket}\right)^{-1}$. For every $n\in\N^*$, we define also
$$ \psi^{(n,m)}_{\beta}=\hat{\bf{G}}_{\beta}^{(n,m)}(1,n)m\text{ and }Z_{\beta}^{(n,m)}=\frac{\hat{\bf{G}}_{\beta}^{(n,m)}(1,1)}{\psi^{(n,m)}_{\beta}}.$$
For every $n\in\N^*$, we define 
$$ \mathcal{A}_{n,m}=\sigma(\psi_{\beta}^{(k,m)},1\leq k\leq n)\text{ and }\mathcal{Z}_{n,m}=\sigma(Z_{\beta}^{(k,m)},1\leq k\leq n).$$
By Proposition \ref{prop:rappelpsi}, $(\psi^{(n,m)}_{\beta})_{n\in\N^*}$ is a martingale whose bracket is $(\hat{{\bf{G}}}_{\beta}^{(n,m)}(1,1))_{n\in\N^*}$. Remark that it is analoguous to the case of the geometric Brownian motion in equation \eqref{defimatyor} because $(e_t)_{t\geq 0}$ is a martingale whose bracket is $(T_t)_{t\geq 0}$.
The interest of these discrete objects is that they give a discrete version of the results of Matsumoto and Yor:
\begin{prop}[Discrete version of the Matsumoto-Yor properties]\label{prop:discretematyor}
Let $m\in\N^*$ be fixed.
\begin{enumerate}[(i)]
\item For every $n\in\N^*$, $\mathcal{Z}_{n,m}\subsetneqq \mathcal{A}_{n,m}$.
\item $\left(Z_{\beta}^{(n,m)}\right)_{n\in\N^*}$ is a Markov process. More precisely, for every $n\in\N^*$, the law of $Z_{\beta}^{(n+1,m)}$ conditionally on $\mathcal{Z}_{n,m}$ is 
$$\frac{Z_{\beta}^{(n,m)}}{m}\times\frac{1}{IG\left(\frac{1}{m+\frac{1}{Z_{\beta}^{(n,m)}}},1 \right)}.$$
\item For every $n\in\mathbb{N}^*$, the conditional distribution of $\psi^{(n,m)}_{\beta}$ knowing $\mathcal{Z}_{n,m}$ is an Inverse Gaussian distribution with parameters $(1,1/Z_{\beta}^{(n,m)})$. More precisely, for every $n\in\mathbb{N}^*$ the conditional density of $\psi^{(n,m)}_{\beta}$ knowing $\mathcal{Z}_{n,m}$ is
$$\textbf{1}\{x>0\}\frac{1}{\sqrt{2\pi Z_{\beta}^{(n,m)}x^3}}e^{-\frac{1}{Z_{\beta}^{(n,m)}}\frac{(x-1)^2}{2x}}dx.$$
\end{enumerate}
\end{prop}

Remark that Theorem \ref{thm:matsumotoyor} and Proposition \ref{prop:discretematyor} are very similar. Actually, we can recover Theorem 
\ref{thm:matsumotoyor} by taking a scaling limit in Proposition \ref{prop:discretematyor} thanks to the following proposition:
\begin{prop} \label{prop:scalinglimit}
Let $(\tilde{\psi}^{(m)}(t))_{t\geq 0}$ and $(\tilde{Z}^{(m)}(t))_{t\geq 0}$ be the continuous linear interpolations of \\$(\psi_{\beta}^{(\lfloor mt\rfloor,m)})_{t\geq 0}$ and $(Z_{\beta}^{(\lfloor mt\rfloor,m)})_{t\geq 0}$.
Then, the following convergence does hold for the topology of uniform convergence on compact sets:
$$\left(\tilde{\psi}^{(m)}(t),\tilde{Z}^{(m)}(t)\right)_{t\geq 0}\xrightarrow[m\rightarrow+\infty]{law}\left(e_t,Z_t\right)_{t\geq 0}.$$
\end{prop}
In section \ref{sec:proofmatyor}, we will prove Propositions \ref{prop:discretematyor} and \ref{prop:scalinglimit}. Moreover, we will use these results in order to give a new proof of Theorem \ref{thm:matsumotoyor}.
\begin{rem}
Thanks to the discrete processes $(\psi^{(n,m)}_{\beta})_{n\in\mathbb{N}^*}$ and $(Z_{\beta}^{(n,m)})_{n\in\mathbb{N}^*}$, we were able to define a discrete one-dimensional analogue  of Matsumoto-Yor exponential functionals of the Brownian motion and we recovered the results of Matsumoto and Yor by taking a scaling limit. However, the processes $(\psi^{(n,m)}_{\beta})_{n\in\mathbb{N}^*}$ and $(Z_{\beta}^{(n,m)})_{n\in\mathbb{N}^*}$ can be constructed on any graph and one could prove (i) and (iii) of  Proposition \ref{prop:discretematyor} on any graph. ( However, we do not know whether (ii) is true on a general graph.) Consequently, we can define $(\psi^{(n,m)}_{\beta})_{n\in\mathbb{N}^*}$ and $(Z_{\beta}^{(n,m)})_{n\in\mathbb{N}^*}$ for example on $\mathbb{Z}^d$ with any $d\geq2$. If one could take a scaling limit on $\mathbb{Z}^d$, it would give new contiuous-time processes which should exhibit properties like (i) and (iii) in Theorem \ref{thm:matsumotoyor}.
\end{rem}
\subsection{Scaling limit and continuous version of $H_{\beta}$ on the circle}
\subsubsection{Definition of the discrete operator on the circle}
The main goal of this paper is to define a version of $H_{\beta}$ and $G_{\beta}$ on continuous unidimensional spaces. In order to do this, we define a model on a discretized version of the circle and we will make the  size of the mesh go to 0. Let $n\in\N^*$. Let $\lambda>0$. Let $W_n^{(\lambda)}$ be a  matrix on the discretized circle $\mathcal{C}_{\lceil\lambda  n\rceil}$ such that $(W_n^{(\lambda)})_{i,j}$ is 0 if $i$ and $j$ are not connected and is $n$ otherwise. Let us denote $H_{\beta}^{(\lambda,n)}$ the matrix associated with the random potential $\beta$ with distribution $\nu_{\mathcal{C}_{\lceil\lambda  n\rceil}}^{W_n^{(\lambda)}}$. Moreover, we denote by $G_{\beta}^{(\lambda,n)}$ the inverse of $H_{\beta}^{(\lambda,n)}$. We define also a rescaled continuous bilinear interpolation $(\tilde{G}_{\beta}^{(\lambda,n)})_{t,t'\in\mathcal{C}^{(\lambda)}}$ of $\left(G_{\beta}^{(\lambda,n)}\left(\lceil nt\rceil,\lceil nt'\rceil\right)\right)_{t,t'\in\mathcal{C}^{(\lambda)}}$. More precisely, if $i/n\leq t<(i+1)/n$ and $j/n\leq t'<(j+1)/n$, 
\begin{align}
\tilde{G}_{\beta}^{(\lambda,n)}(t,t')=G_{\beta}^{(\lambda,n)}(i,j)&+n(t-i/n)\left(G_{\beta}^{(\lambda,n)}(i+1,j)-G_{\beta}^{(\lambda,n)}(i,j)\right)\nonumber\\
&\hspace{-4cm}+n(t'-j/n)\left(G_{\beta}^{(\lambda,n)}(i,j+1)-G_{\beta}^{(\lambda,n)}(i,j)\right)\nonumber\\
&\hspace{-4cm}+n^2(t'-j/n)(t-i/n)\left(G_{\beta}^{(\lambda,n)}(i,j)+G_{\beta}^{(\lambda,n)}(i+1,j+1)-G_{\beta}^{(\lambda,n)}(i,j+1)-G_{\beta}^{(\lambda,n)}(i+1,j) \right).\label{interpo}
\end{align}
\subsubsection{Definition of the continuous limit}\label{contimod}
Let $B$ be a Brownian motion on $\R$ such that $B(0)=0$ almost surely. We define the geometric Brownian motion $M$ by $$(M_t)_{t\in\R}=(e^{B_t-t/2})_{t\in\R}.$$
Let $\lambda>0$. Then, we introduce the symetric random kernel $\mathcal{G}^{(\lambda)}$ on $\mathcal{C}^{(\lambda)}$ by
$$\mathcal{G}^{(\lambda)}(t,t')=\frac{M_{t'}M_t}{(M_{\lambda}-M_{-\lambda})^2}\left(M_{\lambda}^2\int_{t'}^{\lambda}\frac{ds}{M_s^2} +M_{\lambda}M_{-\lambda}\int_t^{t'}\frac{ds}{M_s^2}+M_{-\lambda}^2\int_{-\lambda}^t\frac{ds}{M_s^2}\right)$$ for every $t\leq t'\in \mathcal{C}^{(\lambda)}$.

\subsubsection{Results of convergence}
First, the rescaled continuous interpolation $\tilde{G}_{\beta}^{(\lambda,n)}$ of the matrix $G_{\beta}^{(\lambda,n)}$ has a limit in law when $n$ goes to infinity.
\begin{thm}\label{thm:convergence}
Let $\lambda >0$. Then 
$$ \tilde{G}_{\beta}^{(\lambda,n)}\xrightarrow[n\rightarrow+\infty]{law}\mathcal{G}^{(\lambda)}$$
for the topology of uniform convergence on $\left(\mathcal{C}^{(\lambda)}\right)^2$.
\end{thm}
Moreover, $\mathcal{G}^{(\lambda)}$ can be seen as a bilinear form with the following expression.
\begin{prop} \label{formeg}
Let $\lambda >0$. Let $f\in L^2([-\lambda,\lambda])$.
Then
\begin{align*}
\int_{-\lambda}^{\lambda}\int_{-\lambda}^{\lambda} f(t)\mathcal{G}^{(\lambda)}(t,t')\bar{f}(t')dtdt'\\
&\hspace{-4 cm}=\frac{1}{(M_{\lambda}-M_{-\lambda})^2} \int_{-\lambda}^{\lambda}\frac{1}{M_u^2}\left|M_{-\lambda}\int_u^{\lambda}f(t)M_tdt+M_{\lambda}\int_{-\lambda}^uf(t)M_tdt \right|^2du.
\end{align*}
In particular, $\mathcal{G}^{(\lambda)}$ is positive definite almost surely. 
\end{prop}
\subsubsection{Dufresne's type identities in law} \label{subsubsec:idenloi}
In \cite{Duf} (see also \cite{yorexpo} for an alternative proof), Dufresne proved the following famous identity in law:
$$\int_0^{+\infty}e^{2\alpha_s-s}ds\overset{law}=\frac{1}{2\gamma}$$
where $\gamma$ has Gamma distribution with parameters $(1/2,1)$. Recall that, by Proposition \eqref{prop:resumebeta}, in the discrete-time setting, $G_{\beta}^{(\lambda,n)}(i,i)$ is also distributed as the inverse of a Gamma distribution for every $i\in\mathcal{C}_{\lceil \lambda n\rceil}$. Actually, this is not a coincidence and one can recover Dufresne's identity by making $n$ go to infinity in $G_{\beta}^{(\lambda,n)}(i,i)$.
Moreover, by means of the limiting random kernel $\mathcal{G}^{(\lambda)}$, we can prove some new identities in law which generalize Dufresne's identity. 

\begin{prop}\label{idenloipoint}
Let $t\in[-\lambda,\lambda]$. Then, the following identity in law does hold:
$$\mathcal{G}^{(\lambda)}(t,t)=\frac{M_t^2}{(M_{\lambda}-M_{-\lambda})^2}\left(M_{\lambda}^2\int_{t}^{\lambda}\frac{ds}{M_s^2} +M_{-\lambda}^2\int_{-\lambda}^t\frac{ds}{M_s^2}\right)\overset{law}=\frac{1}{2\gamma}$$
 where $\gamma$ is a Gamma distribution with parameters $(1/2,1)$.
\end{prop}
A particular case of Proposition \ref{idenloipoint} implies the following corollary:
\begin{cor} \label{identityaladufresne}
If $\alpha$ is a standard Brownian motion on $\R_+$, then for every $\lambda>0$,
$$\frac{\displaystyle\int_0^{\lambda}e^{2\alpha_s-s}ds}{\left(e^{\alpha_{\lambda}-\lambda/2}-1\right)^2}\overset{law}=\frac{1}{2\gamma}$$
where $\gamma$ is a Gamma distribution with parameters $(1/2,1)$.
\end{cor}
\begin{rem}
The identity of Corollary \ref{identityaladufresne} is not really new. Actually, it can be deduced also from the Matsumoto-Yor properties. This second approach is also explained in the proof of Corollary \ref{identityaladufresne}. Remark also that making $\lambda$ go to infinity in Corollary \ref{identityaladufresne} gives the Dufresne identity. 
The fact that the second proof of Corollary \ref{identityaladufresne} involves the Matsumoto-Yor properties is not very surprising because we explained in section \ref{subsubsec:matyor} that the Matsumoto-Yor properties can be deduced from the properties of $\beta$-potential.
\end{rem}
More generally, we can compute the distribution of $\langle \eta, G_{\beta}^{(\lambda,n)}\eta\rangle$ for every $\eta\in (\R_+^*)^{\mathcal{C}_{\lceil \lambda n\rceil}}$. Combining this with the limit obtained in Theorem \ref{thm:convergence} gives new identities for the geometric Brownian motion.
\begin{prop}\label{idenloifonc}
Let $f$ be a deterministic continuous non-negative function on $\mathcal{C}^{(\lambda)}$.
Then, the following identity in law does hold:
 $$\frac{1}{(M_{\lambda}-M_{-\lambda})^2} \int_{-\lambda}^{\lambda}\frac{1}{M_u^2}\left(M_{-\lambda}\int_u^{\lambda}f(t)M_tdt+M_{\lambda}\int_{-\lambda}^uf(t)M_tdt \right)^2du\overset{law}= \frac{\displaystyle\left(\int_{-\lambda}^{\lambda}f(t)dt\right)^2}{2\gamma}$$
 where $\gamma$ is a Gamma distribution with parameters $(1/2,1)$.
\end{prop}
\begin{rem}
Actually, in Proposition \ref{idenloifonc}, the continuity of $f$ is a very strong assumption. Indeed, we can allow discontinuity points for $f$ but we only focus on the continuous case for sake of convenience.
\end{rem}

\subsubsection{The continuous random operator $\mathcal{H}^{(\lambda)}$}
Let us define the domain
$$\mathcal{D}\left(\mathcal{H}^{(\lambda)} \right)=\begin{Bmatrix}
g\in L^2([-\lambda,\lambda]),&\left(\frac{g}{M}\right)'\in L^2([-\lambda,\lambda]),& \left(M^2\left(\frac{g}{M}\right)'\right)'\in L^2([-\lambda,\lambda]),\\
g(-\lambda)=g(\lambda),& M_{-\lambda}\left(\frac{g}{M}\right)'(-\lambda)=M_{\lambda}\left(\frac{g}{M}\right)'(\lambda)\end{Bmatrix}.$$
In the definition above, $\frac{g}{M}$ means the random function $x\mapsto \frac{g(x)}{M_x}$. The derivative $'$ is defined in the sense of distributions. Moreover, if $g\in\mathcal{D}\left( \mathcal{H}^{(\lambda)}\right)$, $g$ and $M\left(\frac{g}{M}
\right)'$ are well-defined at $-\lambda$ and $\lambda$ because they are actually continuous functions. This stems from Sobolev injections in dimension 1. (Indeed the Sobolev space $H^1([-\lambda,\lambda])$ can be injected in the set of continuous functions.)
For every $f\in L^2([-\lambda,\lambda])$, we define the function $\mathcal{G}^{(\lambda)}f$ such that for every $x\in\mathcal{C}^{(\lambda)}$,
$$\mathcal{G}^{(\lambda)}f(x)=\int_{-\lambda}^{\lambda}\mathcal{G}^{(\lambda)}(x,t)f(t)dt.$$
Consequently, $\mathcal{G}^{(\lambda)}$ can be viewed as an operator on $L^2([-\lambda,\lambda])$. Now, we can state our next result:
\begin{thm}\label{thm:definitionofH}
Let $\lambda>0$. The image of $\mathcal{G}^{(\lambda)}$ is exactly $\mathcal{D}\left(\mathcal{H}^{(\lambda)} \right)$. Therefore, $\mathcal{G}^{(\lambda)}$ has a bijective inverse $\mathcal{H}^{(\lambda)}$ from $\mathcal{D}\left(\mathcal{H}^{(\lambda)} \right)$ onto $L^2([-\lambda,\lambda])$. For every $g\in \mathcal{D}\left(\mathcal{H}^{(\lambda)} \right)$,
$$\mathcal{H}^{(\lambda)}g=-\frac{1}{M}\left(M^2\left(\frac{g}{M} \right)'\right)'.$$
Furthermore, $\mathcal{H}^{(\lambda)}$ is a positive self-adjoint operator (for the classical inner-product on $L^2([-\lambda,\lambda])$) with domain $\mathcal{D}\left(\mathcal{H}^{(\lambda)} \right)$.
\end{thm}
As $\mathcal{C}^{(\lambda)}$ is compact, the operator $\mathcal{H}^{(\lambda)}$ is localized. More precisely, its spectrum $\sigma\left(\mathcal{H}^{(\lambda)} \right)$, that is, the set of real numbers $E$ such that $\mathcal{H}^{(\lambda)}-E$ is not invertible, consists only in a sequence of isolated eigenvalues.
\begin{prop}\label{prop:spectrum}
Almost surely, there exists an increasing positive random sequence $(E_k(\lambda))_{k\geq 0}\in(\R_+^*)^{\N}$ which diverges toward infinity such that
$$ \sigma\left(\mathcal{H}^{(\lambda)} \right)=(E_k(\lambda))_{k\geq 0}.$$
Moreover, for every $k\in\N$, $E_k(\lambda)$ is an eigenvalue of $\mathcal{H}^{(\lambda)}$ with finite multiplicity. The eigenvalues of $\mathcal{H}^{(\lambda)}$ are counted with multiplicity.
\end{prop}
\subsection{Continuous version of $H_{\beta}$ on the real line}
One can wonder what is the limit of $\mathcal{G}^{(\lambda)}$ and $\mathcal{H}^{(\lambda)}$ when $\lambda$ goes to infinity. It would define some operators $\mathcal{G}^{(\infty)}$ and $\mathcal{H}^{(\infty)}$ on $\R$ which are associated with the VRJP on $\R$. The following proposition gives a partial answer to this question.
\begin{prop}\label{volumeinfini}
For the topology of uniform convergence on compact sets in $\R^2$, it holds that
$$\mathcal{G}^{(\lambda)}\xrightarrow[\lambda\rightarrow+\infty]{law}\mathcal{G}^{(\infty)}$$
where $\mathcal{G}^{(\infty)}$ is a symmetric random kernel on $\R^2$ such that for every $t,t'\in\R$ such that $t\leq t'$,
$$\mathcal{G}^{(\infty)}(t,t')=M_{t'}M_t\int_{-\infty}^t\frac{ds}{M_s^2}.$$
\end{prop}
\begin{rem}
It is important to notice that $\mathcal{G}^{(\infty)}$ is not well-defined on the whole Hilbert space $L^2(\R)$. (Contrary to the case of the circle.) We could define for every  "nice function" $f$,
$$\int_{-\infty}^{+\infty}\int_{-\infty}^{+\infty}f(t)\bar{f}(t')\mathcal{G}^{(\infty)}(t,t')dtdt'=\int_{-\infty}^{\infty}\frac{1}{M_u^2}\left| \int_u^{+\infty}f(t)M_tdt\right|^2du.$$
However, this quantity is not almost surely finite for every $f$ and the problem is to give sense to these "nice functions". We strongly suspect that the set of "nice functions" is actually $L^2(\R)\cap L^1(\R)$ (because of Proposition \ref{idenloifonc}). However, we are not able to prove it for now. It would also be possible to define $\mathcal{H}^{(\infty)}$ as the inverse of $\mathcal{G}^{(\infty)}$, that is, by $$\mathcal{H}^{(\infty)}g=-\frac{1}{M}\left(M^2\left(\frac{g}{M}\right)'\right)'.$$
Nevertheless, this is not clear what should be the domain of $\mathcal{H}^{(\infty)}$. We will try to solve this problem as soon as possible.
\end{rem}
\begin{rem}
At first sight, the expression of $\mathcal{G}^{(\infty)}$ is not symmetrical in the sense that  when $t'$ increases, then only $M_{t'}$ changes and when $t$ decreases, only $M_t\int_{-\infty}^t\frac{ds}{M_s^2}$ changes. The dissymmetry between $M_{t'}$ and $M_t\int_{-\infty}^t\frac{ds}{M_s^2}$ is surprising because the law of $G_{\beta}^{(\lambda,n)}$ is totally symmetric in the two directions of the circle. However, we will see that we need to choose an orientation of the circle because we will describe the potential $\beta$ by means of i.i.d Inverse Gaussian random variables $(A_i)_{i\in \mathcal{C}_{\lceil \lambda n\rceil}}$ as in equation \eqref{eqdefsimple}. It explains the apparent dissymetry in Proposition \ref{volumeinfini}. Moreover, Corollary \ref{MBG} explains why this dissimmetry does not exists "in law". It exists only "almost surely".
\end{rem}

Thanks to Proposition \ref{volumeinfini}, one can also prove a new functional identity in law:
\begin{cor}\label{MBG}
Let $\alpha$ be a standard Brownian motion on $\R_+$. Then the process
$$\left(e^{-\alpha_t+t/2}\frac{\displaystyle\int_t^{+\infty}e^{2\alpha_s-s}ds}{\displaystyle\int_{0}^{+\infty}e^{2\alpha_s-s}ds}\right)_{t\geq 0} $$
is a geometric Brownian motion starting from 1.
\end{cor}
\subsection{The asymptotic density of states}
Moreover, one can look for the asymptotic density of states of $\mathcal{H}^{(\lambda)}$ when $\lambda$ goes to infinity. For every $E\in\R_+^*$, let us define the random variable $N_{\lambda}(E)$ which is the number of eigenvalues of $\mathcal{H}^{(\lambda)}$ which are lower than $E$. Then, we have the following result:
\begin{thm}\label{DOS}
For every $E>0$, 
$$\frac{N_{\lambda}(E)}{2\lambda}\xrightarrow[\lambda\rightarrow+\infty]{\P}\frac{\sqrt{E}}{\pi}:=N_{\infty}(E).$$
\end{thm}
In some way, $N_{\infty}$ is the integrated density of states of $\mathcal{H}^{(\infty)}$. (However $\mathcal{H}^{(\infty)}$ can not be defined rigorously as an operator for now.)
\begin{rem}
The behaviour of the density of states around $0$ is reminiscent of what is going on for the operator $H_{\beta}$ on $\mathbb{Z}^d$. Indeed, by Theorem 1 and 3 in \cite{DOS}, if the weights $W$ are small enough, then the density of states of the discrete Schrödinger operator $H_{\beta}$ behaves like $\sqrt{E}$ near $0$ up to logarithmical corrections.
\end{rem}
\begin{rem}
The density of states $E\mapsto \frac{\sqrt{E}}{\pi}$ is exactly the density of states of $-\Delta$ on $\R$. However it is possible to see a difference  between the eigenvalues of $\mathcal{H}^{(\infty)}$ and $-\Delta$ if we look at the microscopic scale.
Moreover, in the case of the Anderson model, the integrated density of states can be computed explicitely. (see \cite{frilo}, \cite{Halperin} and \cite{Funak}) Then the behaviour at infinity of the integrated density of states is also $\frac{\sqrt{E}}{\pi}$, exactly as for our operator $\mathcal{H}^{(\infty)}$ and $-\Delta$.

\end{rem}
\subsection{Organisation of the paper}
\begin{itemize}
\item In section \ref{sec:background}, we remind several important facts regarding the $\beta$ potential.
\item Section \ref{sec:premlem} is devoted to the proof of a few lemmas which will be useful in the sequel of this paper.
\item In section \ref{sec:proofmatyor}, we give a new proof of Matsumoto-Yor properties by means of the $\beta$-potential. This section is independent from the next ones.
\item In other sections, we prove the results concerning the continuous versions of the operator $H_{\beta}$.
\end{itemize}
\section{Background on the $\beta$-potential}\label{sec:background}
Let $(V,E)$ a finite graph. The $\beta$-potential with distribution $\nu_V^W$ which is defined in \eqref{transfolaplace} is a special case of a more general family of random potentials which appear naturally when taking restrictions. Let us consider a symmetric matrix $W$ on $V\times V$ with non-negative entries $(W_{i,j})_{i,j\in V\times V}$ and let $\eta:=(\eta_i)_{i\in V}$ be a vector on $V$ with non-negative entries. Recall that for every $\beta\in\R_+^V$, $H_{\beta}$ is a matrix such that for every $i,j\in V$,
\begin{align*}
H_{\beta}(i,j)=\textbf{1}\{i=j\}2\beta_i-W_{i,j}.
\end{align*}
We can generalize the measure $\nu_V^W$ thanks to the following proposition.
\begin{prop}[Theorem 2.2 in \cite{Letac} or Lemma 4 in \cite{sabotzeng}]\label{prop:defgen}
Let us define the measure $\nu_V^{W,\eta}$ on $\R_+^V$ by
$$\nu_V^{W,\eta}(d\beta):=\textbf{1}\{H_{\beta}>0\}\left(\frac{2}{\pi}\right)^{|V|/2}e^{-\frac{1}{2}\langle {\bf{1}},H_{\beta} {\bf{1}}\rangle-\frac{1}{2}\langle \eta, (H_{\beta})^{-1}\eta\rangle+\langle \eta,\bf{1}\rangle}\frac{d\beta_V}{\sqrt{det(H_{\beta})}}$$
where ${\bf{1}}$ stands for a vector whose entries are all equal to 1. Then $\nu_V^{W,\eta}$ is a probability measure. Moreover, its Laplace transform is, for any $t\in\R_+^V$,
$$\int e^{-\langle t,\beta\rangle}\nu_V^{W,\eta}(d\beta)=e^{-\sum\limits_{i\in V}\eta_i(\sqrt{t_i+1}-1)-\frac{1}{2}\sum\limits_{\{i,j\}\in E} W_{i,j}\left(\sqrt{(1+t_i)(1+t_j)}-1\right)}\prod\limits_{i\in V}\frac{1}{\sqrt{1+t_i}}.$$
\end{prop}
Remark that the Laplace transform in \eqref{transfolaplace} is the same as the Laplace transform of $\nu_V^{W,0}$ given in Proposition \ref{prop:defgen}.
Further, the measures of type $\nu_{V}^{W,\eta}$ are convenient when we want to manipulate them because they form a family which is stable under restriction and conditioning.
\begin{prop}[Lemma 5 in \cite{sabotzeng} or Proposition 4.3 in \cite{Letac}]\label{prop:restrictions}
Let $U$ be a subset of $V$. If $\beta\sim\nu_V^{W,\eta }$, it holds that
\begin{enumerate}[(i)]
\item  $\beta_U$ follows the distribution $\nu_U^{W_{U,U},\hat{\eta}}$, where
$$\hat{\eta}=\eta_U+W_{U,U^c}({\bf{1}}_{U^c}),$$
\item Conditionally on $\beta_U$, $\beta_{U^c}$ follows the distribution $\nu_{U^c}^{\check{W},\check{\eta}}$, where
$$\check{W}=W_{U^c,U^c}+W_{U^c,U}((H_{\beta})_{U,U})^{-1}W_{U,U^c},\hspace{0.3 cm}\check{\eta}=\eta_{U^c}+W_{U^c,U}((H_{\beta})_{U,U})^{-1}\eta_U.$$

\end{enumerate}
\end{prop}

 Let $\beta\in\R_+^V$ be such that $H_{\beta}$ is positive definite. Let $G_{\beta}$ be the inverse of $H_{\beta}$. Let $i,j\in V$. A path in the graph $(V,E)$ from $i$ to $j$ consists in a finite sequence $\sigma=(\sigma_0,\cdots,\sigma_m)$ in $V$ such that $\sigma_0=i$, $\sigma_m=j$ and for every $k\in\{0,\cdots,m-1\}$, $\{\sigma_k,\sigma_{k+1}\}\in E$. The length $m$ of $\sigma$ will be denoted by $|\sigma|$. Let $\mathcal{P}_{i,j}^V$ be the set of paths from $i$ to $j$. We define also the set $\bar{{\mathcal{P}}}_{i,j}^V$ which is the collection of paths $\sigma$ from $i$ to $j$ such that $\sigma_k\neq j$ for every $k\in\{0,\cdots,m-1\}$. Moreover, for any path $\sigma$, we define
$$W_{\sigma}=\prod_{k=0}^{|\sigma|-1}W_{\sigma_k,\sigma_{k+1}},\hspace{0.3 cm}(2\beta)_{\sigma}=\prod\limits_{k=0}^{|\sigma|}(2\beta_{\sigma_k}),\hspace{0.3 cm}(2\beta)_{\sigma}^-=\prod\limits_{k=0}^{|\sigma|-1}(2\beta_{\sigma_k}).$$
Then, we have the following useful description of $G_{\beta}$:
\begin{prop}\label{sommes sur les chemins}
Let $(\beta_i)_{i\in V}$ be a random field on $V$ with distribution $\nu_V^{W,\eta}$. Then for every $i,j\in V$, almost surely,
$$G_{\beta}(i,j)=\sum\limits_{\sigma\in \mathcal{P}_{i,j}^V}\frac{W_{\sigma}}{(2\beta)_{\sigma}},\hspace{1 cm}\frac{G_{\beta}(i,j)}{G_{\beta}(i,i)}=\sum\limits_{\sigma\in\bar{{\mathcal{P}}}_{j,i}^V}\frac{W_{\sigma}}{(2\beta)_{\sigma}^-}.$$
\end{prop}
\section{Preliminary lemmas} \label{sec:premlem}
In this section we will prove a few lemmas about Inverse Gaussian random variables which will be crucial in the sequel of this paper. 
\begin{lem}\label{lem:estimees}
Let $A^{(n)}$ be an Inverse Gaussian random variable with parameters $(1,n)$. Then we know that
$$\E\left[\ln(A^{(n)})\right]=-\frac{1}{2n}+o\left(\frac{1}{n} \right)\text{ and } Var\left[\ln(A^{(n)})\right]=\frac{1}{n}+o\left(\frac{1}{n}\right).$$
Moreover, for any $n\in\N\backslash\{0,1\}$ and for any $v>0$, it holds that
$$\P\left[|\ln(A^{(n)})|>v\right]\leq \frac{2e}{v\sqrt{\pi(n-1)}}e^{-(n-1)v^2/2}.$$
\end{lem}
\begin{proof}
This proof is very similar with the proof of Lemma 3.4 in \cite{LST} but we do it again here for the paper to be self-contained. Let $n\in\N^*$.
The density of $\ln(A^{(n)})$ is
$$\frac{\sqrt{n}}{\sqrt{2\pi}}e^{-u/2-2n\sinh(u/2)^2}du.$$
Therefore,
\begin{align}
\E\left[\ln(A^{(n)})\right]&=\int_{-\infty}^{+\infty}\frac{\sqrt{n}}{\sqrt{2\pi}}ue^{-u/2-2n\sinh\left(\frac{u}{2}\right)^2}du\nonumber\\
&=\frac{\sqrt{n}}{\sqrt{2\pi}}\int_0^{+\infty}\left(ue^{-u/2}-ue^{u/2}\right)e^{-2n\sinh\left(\frac{u}{2}\right)^2}du\nonumber\\
&=-\frac{8\sqrt{n}}{\sqrt{2\pi}}\int_0^{+\infty}u\sinh(u)e^{-2n\sinh\left(u\right)^2}du.\label{premlem0}\\
\end{align}
Now, let us do the change of variable $t=\sinh(u)$ in \eqref{premlem0}. It yields
\begin{align}
\E\left[\ln(A^{(n)})\right]&=-\frac{8\sqrt{n}}{\sqrt{2\pi}}\int_0^{+\infty }\frac{t\times\mathrm{ argsinh}(t)}{\sqrt{1+t^2}}e^{-2nt^2}dt\nonumber\\
&=-\frac{1}{n}\times\frac{8}{\sqrt{2\pi}}\int_0^{+\infty}\frac{\sqrt{n}t\times\mathrm{ argsinh}(t/\sqrt{n})}{\sqrt{1+t^2/n}}e^{-2t^2}dt.\label{premlem01}
\end{align}
Besides, for every $t>0$ and for every $n\in\N^*$, $\frac{\sqrt{n}t\times\mathrm{ argsinh}(t/\sqrt{n})}{\sqrt{1+t^2/n}}e^{-2t^2}\leq t^2e^{-2t^2}$. Therefore, we can apply the dominated convergence theorem in \eqref{premlem01} which implies 
\begin{align}
\E\left[\ln(A^{(n)})\right]&=-\frac{1}{n}\times\frac{8}{\sqrt{2\pi}}\int_0^{+\infty}t^2e^{-2t^2}dt+o\left(\frac{1}{n}\right)=-\frac{1}{2n}+o\left(\frac{1}{2n}\right).\label{premlem02}
\end{align}
Now, let us define $B^{(n)}=\sqrt{n}\left(\ln(A^{(n)})+\frac{1}{2n}\right)$. Observe that the density of $B^{(n)}$ is
\begin{align*}
\frac{1}{\sqrt{2\pi}}e^{-\frac{1}{2}\left(\frac{v}{\sqrt{n}}-\frac{1}{2n} \right)-2n\sinh\left(\frac{1}{2}\left(\frac{v}{\sqrt{n}}-\frac{1}{2n} \right)\right)^2}dv.
\end{align*}
Therefore, as $\sinh(x)^2\geq x^2$ for every $x\in\R$, for any positive function $F$ of $\R$ into itself,
\begin{align}
\E\left[F(B^{(n)})\right]&=\int_{-\infty}^{+\infty}F(v)\frac{1}{\sqrt{2\pi}}e^{-\frac{1}{2}\left(\frac{v}{\sqrt{n}}-\frac{1}{2n} \right)-2n\sinh\left(\frac{1}{2}\left(\frac{v}{\sqrt{n}}-\frac{1}{2n} \right)\right)^2}dv\nonumber\\
&\leq \frac{1}{\sqrt{2\pi}}e^{\frac{1}{8n}}\int_{-\infty}^{\infty}F(v)e^{-v^2/2}dv.\label{premlem1}
\end{align}
Consequently, by the dominated convergence theorem,
\begin{align}
\E\left[{B^{(n)}}^2\right]\xrightarrow[n\rightarrow+\infty]{}\frac{1}{\sqrt{2\pi}}\int_{-\infty}^{+\infty}v^2e^{-v^2/2}dv&=1.\label{premlem2}
\end{align}
Combining \eqref{premlem2} with \eqref{premlem02}, we get
$$Var\left[\ln(A^{(n)})\right]=\frac{1}{n}+o\left(\frac{1}{n}\right).$$
Now, let us look at the tail of $\ln(A^{(n)})$. Let $v>0$.
\begin{align}
\P\left(|\ln(A^{(n)})|>v \right)&=\int_{\R\backslash[-v,v]}\frac{\sqrt{n}}{\sqrt{2\pi}}\exp\left(-2n\sinh(x/2)^2-x/2 \right)dx\nonumber\\
&\leq \int_{\R\backslash[-v,v]}\frac{\sqrt{n}}{\sqrt{2\pi}}\exp\left(-nx^2/2-x/2 \right)dx\nonumber\\
&\leq 2e\int_v^{+\infty}\frac{\sqrt{n}}{\sqrt{2\pi}}\exp\left(-(n-1)x^2/2 \right)dx\label{supremum1}
\end{align}
where in the first inequality we used the fact that $\sinh(x)\geq x$ for every $x>0$ and in the second inequality we used the fact that for every $x\in \R$, $e^{-x/2}\leq e\times e^{x^2/2}$.
Therefore, by \eqref{supremum1}, for every $v>0$,
\begin{align}
\P\left(|\ln(A^{(n)})|>v \right)&\leq 2e\int_{v\sqrt{n-1}}^{+\infty}\frac{\sqrt{n}}{\sqrt{2\pi(n-1)}}\exp(-x^2/2)dx\nonumber\\
&\leq \frac{2e}{\sqrt{\pi}}\int_{v\sqrt{n-1}}^{+\infty}\exp(-x^2/2)dx\nonumber\\
&\leq \frac{2e}{v\sqrt{\pi(n-1)}}e^{-(n-1)v^2/2}\label{supremum2}
\end{align}
where in the second inequality we used the fact that $n\leq 2(n-1)$ for every $n\geq 2$ and in the last inequality we used the fact that for $x\geq v\sqrt{n-1}$, it holds that $1\leq x/(v\sqrt{n-1})$.
It conludes the proof of Lemma \ref{lem:estimees}.
\end{proof}
\begin{lem}\label{lem:supremum}
Let $c>0$. Let $n\in\N^*$. Let $(A_i^{(n)})_{i\in\N^*}$ be a sequence of independent Inverse Gaussian random variables with parameters $(1,n)$. Then, for every $\varepsilon>0$, it holds that
$$\underset{n\rightarrow+\infty}\lim\hspace{0.1 cm}\P\left(\underset{ i\in\llbracket 1,\lceil cn\rceil\rrbracket}\sup \hspace{0.1 cm} |\ln(A_i^{(n)})|>\varepsilon\right)=0.$$
\end{lem}
\begin{proof}
Let $\varepsilon>0$. By Lemma \ref{lem:supremum}, we know that
\begin{align}
\P\left(\underset{ i\in\llbracket 1,\lceil cn\rceil\rrbracket}\sup \hspace{0.1 cm} |\ln(A_i^{(n)})|>\varepsilon\right)&=1-\left(1-\P\left(|\ln(A_1^{(n)})|>\varepsilon \right) \right)^{\lceil c n\rceil}\nonumber\\
&\leq 1-\left(1-  \frac{2e}{\varepsilon\sqrt{\pi(n-1)}}e^{-(n-1)\varepsilon^2/2}\right)^{\lceil c n\rceil}
\end{align}
which goes to $0$ as $n$ goes to infinity.
\end{proof}
\begin{lem}\label{lem:convergenceprocessus}
Let $n\in\N^*$. Let $(A_i^{(n)})_{i\in\N^*}$ be a sequence of independent random variables which are distributed as an Inverse Gaussian random variable with parameters $(1,n)$. Let us define the process $t\mapsto Y_t^{(n)}$ which is a random continuous function such that  if $j/n\leq t<(j+1)/n)$,
$$Y_t^{(n)}=\prod\limits_{i=1}^{ j}A_i^{(n)} +n(t-j/n)\left(\prod\limits_{i=1}^{ j+1}A_i^{(n)}-\prod\limits_{i=1}^{ j}A_i^{(n)} \right).$$
Then, the following convergence holds for the topology of uniform convergence on compact sets:
$$(Y_t^{(n)})_{t\geq 0}\xrightarrow[n\rightarrow+\infty]{law}(e^{\alpha_t-t/2} )_{t\geq 0}$$
where $\alpha$ is a Brownian motion.
\end{lem}
\begin{proof}
By Lemma \ref{lem:supremum}, for every $T> 0$,
$$\left(\ln(Y_t^{(n)})\right)_{t\in[0,T]}=\left(\sum\limits_{i=1}^{\lfloor tn\rfloor}\ln(A_i^{(n)})-\E\left[\ln(A_i^{(n)})\right] +\sum\limits_{i=1}^{\lfloor tn\rfloor}\E\left[\ln(A_i^{(n)})\right]\right)_{t\in[0,T]}+o_{n,T}^{\P}(1)$$
where $o_{n,T}^{\P}(1)$ is a random function whose supremum goes toward $0$ in probability when $n$ goes to infinity. By Lemma \ref{lem:estimees}, we know that $\sum\limits_{i=1}^{\lfloor tn\rfloor}\E\left[\ln(A_i^{(n)})\right]$ converges toward $-t/2$.
Moreover, $t\mapsto \sum\limits_{i=1}^{\lfloor tn\rfloor}\ln(A_i^{(n)})-\E\left[\ln(A_i^{(n)})\right]$ is a martingale. Furthermore, as in the proof of Lemma 3.4 in \cite{LST}, one can combine the estimates of Lemma \ref{lem:estimees} and the martingale functional central limit theorem (see Theorem 1.4, Section 7.1 in \cite{EK}) in order to prove that
$$\left(\sum\limits_{i=1}^{\lfloor tn\rfloor}\ln(A_i^{(n)})-\E\left[\ln(A_i^{(n)})\right]\right)_{t\geq 0}$$
converges toward a Brownian motion.
\end{proof}
\begin{lem}\label{lem:couplage}
Let $K\geq 1$. Let $c>0$.
Let $A_1$ be an Inverse Gaussian random variable with parameters $(1,K)$. Then, it is possible to find a coupling with a random variable $A_2$ which  has an Inverse Gaussian distribution  with parameters $(1,K+c)$ such that
$$|\ln(A_1)-\ln(A_2)|\leq \left(\frac{1}{A_1}+\frac{1}{A_2}\right)\times\left(\frac{cR^{(1)}}{K^{3/2}}+Ber\times\frac{cR^{(2)}+\sqrt{c}R^{(3)}+R^{(4)}}{\sqrt{K}}\right)$$
where for every $i\in\{1,2,3,4\}$, $R^{(i)}$ is a positive random variable and conditionally on $\{R^{(i)},i\in\{1,2,3,4\}\}$, $Ber$ is a Bernoulli random variable whose parameter is smaller than $\frac{cR^{(1)}}{K^{3/2}}$. Moreover, there exists a positive constant $\kappa$ which does not depend on $K$ and $c$ such that for every $i\in\{1,2,3,4\}$, $\E\left[{R^{(i)}}^4\right]\leq \kappa$.
\end{lem}
\begin{proof}
Following \cite{wiki}, we can construct $A_1$ and $A_2$ in the following way:
Let $\gamma$ be a Gamma random variable with parameters $(1/2,1)$. Let $U$ be a uniform random variable which is independent of $\gamma$.
Now, let us consider
$$X_1=1+\frac{\gamma}{K}-\frac{\sqrt{\gamma}}{K}\sqrt{2K+\gamma}$$
and $$X_2=1+\frac{\gamma}{K+c}-\frac{\sqrt{\gamma}}{K+c}\sqrt{2(K+c)+\gamma}.$$
For every $i\in\{1,2\}$, if $U\leq \frac{1}{1+X_i}$, then we define $A_i=X_i$ and if $U>\frac{1}{1+X_i}$, then we define $A_i=1/X_i$. According to \cite{wiki} , $A_1\sim IG(1,K)$ and $A_2\sim IG(1,K+c)$. Now, let us show that this coupling satisfies the required estimate.

First remark that
\begin{align}
|\ln(A_1)-\ln(A_2)|&=\left|\ln\left( \frac{A_1}{A_2}\right)\right|\nonumber\\
&\leq\ln\left(1+\frac{|A_1-A_2|}{A_1}\right)+\ln\left(1+\frac{|A_1-A_2|}{A_2}\right)\nonumber\\
&\leq \left(\frac{1}{A_1}+\frac{1}{A_2}\right)|A_1-A_2|.\label{premlem3}
\end{align}
Therefore, it is enough to find an upper bound for $|A_1-A_2|$ in order to prove Lemma \ref{lem:couplage}. To do this, we have to consider three situations.\\
\textbf{Situation 1:} Let us assume that $U\leq \frac{1}{1+X_1}$ and $U\leq \frac{1}{1+X_2}$. Then, it holds that
\begin{align}
|A_1-A_2|&=|X_1-X_2|\nonumber\\
&\leq \frac{\gamma c}{K^2}+\sqrt{\gamma}\left| \frac{1}{K}\sqrt{2K+\gamma}-\frac{1}{K+c}\sqrt{2(K+c)+\gamma}\right|\nonumber\\
&\leq \frac{\gamma c}{K^2}+
\sqrt{\gamma}\sqrt{2K+\gamma}\left(\frac{1}{K}-\frac{1}{K+c}\right)+
\sqrt{\gamma}\frac{1}{K+c}\left(\sqrt{2(K+c)+\gamma}-\sqrt{2K+\gamma}\right)\nonumber\\
&\leq\frac{\gamma c}{K^2}+c\sqrt{\gamma}\frac{\sqrt{2K}+\sqrt{\gamma}}{K^2}+\sqrt{\gamma}\frac{1}{K}\frac{2c}{\sqrt{2(K+c)+\gamma}+\sqrt{2K+\gamma}}\nonumber\\
&\leq \frac{\gamma c}{K^2}+c\sqrt{\gamma}\frac{\sqrt{2K}+\sqrt{\gamma}}{K^2}+\frac{\sqrt{\gamma} c}{\sqrt{2}K^{3/2}}.\label{premlem4}
\end{align}
Therefore, by \eqref{premlem4}, in the situation 1, there exists a positive random variable $R^{(1)}$ whose fourth moment is bounded by some constant $\kappa^{(1)}$ which does not depend on $K$ and $c$ such that
\begin{align}
|A_1-A_2|\leq\frac{cR^{(1)}}{K^{3/2}}.\label{premlem5}
\end{align}
Rermark that $U$ and $R^{(1)}$ are independent.\\
\textbf{Situation 2:} Now, let us assume that $U> \frac{1}{1+X_1}$ and $U> \frac{1}{1+X_2}$.
Remark that
$$\frac{1}{X_1}=1+\frac{\gamma}{K}+\frac{\sqrt{\gamma}}{K}\sqrt{2K+\gamma}$$
and
$$\frac{1}{X_2}=1+\frac{\gamma}{K+c}+\frac{\sqrt{\gamma}}{K+c}\sqrt{2(K+c)+\gamma}.$$
Therefore, exactly as in the first situation, one can show that
\begin{align}
|A_1-A_2|\leq\frac{cR^{(1)}}{K^{3/2}}.\label{premlem6}
\end{align}\\
\textbf{Situation 3:}
Now, let us consider the case where $U\leq \frac{1}{1+X_1}$ and $U> \frac{1}{1+X_2}$ or the case where $U> \frac{1}{1+X_1}$ and $U\leq \frac{1}{1+X_2}$. These two subcases are similar. Thus we will only treat the first one. If we assume that $U\leq \frac{1}{1+X_1}$ and $U> \frac{1}{1+X_2}$, then
\begin{align}
|A_1-A_2|&=\left|X_1-\frac{1}{X_2} \right|\nonumber\\
&\leq \frac{\gamma c}{K^2}+\frac{\sqrt{\gamma}}{K}\sqrt{2K+\gamma}+\frac{\sqrt{\gamma}}{K+c}\sqrt{2(K+c)+\gamma}\nonumber\\
&\leq \frac{cR^{(2)}+\sqrt{c}R^{(3)}+R^{(4)}}{\sqrt{K}}\label{premlem7}
\end{align}
where $R^{(2)},R^{(3)}$ and $R^{(4)}$ are positive random variables whose fourth moments are bounded by some constant $\kappa^{(2)}$ which does not depend on $c$ and $K$. Remark that $U$ is independent from $R^{(1)}$, $R^{(2)}$, $R^{(3)}$ and $R^{(4)}$. Moreover, in this situation 3, we know that $U\in\left[\frac{1}{1+X_1},\frac{1}{1+X_2}\right]$ or $U\in\left[\frac{1}{1+X_2},\frac{1}{1+X_1}\right]$. Therefore, $U$ belongs to some interval $\mathcal{I}$ whose size is
\begin{align}
\left|\frac{1}{1+X_1} -\frac{1}{1+X_2}\right|&=\frac{|X_1-X_2|}{(1+X_1)(1+X_2)}\nonumber\\
&\leq |X_1-X_2|\nonumber\\
&\leq \frac{cR^{(1)}}{K^{3/2}}\label{premlem8}
\end{align}
where in the last inequality, we used \eqref{premlem5}. Together with \eqref{premlem7}, it implies that, in situation 3,
\begin{align}
|A_1-A_2|\leq \left(\frac{cR^{(2)}+\sqrt{c}R^{(3)}+R^{(4)}}{\sqrt{K}}\right)\textbf{1}\{U\in\mathcal{I}\}\label{premlem9}
\end{align}
where the size of $\mathcal{I}$ is lower than $\frac{cR^{(1)}}{K^{3/2}}$ with $R^{(1)}$ independent of $U$. Finally, choosing $\kappa=\max(\kappa^{(1)},\kappa^{(2)})$ and combining \eqref{premlem5}, \eqref{premlem6} and \eqref{premlem9} concludes the proof.
\end{proof}
\vspace{2cm}
\section{Proof of the results of section \ref{subsubsec:matyor} }
\label{sec:proofmatyor}
First, let us prove Proposition \ref{prop:discretematyor}.
\begin{proof}[Proof of Proposition \ref{prop:discretematyor}]
\textbf{Step 1: Proof of (i) and (iii).}
One remarks that (iii) is just a particular case of Lemma 4.1 in \cite{rapmart}. Actually, in Lemma 4.1 in \cite{rapmart}, we condition with respect to the sigma-field $\sigma(\beta_i,i\in\llbracket 2,n\rrbracket)$ and not with respect to $\mathcal{Z}_{n,m}$. Nevertheless, we will see in the proof of $(ii)$ that for every $n\in\N^*$ and for every $m\in\N^*$,
$$\mathcal{Z}_{n,m}\subset\sigma(\beta_i,i\in\llbracket 2,n\rrbracket).$$
Moreover (i) stems directly from (iii). 
\\
\textbf{Step 2: Proof of (ii).}
Let $m\geq 1$. For sake of convenience, for every $n\in\N^*$, we denote ${\bf{H}}_n=({\bf{H}}_{\beta}^{(m)})_{\llbracket1,n\rrbracket,\llbracket 1,n\rrbracket}$ and $\hat{\bf{G}}_n={\bf{H}}_n^{-1}$. The strategy of the proof is the following one: first, we will establish an algebraic relation between $(\hat{\bf{G}}_n,\psi_{\beta}^{(n,m)})$ and $(\hat{\bf{G}}_{n-1},\psi_{\beta}^{(n-1,m)})$ by means of the Schur complements. Then we will condition this algebraic relation with respect to the $\sigma$-field $\sigma(\beta_i,i\in\llbracket 2,n-1\rrbracket)$ thanks to the conditioning properties of $\beta$ given by Proposition \ref{prop:restrictions}. We will divide this proof into two main lemmas. Here is the first one.
\begin{lem}\label{lemslice1}
Let $n\in\N^*\backslash\{1\}$. It holds that,
\begin{align}
Z_{\beta}^{(n,m)}
&=\frac{Z_{\beta}^{(n-1,m)}}{m}\left(2\beta_n-m^2\hat{\bf{G}}_{2,n-1}(n-1,n-1)\right)
\end{align}
where $\hat{\bf{G}}_{2,n-1}$ is the inverse of $({\bf{H}}_{\beta}^{(m)})_{\llbracket 2,n-1\rrbracket,\llbracket 2,n-1\rrbracket}$.
\end{lem}
\begin{proof}[Proof of Lemma \ref{lemslice1}]
For every $n\in\N^*$, let $C_{n}$ be a vector of size $n$ such that
$$C_n=\begin{pmatrix}
0\\
\vdots\\
0\\
-m
\end{pmatrix}.$$
With this notation, remark that for every integer $n\geq 2$,
$${\bf{H}}_n=\begin{pmatrix}
{\bf{H}}_{n-1} & C_{n-1}\\
C_{n-1}^T & 2\beta_n
\end{pmatrix}.$$
For every $n\geq 2$, let us define $D_n=2\beta_n-C_{n-1}^T\hat{\bf{G}}_{n-1}C_{n-1}$. Using the Schur complement, we get that for every integer $n\geq 2$,
\begin{align}
\hat{\bf{G}}_n=\left(\begin{array}{c|c}
\hat{\bf{G}}_{n-1}+\hat{\bf{G}}_{n-1}C_{n-1}D_n^{-1}C_{n-1}^T\hat{\bf{G}}_{n-1}& -\hat{\bf{G}}_{n-1}C_{n-1}D_n^{-1}\\ \hline
*&*
\end{array}\right).\label{labmy1}
\end{align}
Now, let us fix an integer $n\geq 2$. If we apply \eqref{labmy1} at points $(1,1)$ and $(1,n)$, we obtain
\begin{align}
&\hat{\bf{G}}_n(1,1)=\hat{\bf{G}}_{n-1}(1,1)+m^2D_n^{-1}\hat{\bf{G}}_{n-1}(1,n-1)^2,\label{labmy2}\\
&\psi^{(n,m)}_{\beta}=m\hat{\bf{G}}_n(1,n)=m^2D_n^{-1}\hat{\bf{G}}_{n-1}(1,n-1)=mD_n^{-1}\psi_{\beta}^{(n-1,m)}.\label{labmy3}
\end{align}
Therefore, combining \eqref{labmy2} and \eqref{labmy3}, we get
\begin{align}
Z_{\beta}^{(n,m)}=\frac{\hat{\bf{G}}_n(1,1)}{\psi^{(n,m)}_{\beta}}&=\frac{\hat{\bf{G}}_{n-1}(1,1)}{mD_n^{-1}\psi_{\beta}^{(n-1,m)}}+\hat{\bf{G}}_{n-1}(1,n-1)\nonumber\\
&=Z_{\beta}^{(n-1,m)}\times\frac{D_n}{m}+\hat{\bf{G}}_{n-1}(1,n-1).\label{labmy4}
\end{align}
Moreover, $D_n=2\beta_n-m^2\hat{\bf{G}}_{n-1}(n-1,n-1)$. Together with \eqref{labmy4}, it yields
\begin{align}
Z_{\beta}^{(n,m)}&=\frac{Z_{\beta}^{(n-1,m)}}{m}\times2\beta_n-mZ_{\beta}^{(n-1,m)}\hat{\bf{G}}_{n-1}(n-1,n-1)+\hat{\bf{G}}_{n-1}(1,n-1)\nonumber\\
&\hspace{-0.8 cm}=\frac{Z_{\beta}^{(n-1,m)}}{m}\times2\beta_n-\hat{\bf{G}}_{n-1}(n-1,n-1)\frac{\hat{\bf{G}}_{n-1}(1,1)}{\hat{\bf{G}}_{n-1}(1,n-1)}+\hat{\bf{G}}_{n-1}(1,n-1)\nonumber\\
&\hspace{-0.8 cm}=\frac{Z_{\beta}^{(n-1,m)}}{m}\times2\beta_n-\frac{\hat{\bf{G}}_{n-1}(1,1)}{\hat{\bf{G}}_{n-1}(1,n-1)}\left(\hat{\bf{G}}_{n-1}(n-1,n-1)-\frac{\hat{\bf{G}}_{n-1}(1,n-1)}{\hat{\bf{G}}_{n-1}(1,1)}\hat{\bf{G}}_{n-1}(1,n-1) \right)\label{labmy5}
\end{align}
where the second equality comes from the definition of $Z_{\beta}^{(n-1,m)}$. Besides, according to Proposition \ref{sommes sur les chemins}, $\hat{\bf{G}}_{n-1}(n-1,n-1)$ can be interpreted as a sum over the set of paths from $n-1$ to $n-1$ in $\llbracket 1,n-1\rrbracket$. Moreover, thanks to Proposition \ref{sommes sur les chemins} again, $\frac{\hat{\bf{G}}_{n-1}(1,n-1)}{\hat{\bf{G}}_{n-1}(1,1)}\hat{\bf{G}}_{n-1}(1,n-1)$ can be interpreted as a sum over the set of paths from $n-1$ to $n-1$ in $\llbracket 1,n-1\rrbracket$ which go through $1$. Therefore the difference,
$$\hat{\bf{G}}_{n-1}(n-1,n-1)-\frac{\hat{\bf{G}}_{n-1}(1,n-1)}{\hat{\bf{G}}_{n-1}(1,1)}\hat{\bf{G}}_{n-1}(1,n-1)$$
can be interpreted as a sum over the set of paths from $n-1$ to $n-1$ in $\llbracket 2,n-1\rrbracket$. Consequently,  by Proposition \ref{sommes sur les chemins} again,
$$\hat{\bf{G}}_{n-1}(n-1,n-1)-\frac{\hat{\bf{G}}_{n-1}(1,n-1)}{\hat{\bf{G}}_{n-1}(1,1)}\hat{\bf{G}}_{n-1}(1,n-1) =\hat{\bf{G}}_{2,n-1}(n-1,n-1)$$
where $\hat{\bf{G}}_{2,n-1}$ is the inverse of $({\bf{H}}_{\beta}^{(m)})_{\llbracket 2,n-1\rrbracket,\llbracket 2,n-1\rrbracket}$. Together with \eqref{labmy5}, it implies that
\begin{align}
Z_{\beta}^{(n,m)}&=Z_{\beta}^{(n-1,m)}\times\frac{2\beta_n}{m}-mZ_{\beta}^{(n-1,m)}\hat{\bf{G}}_{2,n-1}(n-1,n-1)\nonumber\\
&=\frac{Z_{\beta}^{(n-1,m)}}{m}\left(2\beta_n-m^2\hat{\bf{G}}_{2,n-1}(n-1,n-1)\right).\label{labmy6}
\end{align}
\end{proof}
Now, let us enounce the second fundamental lemma of this proof.
\begin{lem}
\label{lemslice2}
For every $n\in\N^*\backslash\{1\}$, it holds that
\begin{align*}
\mathcal{L}\left(2\beta_n|\sigma\left(\beta_i,i\in\llbracket 2,n-1\rrbracket \right)\right)
&=m^2\hat{\bf{G}}_{2,n-1}(n-1,n-1)+\frac{1}{IG\left(\frac{1}{m+\frac{1}{Z_{\beta}^{(n-1,m)}}},1 \right)}.
\end{align*}
\end{lem}
 \begin{proof}[Proof of Lemma \ref{lemslice2}]
Now, let us condition $\beta_n$ in Lemma \ref{lemslice1} with respect to $\sigma(\beta_i,i\in\llbracket 2,n-1\rrbracket)$. Recall that we assumed $\beta\sim \nu_{\N^*}^{K_m}$.  Thanks to Proposition \ref{prop:restrictions}, conditionally on $\sigma(\beta_i,i\in\llbracket 2,n-1\rrbracket)$, $\beta_n$ is distributed as $\nu_{\{n\}}^{\eta_{n,m},W_{n,m}}$ where
$$W_{n,m}=m^2\hat{\bf{G}}_{2,n-1}(n-1,n-1)$$
and
$$\eta_{n,m}=m+m^2\hat{\bf{G}}_{2,n-1}(2,n-1).$$
However, thanks to Proposition \ref{sommes sur les chemins},
$$m\hat{\bf{G}}_{2,n-1}(2,n-1)=\frac{\hat{\bf{G}}_{n-1}(1,n-1)}{\hat{\bf{G}}_{n-1}(1,1)}.$$
Consequently,
$$\eta_{n,m}=m+m\frac{\hat{\bf{G}}_{n-1}(1,n-1)}{\hat{\bf{G}}_{n-1}(1,1)}=m+\frac{1}{Z_{\beta}^{(n-1,m)}}.$$
Therefore, 
\begin{align*}
\mathcal{L}\left(2\beta_n|\sigma\left(\beta_i,i\in\llbracket 2,n-1\rrbracket \right)\right)&=W_{n,m}+\frac{1}{IG\left(\frac{1}{\eta_{n,m}},1\right)}\\
&=m^2\hat{\bf{G}}_{2,n-1}(n-1,n-1)+\frac{1}{IG\left(\frac{1}{m+\frac{1}{Z_{\beta}^{(n-1,m)}}},1 \right)}.
\end{align*}
\end{proof}
Combining Lemmas \ref{lemslice1} and \ref{lemslice2}, it holds that

\begin{align}
\mathcal{L}\left(Z_{\beta}^{(n,m)}|\sigma\left(\beta_i,i\in\llbracket 2,n-1\rrbracket \right) \right)=\frac{Z_{\beta}^{(n-1,m)}}{m}\times\frac{1}{IG\left(\frac{1}{m+\frac{1}{Z_{\beta}^{(n-1,m)}}},1 \right)}.\label{labmy7}
\end{align}
Moreover, by Proposition \ref{sommes sur les chemins}, for every $k\in\llbracket 1,n-1\rrbracket$,
\begin{align}
\frac{1}{Z_{\beta}^{(k,m)}}=\frac{m\hat{\bf{G}}_k(1,k)}{\hat{\bf{G}}_k(1,1)}=m\sum\limits_{\sigma\in\bar{{\mathcal{P}}}_{k,1}^{\llbracket 1,k\rrbracket}}\frac{W_{\sigma}}{(2\beta)_{\sigma}^-}.\label{labmy8}
\end{align}
Remark that the sum in the right-hand side of \eqref{labmy8} never contains $\beta_1$. Thus, for every $k\in\llbracket 1,n-1\rrbracket$, $Z_{\beta}^{(k,m)}$ is measurable with respect to $\sigma\left(\beta_i,i\in\llbracket 2,n-1\rrbracket \right)$. This implies that 
$$\mathcal{Z}_{n-1,m}\subset \sigma\left(\beta_i,i\in\llbracket 2,n-1\rrbracket \right).$$
Together with \eqref{labmy7}, it yields
\begin{align}
\mathcal{L}\left(Z_{\beta}^{(n,m)}|\mathcal{Z}_{n-1,m} \right)=\frac{Z_{\beta}^{(n-1,m)}}{m}\times\frac{1}{IG\left(\frac{1}{m+\frac{1}{Z_{\beta}^{(n-1,m)}}},1 \right)}.\nonumber
\end{align}
It conludes the proof of (ii).
\end{proof}
Now, we still have to prove that $(\psi_{\beta}^{(n,m)})_{n\in\N^*}$ and $(Z_{\beta}^{(n,m)})_{n\in\N^*}$ converge toward the exponential functionals of the Brownian motion introduced by Matsumoto and Yor when we take the scaling limit as $m$ goes to infinity. To do so, we need first to prove a lemma which gives a useful representation of the $\beta$-field with distribution $\nu_{\N^*}^{K_m}$. Note that this construction is very specific to the one-dimensional structure of the graph.
\begin{lem}\label{lem:easyconstruction}
Let $m\in\N^*$. Let $(A_i^{(m)})_{i\in\N^*}$ be a sequence of independent Inverse Gaussian random variables with parameters $(1,m)$. We define $\beta_1=\frac{m}{2A_1^{(m)}}$ and for every $i\in\N^*\backslash\{1\}$,
$$\beta_i=\frac{m}{2}A_{i-1}^{(m)}+\frac{m}{2A_i^{(m)}}.$$
Then, $\beta\sim\nu_{\N^*}^{K_m}$.
\end{lem}
\begin{proof}
For every  $i\in\N^*$, let us write $A_i$ for $A_i^{(m)}$ for sake of convenience. Let $k\in\N^*$.
Let $(t_i)_{i\in \llbracket 1,k\rrbracket}\in \R_+^{k}$. Then, it holds that
$$
\begin{array}{ll}
\E\left[\exp\left(-\frac{t_1m}{2A_1}-\sum\limits_{i=2}^k\frac{t_im}{2}\left(A_{i-1}+\frac{1}{A_i}\right)\right) \right]&\\
&\hspace{-6 cm}=\displaystyle\E\left[\exp\left(-\sum\limits_{i=1}^{k-1}\frac{m}{2}\left(t_{i+1}A_{i}+\frac{t_i}{A_i}\right)\right)\times\exp\left(-\frac{m}{2}\frac{t_k}{A_k} \right)\right]\\
&\hspace{-6cm}\displaystyle=\prod\limits_{i=1}^{k-1}\E\left[\exp\left(-\frac{m}{2}\left(t_{i+1}A_{i}+\frac{t_i}{A_i}\right)\right) \right]\times\E\left[\exp\left(-\frac{m}{2}\frac{t_k}{A_k} \right) \right]\\
&\hspace{-6cm}=\displaystyle\prod\limits_{i=1}^{k-1}\int_0^{+\infty}\frac{\sqrt{m}}{\sqrt{2\pi x^3}}e^{-\frac{m(x-1)^2}{2x}} e^{-m\frac{t_{i+1}x+t_i/x}{2}}dx\times\int_{0}^{+\infty}\frac{\sqrt{m}}{\sqrt{2\pi x^3}}e^{-\frac{m(x-1)^2}{2x}} e^{-m\frac{t_k}{2x}}dx.
\end{array}$$
Moreover, for every $i\in \llbracket 1,k-1\rrbracket$, remark that
\begin{align*}
\int_0^{+\infty}\hspace{-0.3 cm}\frac{\sqrt{m}}{\sqrt{2\pi x^3}}e^{-\frac{m(x-1)^2}{2x}} e^{-m\frac{t_{i+1}x+t_i/x}{2}}dx&=\int_0^{+\infty}\frac{\sqrt{m(1+t_i)}}{\sqrt{2\pi x^3}}\exp\left(-m(1+t_i)\frac{\left(x-\sqrt{\frac{1+t_i}{1+t_{i+1}}}\right)^2}{2x\frac{1+t_i}{1+t_{i+1}}} \right)dx\\
&\hspace{0.4 cm}\times\frac{1}{\sqrt{1+t_i}}\times \exp\left(-m\left(\sqrt{1+t_i}\sqrt{1+t_{i+1}}-1 \right)\right)\\
&=\frac{1}{\sqrt{1+t_i}}\exp\left(-m\left(\sqrt{1+t_i}\sqrt{1+t_{i+1}}-1 \right)\right)
\end{align*}
because in the first equality we recognised the density of an Inverse Gaussian random variable with parameters $\left(\sqrt{\frac{1+t_i}{1+t_{i+1}}},m(1+t_i)\right)$ for every $i\in\llbracket 1,k-1\rrbracket$. Besides, one can prove in the same way that
$$\int_{0}^{+\infty}\frac{\sqrt{m}}{\sqrt{2\pi x^3}}e^{-\frac{m(x-1)^2}{2x}} e^{-m\frac{t_k}{2x}}dx=\frac{1}{\sqrt{1+t_k}}\exp\left(-m(\sqrt{1+t_k}-1)\right).$$
Therefore,
\begin{align*}
\E\left[\exp\left(-\frac{t_1m}{2A_1}-\sum\limits_{i=2}^{k}\frac{t_im}{2}\left(A_{i-1}+\frac{1}{A_i}\right)\right) \right]&\\
&\hspace{-5.5 cm}=\exp\left(-\sum\limits_{i=1}^{k-1}m\left(\sqrt{1+t_i}\sqrt{1+t_{i-1}}-1 \right)\right)\times\exp\left(-m(\sqrt{1+t_k}-1) \right)\times\prod\limits_{i=1}^k\frac{1}{\sqrt{1+t_i}}.
\end{align*}
This is exactly the Laplace Transform in \eqref{transfolaplaceinfini}.
\end{proof}
\begin{proof}[Proof of Proposition \ref{prop:scalinglimit}]
Let $m\in\N^*$. First, let us use the construction of the $\beta$-field given in Lemma \ref{lem:easyconstruction}. For every $i\in\N^*$, we write $A_i^{(m)}=A_i$ for sake of convenience.
Let $n\in\N^*$. First, we have to compute $\left(\hat{\bf{G}}_n(i,j)\right)_{i,j\in \llbracket 1,n \rrbracket}$. This computation requires two lemmas. Here is the first one:
\begin{lem}\label{lem:quotient}For every $i\in\llbracket 1,n-1\rrbracket$, for every $j\in\llbracket i+1, n\rrbracket$,
$$\frac{\hat{\bf{G}}_n(i,j)}{\hat{\bf{G}}_n(i+1,j)}=A_i.$$
\end{lem}
\begin{proof}
Our proof is based on an induction with respect to $i$: for every $i\in\llbracket1,n-1\rrbracket$, we will prove $$\mathcal{P}^{(i)}="\forall j\in\llbracket i+1, n\rrbracket,\frac{\hat{\bf{G}}_n(i,j)}{\hat{\bf{G}}_n(i+1,j)}=A_i".$$
Now, let us prove $\mathcal{P}^{(1)}$. We know that ${\bf{H}}_n\hat{\bf{G}}_n=I_n$. In particular, it holds that for every $j\in\llbracket 2,n\rrbracket$,
$$2\beta_1\hat{\bf{G}}_n(1,j)=m\hat{\bf{G}}_n(2,j).$$
Moreover, we know that $2\beta_1=\frac{m}{A_1}$.
Therefore, $$\frac{\hat{\bf{G}}_n(1,j)}{\hat{\bf{G}}_n(2,j)}=A_1.$$
Therefore, $\mathcal{P}^{(1)}$ is true. Now, let us assume that $\mathcal{P}^{(i-1)}$ is true for some $i\in\llbracket 2,n-1 \rrbracket$. Let us prove $\mathcal{P}^{(i)}$.
Again, we use the fact that ${\bf{H}}_n\hat{\bf{G}}_n=I_n$ which implies that for every $j\in \llbracket i+1, n\rrbracket$,
$$2\beta_i\hat{\bf{G}}_n(i,j)=m\hat{\bf{G}}_n(i-1,j)+m\hat{\bf{G}}(i+1,j).$$
Consequently,
\begin{align}\label{eq:inversion1}
\left(\frac{1}{A_i}+A_{i-1}\right)=\frac{\hat{\bf{G}}_n(i-1,j)}{\hat{\bf{G}}_n(i,j)}+\frac{\hat{\bf{G}}(i+1,j)}{\hat{\bf{G}}_n(i,j)}.
\end{align}
By $\mathcal{P}^{(i-1)}$, it holds that
$$\frac{\hat{\bf{G}}_n(i-1,j)}{\hat{\bf{G}}_n(i,j)}=A_{i-1}.$$
Combining it with \eqref{eq:inversion1} implies that for every $j\in \llbracket i+1, n\rrbracket$,
$$\frac{\hat{\bf{G}}_n(i,j)}{\hat{\bf{G}}_n(i+1,j)}=A_i $$
which means that $\mathcal{P}^{(i)}$ is true.
\end{proof}
Now, we need some notations. For every $i\in\llbracket 1,n\rrbracket$ and for every $j\in\llbracket i,n\rrbracket$, we define
$$c^{(i)}(j,j+1)=m A_j^{-1}\prod\limits_{i\leq k\leq j-1}A_k^{-2}.$$
Moreover, we define $\mathcal{R}^{(i)}(i\longleftrightarrow n+1)$ which is the effective resistance between $i$ and $n+1$ in the graph $\llbracket i,n+1\rrbracket$ with conductances $c^{(i)}$. Let us state the second lemma in order to compute $\hat{\bf{G}}_n$.
\begin{lem}\label{lem:calculdiag}For every $i\in\llbracket 1,n\rrbracket$,
$$\hat{\bf{G}}_n(i,i)=\mathcal{R}^{(i)}(i\longleftrightarrow n+1).$$
In particular, for every $i\in\llbracket1,n\rrbracket$,
$$\hat{\bf{G}}_n(i,i)=\frac{1}{m}\left(A_i+\sum\limits_{k=i}^{n-1}\prod\limits_{r=i}^kA_r^2\times A_{k+1} \right).$$
\end{lem}
\begin{proof}
For every $j\in\llbracket i+1,n\llbracket$, we define 
$$h(j)=\frac{\hat{\bf{G}}_n(i,j)}{\hat{\bf{G}}_n(i,i)}\prod\limits_{i\leq k\leq j-1}A_k.$$
Furthermore, we define $h(i)=1$ and $h(n+1)=0$. Now, we aim to prove that $h$ is harmonic on $\llbracket i,n+1\rrbracket$ with conductances $c^{(i)}$. Let $j\in\llbracket i+1,n\rrbracket$. Remark that
\begin{align*}
c^{(i)}(j-1,j)h(j-1)+c^{(i)}(j,j+1)h(j+1)&=\frac{m\hat{\bf{G}}_n(i,j-1)\prod\limits_{i\leq k\leq j-2} A_k}{\hat{\bf{G}}_n(i,i)A_{j-1}\prod\limits_{i\leq k\leq j-2}A_k^2}+\frac{m\hat{\bf{G}}_n(i,j+1)\prod\limits_{i\leq k\leq j} A_k}{\hat{\bf{G}}_n(i,i)A_{j}\prod\limits_{i\leq k\leq j-1}A_k^2}\\
&=\frac{m\hat{\bf{G}}_n(i,j-1)+m\hat{\bf{G}}_n(i,j+1)}{\hat{\bf{G}}_n(i,i)\prod\limits_{i\leq k\leq j-1}A_k}\\
&=\frac{\hat{\bf{G}}_n(i,j)}{\hat{\bf{G}}_n(i,i)}\frac{2\beta_j}{\prod\limits_{i\leq k\leq j-1}A_k}.
\end{align*}
Moreover, we know that $2\beta_j=\frac{m}{A_j}+mA_{j-1}$. Therefore
\begin{align*}
c^{(i)}(j-1,j)h(j-1)+c^{(i)}(j,j+1)h(j+1)&=\frac{\hat{\bf{G}}_n(i,j)}{\hat{\bf{G}}_n(i,i)}\prod\limits_{i\leq k\leq j-1} A_k\times\left(c^{(i)}(j-1,j)+c^{(i)}(j,j+1)\right)\\
&=h(j)\times\left(c^{(i)}(j-1,j)+c^{(i)}(j,j+1)\right).
\end{align*}
Therefore, $h$ is harmonic. Thus, by identity (2.3) in \cite{LP}, we get
$$\frac{1}{\mathcal{R}^{(i)}(i\longleftrightarrow n+1)}=c^{(i)}(i,i+1)\left(1-h(i+1)\right).$$
Consequently, we have
\begin{align}
\frac{1}{\mathcal{R}^{(i)}(i\longleftrightarrow n+1)}=\frac{m}{A_i}\left(1-\frac{A_i\hat{\bf{G}}_n(i,i+1)}{\hat{\bf{G}}_n(i,i)}\right)\label{eq:leminverseG1}
\end{align}
Recall that ${\bf{H}}_n\hat{\bf{G}}_n=I_n$. Therefore,
$$2\beta_i\hat{\bf{G}}_n(i,i)=1+m\hat{\bf{G}}_n(i-1,i)+m\hat{\bf{G}}_n(i+1,i).$$
It yields
$$\frac{1}{A_i}+A_{i-1}=\frac{1}{m\hat{\bf{G}}_n(i,i)}+\frac{\hat{\bf{G}}_n(i-1,i)}{\hat{\bf{G}}_n(i,i)}+\frac{\hat{\bf{G}}_n(i+1,i)}{\hat{\bf{G}}_n(i,i)}.$$
However, by Lemma \ref{lem:quotient},
$$\frac{\hat{\bf{G}}_n(i-1,i)}{\hat{\bf{G}}_n(i,i)}=A_{i-1}$$ which implies that
$$\frac{1}{A_i}-\frac{\hat{\bf{G}}_n(i+1,i)}{\hat{\bf{G}}_n(i,i)}=\frac{1}{m\hat{\bf{G}}_n(i,i)}.$$
Combining this with \eqref{eq:leminverseG1} yields
$$\frac{1}{\mathcal{R}^{(i)}(i\longleftrightarrow n+1)}=\frac{1}{\hat{\bf{G}}_n(i,i)}$$
which concludes the proof of the first part of Lemma \ref{lem:calculdiag}. In order to prove the second part of Lemma \ref{lem:calculdiag}, observe that
\begin{align*}
\hat{\bf{G}}_n(i,i)&=\mathcal{R}^{(i)}(i\longleftrightarrow n+1)\\
&=\sum\limits_{k=i}^n\frac{1}{c^{(i)}(k,k+1)}=\frac{1}{m}\left(A_i+\sum\limits_{k=i}^{n-1}\prod\limits_{r=i}^kA_r^2\times A_{k+1} \right).
\end{align*}
\end{proof}
Let $i,j\in\llbracket 1,n\rrbracket$ such that $j<i$. By the Lemmas \ref{lem:quotient} and \ref{lem:calculdiag},
\begin{align}\label{calculhorrible}
\hat{\bf{G}}_n(j,i)=\prod\limits_{k=j}^{i-1} \frac{\hat{\bf{G}}_n(k,i)}{\hat{\bf{G}}_n(k+1,i)}\times \hat{\bf{G}}_n(i,i)=\prod\limits_{k=j}^{i-1}A_k\times \frac{1}{m}\left(A_i+\sum\limits_{k=i}^{n-1}\prod\limits_{r=i}^kA_r^2\times A_{k+1} \right).
\end{align}
In particular,
\begin{align}
\psi^{(n,m)}_{\beta}=m\hat{\bf{G}}_n(1,n)=\prod\limits_{i=1}^nA_i\text{ and }\hat{{\bf{G}}}_n(1,1)=\frac{1}{m}\left(A_1+\sum\limits_{k=1}^{n-1}\prod\limits_{r=1}^kA_r^2\times A_{k+1}\right).\label{labmy9}
\end{align}
Therefore,
\begin{align}
Z_{\beta}^{(n,m)}=\frac{\frac{1}{m}\left(A_1+\sum\limits_{k=1}^{n-1}\prod\limits_{r=1}^kA_r^2\times A_{k+1}\right)}{\prod\limits_{i=1}^nA_i}.\label{labmy10}
\end{align}
Recall that $(\tilde{\psi}^{(m)}(t))_{t\geq 0}$ and $(\tilde{Z}^{(m)}(t))_{t\geq 0}$ are respectively the continuous linear interpolations of $(\psi_{\beta}^{(\lfloor mt\rfloor,m)})_{t\geq 0}$ and $(Z_{\beta}^{(\lfloor mt\rfloor,m)})_{t\geq 0}$. Therefore, combining \eqref{labmy9}, \eqref{labmy10} and Lemma \ref{lem:supremum}, for every $T>0$, for every $m\in\N^*$, it holds that
\begin{align}
\left(\tilde{\psi}^{(m)}(t),\tilde{Z}^{(m)}(t) \right)_{t\in[0,T]}=\left(Y_t^{(m)},\frac{\displaystyle\int_0^t(Y_s^{(m)})^2ds}{Y_t^{(m)}}\right)_{t\in[0,T]}+o_{m,T}^{\P}(1)\label{labmy11}
\end{align}
where $o_{m,T}^{\P}(1)$ is a random function whose supremum on $[0,T]$ goes to $0$ in probability and $Y^{(m)}$ is defined in Lemma \ref{lem:convergenceprocessus}. Thus, we can use Lemma \ref{lem:convergenceprocessus} in \eqref{labmy11} in order to conclude the proof of Proposition \ref{prop:scalinglimit}.
\end{proof}
Now, we are ready to give a new proof of Theorem \ref{thm:matsumotoyor} thanks to the discrete version of the Matsumoto Yor properties given by Proposition \ref{prop:discretematyor}.
\begin{proof}[New proof of Theorem \ref{thm:matsumotoyor}]
First, let us prove point (iii) of Theorem \ref{thm:matsumotoyor}. Let $t>0$. Let $k\in\N^*$ and let $t_1<t_2<\cdots<t_k\leq t$ be $k$ positive real numbers which are smaller than $t$. Let $m\in\N^*$ and let $F$ be a bounded continuous function from $\R^{k+1}$ into $\R$. By (iii) in Proposition \ref{prop:discretematyor}, it holds that,
\begin{align}
\E\left[F\left(\psi^{(\lfloor mt\rfloor,m)}_{\beta},Z^{(m)}_{\lfloor mt_1\rfloor},\cdots, Z^{(\lfloor mt\rfloor,m)}_{\beta }\right) \right]\nonumber\\
&\hspace{-6.3cm}=\E\left[\int_0^{+\infty}F\left(x,Z^{(\lfloor mt_1\rfloor,m)}_{\beta},\cdots, Z^{(\lfloor mt_k\rfloor ,m)}_{\beta}\right) \frac{1}{\sqrt{2\pi Z_{\beta}^{(\lfloor mt\rfloor,m)}x^3}}\exp\left({\displaystyle-\frac{1}{Z_{\beta}^{(\lfloor mt\rfloor,m)}}\frac{(x-1)^2}{2x}}\right)dx\right].\label{labmy12}
\end{align}
Thanks to Proposition \ref{prop:scalinglimit}, one can make $m$ go to infinity in \eqref{labmy12} which yields
\begin{align}
\E\left[F\left(e_t,Z_{t_1},\cdots, Z_{t_k}\right) \right]=\E\left[\int_0^{+\infty}F\left(x,Z_{t_1},\cdots, Z_{t_k}\right) \frac{1}{\sqrt{2\pi Z_t x^3}}\exp\left(\displaystyle-\frac{1}{Z_t}\frac{(x-1)^2}{2x}\right)dx\right].\nonumber
\end{align}
It proves (iii) in Theorem \ref{thm:matsumotoyor}. Moreover (i) in Theorem \ref{thm:matsumotoyor} is a direct consequence of (iii). Now, let us prove (ii). Let $\varepsilon>0$. Let $T>\varepsilon$. Let $t\in[\varepsilon,T]$. Let $m\in\N^*$. Let us define 
$$\mathcal{M}_t^{(\varepsilon,m)}=\ln(Z_{\beta}^{(\lfloor  mt\rfloor,m)})-\ln(Z_{\beta}^{(\lfloor  m\varepsilon\rfloor,m)})-\sum\limits_{k=\lfloor m\varepsilon \rfloor}^{\lfloor mt\rfloor-1}\E\left[\ln(Z_{\beta}^{(k+1,m)})-\ln(Z_{\beta}^{(k,m)})|\mathcal{Z}_{k,m}\right].$$
Remark that $\left(\mathcal{M}_t^{(\varepsilon,m)},t\geq\varepsilon\right)$ is a martingale. We would like to show that it converges toward a Brownian motion. First, by (ii) in Proposition \ref{prop:discretematyor}, one can get the following useful representation of $\mathcal{M}_t^{(\varepsilon,m)}$:
\begin{align}
\mathcal{M}_t^{(\varepsilon,m)}&=\ln(Z_{\beta}^{(\lfloor  mt\rfloor,m)})-\ln(Z_{\beta}^{(\lfloor  m\varepsilon\rfloor,m)})+\sum\limits_{k=\lfloor m\varepsilon \rfloor}^{\lfloor mt\rfloor-1}\ln\left(\frac{m}{m+1/Z_{\beta}^{(k,m)}}\right)\nonumber\\
&\hspace{0.4cm}+\sum\limits_{k=\lfloor m\varepsilon \rfloor}^{\lfloor mt\rfloor-1}\E\left[\ln(IG(1,m+1/Z_{\beta}^{(k,m)}))|\mathcal{Z}_{k,m}\right].\label{labmy13}
\end{align}
Besides, another useful representation of $\mathcal{M}_t^{(\varepsilon,m)}$ is the following one:
\begin{equation}
\mathcal{M}_t^{(\varepsilon,m)}=\sum\limits_{k=\lfloor m\varepsilon\rfloor}^{\lfloor mt\rfloor-1}\Delta_k^{(m)}\label{labmy14}
\end{equation}
where for every $k\in\N^*$, 
$$\Delta_k^{(m)}=\ln(Z_{\beta}^{(k+1,m)})-\ln(Z_{\beta}^{(k,m)})+
\ln\left(\frac{m}{m+1/Z_{\beta}^{(k,m)}} \right)+\E\left[\ln(IG(1,m+1/Z_{\beta}^{(k,m)}))|\mathcal{Z}_{k,m}\right].$$ By (ii) in Proposition \ref{prop:discretematyor}, for every $k\in\N^*$,
\begin{align}
\Delta_k^{(m)}=-\ln(IG_k(1,m+1/Z_{\beta}^{(k,m)}))
+\E\left[\ln(IG_k(1,m+1/Z_{\beta}^{(k,m)}))|\mathcal{Z}_{k,m}\right]\label{labmy15}
\end{align}
where conditionally on $(Z_{\beta}^{(k,m)})_{k\in\N^*}=(z_k)_{k\in \N^*}$, $(IG_k(1,m+1/z_k))_{k\in\N^*}$ is a sequence of independent random variables such that for every $k\in\N^*$, $IG_k(1,m+1/z_k)$ is an Inverse Gaussian random variable with parameters $(1,m+1/z_k)$.
Let $\delta>0$. By Proposition \ref{prop:scalinglimit},
\begin{align}
\underset{k\in\llbracket\lfloor m\varepsilon\rfloor,\lfloor mT\rfloor\rrbracket}\sup\hspace{0,1 cm}\frac{1}{Z_{\beta}^{(k,m)}}\xrightarrow[m\rightarrow+\infty]{law}\underset{s\in[\varepsilon,T]}\sup\hspace{0,1 cm}\frac{1}{Z_s}.\label{labmy15bis}
\end{align}
Remark that in the limit above, the restriction on $[\varepsilon,T]$ is crucial. Indeed, $\varepsilon=0$ would give an infinite limit.
Furthermore, by \eqref{labmy15bis}, there exists a positive constant $C_{\delta}$ such that for every $m\in\N^*$,
\begin{align}
\P\left(\underset{k\in\llbracket\lfloor m\varepsilon\rfloor,\lfloor mT\rfloor\rrbracket}\sup\hspace{0,1 cm}\frac{1}{Z_{\beta}^{(k,m)}}>C_{\delta} \right)\leq\delta.\label{labmy16}
\end{align}
In the sequel, we use the notation $A(\delta)$ to designate the event
$$\left\{\underset{k\in\llbracket\lfloor m\varepsilon\rfloor,\lfloor mT\rfloor\rrbracket}\sup\hspace{0,1 cm}\frac{1}{Z_{\beta}^{(k,m)}}\leq C_{\delta}\right\}.$$
Conditionally on $A(\delta)$, one can use the estimate of the expectation of an Inverse Gaussian random variable given by Lemma \ref{lem:estimees} to obtain that for every $k\in\llbracket \lfloor m\varepsilon,\rfloor,\lfloor mT\rfloor\rrbracket$,
\begin{align}
\Delta_k^{(m)}=-\ln(IG_k(1,m+1/Z_{\beta}^{(k,m)}))
+\E\left[\ln(IG_k(1,m))\right]+\frac{1}{m}o_{m,\delta}(1)\label{labmy17}
\end{align}
where $o_{m,\delta}(1)$ goes to zero when $m$ goes to infinity and does not depend on $k$. Now, let use the coupling of Lemma \ref{lem:couplage} in \eqref{labmy17}. Therefore, for every $k\in\llbracket \lfloor m\varepsilon,\rfloor,\lfloor mT\rfloor\rrbracket$, conditionally on $A(\delta)$, there exists random variables $(R_k^{(i)},i\in\{1,2,3,4\})$ and $Ber_k$ exactly as in Lemma \ref{lem:couplage} such that
\begin{align}
\Delta_k^{(m)}=-\ln(IG_k(1,m))
+\E\left[\ln(IG_k(1,m))\right]+\frac{1}{m}o_{m,\delta}(1)+J_k^{(m)}\label{labmy18}
\end{align}
where $(IG_k(1,m))_{k\in\N^*}$ is a sequence of independent Inverse Gaussian random variables with parameters $(1,m)$ and for every $k\in\N^*$
\begin{align}
|J_{k}^{(m)}|&\leq \left(\frac{1}{IG_k(1,m)}+\frac{1}{IG_k(1,m+1/Z_{\beta}^{(k,m)})}\right)\nonumber\\
&\hspace{0.4 cm}\times\left(\frac{C_{\delta}R_k^{(1)}}{m^{3/2}}+Ber_k\times\frac{C_{\delta}R_k^{(2)}+\sqrt{C_{\delta}}R_k^{(3)}+R_k^{(4)}}{\sqrt{m}}\right).\nonumber
\end{align}
Moreover, for every $i\in\{1,2,3,4\}$, $R_k^{(i)}$ is a positive random variable and conditionally on $\{R_k^{(i)},i\in\{1,2,3,4\}\}$, $Ber_k$ is a Bernoulli random variable whose parameter is smaller than $\frac{C_{\delta}R_k^{(1)}}{m^{3/2}}$. Moreover, there exists a positive constant $\kappa$ which does not depend on $m$, $k$ and $C_{\delta}$ such that for every $i\in\{1,2,3,4\}$, $\E\left[{R_k^{(i)}}^4\right]\leq \kappa$.
Consequently, conditionally on $A(\delta)$, using \eqref{labmy18} in \eqref{labmy14} yields for every $t\in[\varepsilon, T]$,
\begin{align}
\mathcal{M}_t^{(\varepsilon,m)}=\sum\limits_{k=\lfloor m\varepsilon\rfloor}^{\lfloor mt \rfloor}\left(-\ln(IG_k(1,m))
+\E\left[\ln(IG_k(1,m)\right]\right)+\sum\limits_{k=\lfloor m\varepsilon\rfloor}^{\lfloor mt \rfloor} J_k^{(m)}+(T-\varepsilon)o_{m,\delta}(1).\label{labmy19}
\end{align}
Now, let us show that the term $\sum\limits_{k=\lfloor m\varepsilon\rfloor}^{\lfloor mt \rfloor} J_k^{(m)}$ is negligible under $A(\delta)$. Remark that
\begin{align}
\E\left[\textbf{1}\{A(\delta)\}\sum\limits_{k=\lfloor m\varepsilon\rfloor}^{\lfloor mt \rfloor} J_k^{(m)}\right]&\leq\sqrt{2}\sum\limits_{k=\lfloor m\varepsilon\rfloor}^{\lfloor mt \rfloor}\E\left[\frac{1}{IG_k(1,m)^2}+\frac{1}{IG_k(1,m+1/Z_{\beta}^{(k,m)})^2}\right]^{1/2}\nonumber\\
&\hspace{0,8 cm}\times \E\left[\left(\frac{C_{\delta}R_k^{(1)}}{m^{3/2}}+Ber_k\times\frac{C_{\delta}R_k^{(2)}+\sqrt{C_{\delta}}R_k^{(3)}+R_k^{(4)}}{\sqrt{m}}\right)^2 \right]^{1/2}\nonumber\\
&\hspace{-2 cm}\leq 2\sqrt{7}\sum\limits_{k=\lfloor m\varepsilon\rfloor}^{\lfloor mt \rfloor}\E\left[\left(\frac{C_{\delta}R_k^{(1)}}{m^{3/2}}+Ber_k\times\frac{C_{\delta}R_k^{(2)}+\sqrt{C_{\delta}}R_k^{(3)}+R_k^{(4)}}{\sqrt{m}}\right)^2 \right]^{1/2}\label{labmy20}
\end{align}
where we used the fact that for every $K\geq 1$, $\E\left[IG(1,K)^{-2}\right]=(1+1/K)^2+
1/K+2/K^2\leq 7$.
Now, in \eqref{labmy20}, we can apply the estimates concerning the random variables $Ber_k$ and $(R_k^{(i)},i\in\{1,2,3,4\})$ given by Lemma \ref{lem:couplage}. Therefore, it holds that
\begin{align}
\E\left[\textbf{1}\{A(\delta)\}\sum\limits_{k=\lfloor m\varepsilon\rfloor}^{\lfloor mt \rfloor} J_k^{(m)}\right]\nonumber\\
&\hspace{-3cm}\leq 4\sqrt{7}\sum\limits_{k=\lfloor m\varepsilon\rfloor}^{\lfloor mt \rfloor} \E\left[\frac{C_{\delta}^2(R_k^{(1)})^2}{m^3}+Ber_k^2\left(\frac{C_{\delta}^2(R_k^{(2)})^2
+C_{\delta}(R_k^{(3)})^2+(R_k^{(4)})^2}{m} \right) \right]^{1/2}\nonumber\\
&\hspace{-3cm}\leq 4\sqrt{7}\sum\limits_{k=\lfloor m\varepsilon\rfloor}^{\lfloor mt \rfloor} \E\left[\frac{C_{\delta}^2(R_k^{(1)})^2}{m^3}+\frac{C_{\delta}R_k^{(1)}}{m^{3/2}}\left(\frac{C_{\delta}^2(R_k^{(2)})^2
+C_{\delta}(R_k^{(3)})^2+(R_k^{(4)})^2}{m} \right) \right]^{1/2}\nonumber\\
&\hspace{-3cm}\leq 4\sqrt{7}(T-\varepsilon)\left(\frac{C_{\delta}\kappa^{1/4}}{\sqrt{m}}+\frac{\kappa^{3/8}}{m^{1/4}}\left(C_{\delta}^{3/2}+C_{\delta}+C_{\delta}^{1/2} \right)\right).
\label{labmy21}
\end{align}
Therefore, we can use the estimate \eqref{labmy21} in \eqref{labmy19} which implies that, on the event $A(\delta)$, for every $t\in[\varepsilon, T]$,
\begin{align}
\mathcal{M}_t^{(\varepsilon,m)}=\sum\limits_{k=\lfloor m\varepsilon\rfloor}^{\lfloor mt \rfloor}\left(-\ln(IG_k(1,m))
+\E\left[\ln(IG_k(1,m)\right]\right)
+(T-\varepsilon)o_{m,\delta}(1)+o_{m,\delta,\varepsilon,T}^{\P}(1)\label{labmy22}
\end{align}
where $o_{m,\delta,\varepsilon,T}^{\P}(1)$ is a random function whose supremum on $[\varepsilon,T]$ goes to 0 in probability when $m$ goes to infinity. Moreover, exactly as in the proof of Lemma \ref{lem:convergenceprocessus}, we know that the continuous linear interpolation of 
$$\left(\sum\limits_{k=\lfloor m\varepsilon\rfloor}^{\lfloor m(t+\varepsilon) \rfloor}\left(-\ln(IG_k(1,m))
+\E\left[\ln(IG_k(1,m)\right]\right)\right)_{t\in[0,T-\varepsilon]}$$ converges in law toward some standard Brownian motion $\alpha^{(\varepsilon)}$ for the topology of uniform convergence.
Therefore, on the event $A(\delta)$, the continuous linear interpolation of $(\mathcal{M}_{\varepsilon+t}^{(\varepsilon,m)})_{t\in[0,T-\varepsilon]}$ converges in law toward $\alpha^{(\varepsilon)}$. Moreover, the probability of the event $A(\delta)$ can be made as close as we want from 1. Therefore, the linear interpolation of $(\mathcal{M}_{\varepsilon+t}^{(\varepsilon,m)})_{t\in[0,T-\varepsilon]}$ converges in law toward $\alpha^{(\varepsilon)}$. Moreover, recall from \eqref{labmy13} that for every $t\in[\varepsilon,T]$,
\begin{align}
\mathcal{M}_t^{(\varepsilon,m)}&=\ln(Z_{\beta}^{(\lfloor  mt\rfloor,m)})-\ln(Z_{\beta}^{(\lfloor  m\varepsilon\rfloor,m)})+\sum\limits_{k=\lfloor m\varepsilon \rfloor}^{\lfloor mt\rfloor-1}\ln\left(\frac{m}{m+1/Z_{\beta}^{(k,m)}}\right)\nonumber\\
&\hspace{0.4cm}+\sum\limits_{k=\lfloor m\varepsilon \rfloor}^{\lfloor mt\rfloor-1}\E\left[\ln(IG(1,m+1/Z_{\beta}^{(k,m)}))|\mathcal{Z}_{k,m}\right].\label{labmy23}
\end{align}
Moreover, as before, one can show that $\underset{k\in\llbracket\lfloor m\varepsilon\rfloor,\lfloor mT\rfloor\rrbracket}\sup\hspace{0,1 cm}\frac{1}{Z_{\beta}^{(k,m)}}$ is tight and by using a Taylor expansion of the logarithm, we get that
\begin{align}
\left(\sum\limits_{k=\lfloor m\varepsilon \rfloor}^{\lfloor mt\rfloor-1}\ln\left(\frac{m}{m+1/Z_{\beta}^{(k,m)}}\right) \right)_{t\in[\varepsilon,T]}&=\left(-\frac{1}{m}\sum\limits_{k=\lfloor m\varepsilon \rfloor}^{\lfloor mt\rfloor-1}\frac{1}{Z_{\beta}^{(k,m)}} \right)_{t\in[\varepsilon,T]}+o_{m,\varepsilon,T}^{\P}(1)\nonumber\\
&=\left(-\int_{\varepsilon}^t\frac{1}{\tilde{Z}^{(m)}(s)}ds \right)_{t\in[\varepsilon,T]}+o_{m,\varepsilon,T}^{\P}(1)\label{labmy24}
\end{align}
where $o_{m,\varepsilon,T}^{\P}(1)$ is a random function whose supremum on $[\varepsilon, T]$ goes to $0$ in probability as $m$ goes to infinity. In the last equality, we used the definition of $\tilde{Z}^{(m)}$ given in Proposition \ref{prop:scalinglimit}. Besides, using Lemma \ref{lem:estimees} and the tightness of $\underset{k\in\llbracket\lfloor m\varepsilon\rfloor,\lfloor mT\rfloor\rrbracket}\sup\hspace{0,1 cm}\frac{1}{Z_{\beta}^{(k,m)}}$ again, we obtain that
\begin{align}
\left(\sum\limits_{k=\lfloor m\varepsilon \rfloor}^{\lfloor mt\rfloor-1}\E\left[IG\left(1,m+1/Z_{\beta}^{(k,m)}\right)|\mathcal{Z}_{k,m}\right]\right)_{t\in[\varepsilon,T]}&=\left(-(t-\varepsilon)/2\right)_{t\in[\varepsilon,T]}+o_{m,\varepsilon,T}^{\P}(1).\label{labmy25}
\end{align}
Therefore, one can combine \eqref{labmy25}, \eqref{labmy24}, Proposition \ref{prop:scalinglimit} and the fact that the linear interpolation of $(\mathcal{M}_{\varepsilon+t}^{(\varepsilon,m)})_{t\in[0,T-\varepsilon]}$ converges in law toward $\alpha^{(\varepsilon)}$ to make $m$ go to infinity in equation \eqref{labmy23}. The convergence in law in equation \eqref{labmy23} holds on the compact set $[\varepsilon, T]$ for the topology of uniform convergence. Actually, we consider this convergence for the continuous linear interpolations of the random functions in \eqref{labmy23}. However, this does not change anything since the error which is associated with this approximation goes to 0 in probability uniformly on $[\varepsilon, T]$. It implies that, almost surely, for every $t\in[0,T-\varepsilon]$,
\begin{align}
\ln(Z_{t+\varepsilon})=\ln(Z_{\varepsilon})+
\alpha^{(\varepsilon)}_t+\int_{0}^t\frac{1}{Z_{s+\varepsilon}}ds+\frac{t}{2}.\label{labmy26}
\end{align}
Finally, using Ito formula in \eqref{labmy26}, almost surely, for every $t\in[0,T-\varepsilon]$,
\begin{align}
Z_{t+\varepsilon}=Z_{\varepsilon}+\int_0^tZ_{s+\varepsilon}d\alpha^{(\varepsilon)}_s+\int_0^tZ_sds+t.\label{labmy27}
\end{align} 
Moreover, the SDE \eqref{labmy27} is satisfied for every $\varepsilon>0$ and for every $T>\varepsilon$. This gives exactly (ii) in Theorem \ref{thm:matsumotoyor}.
\end{proof}
\section{Study of the discrete operator on the circle }
In this section, we will give a simple description of $H_{\beta}^{(\lambda,n)}$ and we will use it to compute $G_{\beta}^{(\lambda,n)}$ explicitely. It is the first step in order to construct the continuous-space operator $\mathcal{H}^{(\lambda)}$.
\subsection{A simple description of the random potential $\beta$}
Recall that an Inverse Gaussian distribution with parameters $(\mu,\lambda)$ has density
$$\textbf{1}\{x>0\}\frac{\sqrt{\lambda}}{\sqrt{2\pi x^3}}e^{-\frac{\lambda(x-\mu)^2}{2\mu^2x}} .$$
Let $\lambda>0$. Let $n\in\N$. Let us consider a family $(A_i^{(n)})_{i\in \mathcal{C}_{\lceil \lambda n\rceil}}$ of i.i.d Inverse Gaussian random variables with parameters $(1,n)$. Most of the time we will write $A_i^{(n)}=A_i$ for sake of convenience. For every $i\in \mathcal{C}_{\lceil \lambda n\rceil}$, let us define 
\begin{align}
\beta_i=\frac{n}{2}\left(A_{i+1}+\frac{1}{A_i}\right).\label{eqdefsimple}
\end{align}
In the definition above, if $i=\lceil\lambda n\rceil$, then $i+1$ is $-\lceil \lambda n\rceil$ in order to respect the circular structure of $\mathcal{C}_{\lceil \lambda n\rceil}$. Recall that $W_n^{(\lambda)}$ is a matrix on the discretized circle $\mathcal{C}_{\lceil\lambda  n\rceil}$ such that $(W_n^{(\lambda)})_{i,j}$ is 0 if $i$ and $j$ are not connected and is $n$ otherwise. Then, we have the following result
\begin{lem}\label{simpledescription}
The distribution of the random potential $(\beta_i)_{i\in \mathcal{C}_{\lceil \lambda n\rceil}}$ which is defined in \eqref{eqdefsimple} is $\nu_{\mathcal{C}_{\lceil \lambda n\rceil}}^{W_n^{(\lambda)}}$.
\end{lem}
\begin{proof}[Proof of Lemma \ref{simpledescription}]
The proof follows the same lines as the proof of Lemma \ref{lem:easyconstruction}.

\end{proof}

\subsection{Computation of $G_{\beta}^{(\lambda,n)}$}
In the sequel of this article, we will assume that the random potential $\beta$ with distribution $\nu_{\mathcal{C}_{\lceil \lambda n\rceil}}^{W_n^{(\lambda)}}$ is constructed with the random field $A$ introduced in the previous section.
A remarkable fact is that this is possible to compute explicitely the matrix $G_{\beta}^{(\lambda,n)}$ as a function of the field $A$. In order to make the computation simpler, let us make a small change of variables. For every $i\in\mathcal{C}_{\lceil \lambda n\rceil}$, let us introduce $u_i=\sqrt{A_iA_{i+1}}$ and $D$ the matrix on $\mathcal{C}_{\lceil \lambda n\rceil}^2$ whose diagonal coefficients are $\sqrt{A_{-\lceil \lambda_n\rceil}},\sqrt{A_{-\lceil \lambda_n\rceil+1}},\cdots$ and so on. Moreover, we introduce the matrix $R_u^{(\lambda,n)}$ which is defined by 
\begin{itemize}
\item $R_u^{(\lambda,n)}(i,i)=u_i^2+1$ for every $i\in\mathcal{C}_{\lceil \lambda n\rceil}$.
\item $R_u^{(\lambda,n)}(i,i+1)=-u_{i}$ for every $i\in\mathcal{C}_{\lceil \lambda n\rceil}$.
\item $R_u^{(\lambda,n)}(i,i-1)=-u_{i-1}$ for every $i\in\mathcal{C}_{\lceil \lambda n\rceil}$.
\item $R_u^{(\lambda,n)}(i,j)=0$ elsewhere.
\end{itemize}
Remark that 
\begin{align}
H_{\beta}^{(\lambda,n)}=nD^{-1}R_u^{(\lambda,n)}D^{-1}.\label{matrixprod}
\end{align}
Therefore, this is enough to compute the inverse of $R_u^{(\lambda,n)}$. Now, we are going to describe the inverse of $R_u^{(\lambda,n)}$. However, some new notation is required. Recall that  $\mathcal{C}_{\lceil \lambda n\rceil}=\{-\lceil \lambda n \rceil,\cdots,\lceil \lambda n\rceil\}$. This means that the discrete circle $\mathcal{C}_{\lceil \lambda n\rceil}$ is oriented by the $+1$ increment. For $i,j\in\mathcal{C}_{\lceil \lambda n\rceil}$, when we write $\vec{\prod}_{k=i}^j$ or $\vec{\sum}_{k=i}^j$, we mean that $k$ is in the set $\{i,i+1,\cdots,j-1,j\}$. 
Then, the inverse of $R_u^{(\lambda,n)}$ is given by the following proposition.
\begin{prop}\label{inversemo}
Let $n\in\N^*$. Let $\lambda>0$. Let $i,j\in\mathcal{C}_{\lceil \lambda n\rceil}$.
If $j\notin\{i-1,i,i+1\}$, then
\begin{align}
(R_u^{(\lambda,n)})^{-1}(i,j)=\frac{\vec{\prod}_{k=i}^{j-1}u_k\times \left(1+\vec{\sum}_{k=j+1}^{i-1}\vec{\prod}_{l=k}^{i-1}u_l^2\right)+\vec{\prod}_{k=j}^{i-1}u_k\times \left(1+\vec{\sum}_{k=i+1}^{j-1}\vec{\prod}_{l=k}^{j-1}u_l^2\right)}{\left(\prod\limits_{k\in \mathcal{C}_{\lceil \lambda n\rceil}}\hspace{-0.3 cm}u_k -1\right)^2}.\label{formule1}
\end{align}
If  $j=i+1$, then
\begin{align}
(R_u^{(\lambda,n)})^{-1}(i,i+1)=\frac{u_i\times \left(1+\vec{\sum}_{k=i+2}^{i-1}\vec{\prod}_{l=k}^{i-1}u_l^2\right)+\vec{\prod}_{k=i+1}^{i-1}u_k}{\left(\prod\limits_{k\in \mathcal{C}_{\lceil \lambda n\rceil}}\hspace{-0.3 cm}u_k -1\right)^2}\label{formule2}
\end{align}
if $j=i-1$, then
\begin{align}
(R_u^{(\lambda,n)})^{-1}(i,i-1)=\frac{u_{i-1}\times \left(1+\vec{\sum}_{k=i+1}^{i-2}\vec{\prod}_{l=k}^{i-2}u_l^2\right)+\vec{\prod}_{k=i}^{i-2}u_k }{\left(\prod\limits_{k\in \mathcal{C}_{\lceil \lambda n\rceil}}\hspace{-0.3 cm}u_k -1\right)^2}
\end{align}
Moreover,
\begin{align}
(R_u^{(\lambda,n)})^{-1}(i,i)=\frac{1+\vec{\sum}_{k=i+1}^{i-1}\vec{\prod}_{l=k}^{i-1}u_l^2}{\left(\prod\limits_{k\in \mathcal{C}_{\lceil \lambda n\rceil}}\hspace{-0.3 cm}u_k -1\right)^2}.\label{formule4}
\end{align}
\end{prop}
\begin{rem}
If we compare the formula for $(R_u^{(\lambda,n)})^{-1}(i,j)$ given in Proposition \ref{inversemo} and the formula given in identiy \eqref{calculhorrible}, one remarks that $(R_u^{(\lambda,n)})^{-1}(i,j)$ looks like a combination of $\hat{\bf{G}}_n(i,j)$ into two directions. This is possible to find an explicit link between $(R_u^{(\lambda,n)})^{-1}$ and $\hat{\bf{G}}_n$ but this is quite involved and we prefer to show a more direct and brutal proof for Proposition \ref{inversemo}.
\end{rem}
\begin{proof}[Proof of Proposition \ref{inversemo}]
Let $L_u^{(\lambda,n)}$ be a matrix with the coefficients given by Proposition \ref{inversemo}. We have to check that $R_u^{(\lambda,n)}L_u^{(\lambda,n)}$ is the identity matrix. The computation is quite awful. That is why we will only show that for every $i\in\mathcal{C}_{\lceil \lambda n\rceil}$, $(R_u^{(\lambda,n)}L_u^{(\lambda,n)})(i,i)$ equals 1 and $(R_u^{(\lambda,n)}L_u^{(\lambda,n)})(i,i+1)$ equals 0. The other computations are not more difficult and we omit it in this paper for sake of convenience. 
Let $i\in\mathcal{C}_{\lceil \lambda n\rceil}$.
Then, by definition of $R_u^{(\lambda,n)}$,
\begin{align}
(R_u^{(\lambda,n)}L_u^{(\lambda,n)})(i,i)&=(1+u_i^2)L_u^{(\lambda,n)}(i,i)-u_{i-1}L_u^{(\lambda,n)}(i,i-1)-u_iL_u^{(\lambda,n)}(i,i+1).\label{calcul1}
\end{align}
For sake of convenience, we multiply identity \eqref{calcul1} by $$\left(\prod\limits_{k\in \mathcal{C}_{\lceil \lambda n\rceil}}\hspace{-0.3 cm}u_k -1\right)^2.$$ Then, using the definition of the coefficients of $L_u^{(\lambda,n)}$, we get
$$\begin{array}{ll}
(R_u^{(\lambda,n)}L_u^{(\lambda,n)})(i,i)\times \left(\prod\limits_{k\in \mathcal{C}_{\lceil \lambda n\rceil}}\hspace{-0.3 cm}u_k -1\right)^2&\vspace{0.1cm}\\
=(1+u_i^2)\times\left( 1+\vec{\sum}_{k=i+1}^{i-1}\vec{\prod}_{l=k}^{i-1}u_l^2\right)-u_{i-1}\times \left(u_{i-1}\times \left(1+\vec{\sum}_{k=i+1}^{i-2}\vec{\prod}_{l=k}^{i-2}u_l^2\right)+\vec{\prod}_{k=i}^{i-2}u_k \right)&\nonumber\vspace{0.1cm}\\
\hspace{0.4 cm}-u_i\times\left( u_i\times \left(1+\vec{\sum}_{k=i+2}^{i-1}\vec{\prod}_{l=k}^{i-1}u_l^2\right)+\vec{\prod}_{k=i+1}^{i-1}u_k\right)&\nonumber\vspace{0.1cm}\\
=1+u_i^2+\vec{\sum}_{k=i+1}^{i-1}\vec{\prod}_{l=k}^{i-1}u_l^2+\vec{\sum}_{k=i+1}^{i-1}\vec{\prod}_{l=k}^{i}u_l^2-u_{i-1}^2-\vec{\sum}_{k=i+1}^{i-2}\vec{\prod}_{l=k}^{i-1}u_l^2-\prod\limits_{k\in \mathcal{C}_{\lceil \lambda n\rceil}}\hspace{-0.3 cm}u_k\nonumber\\
\hspace{0.4 cm}-u_i^2-\vec{\sum}_{k=i+2}^{i-1}\vec{\prod}_{l=k}^{i}u_l^2-\prod\limits_{k\in \mathcal{C}_{\lceil \lambda n\rceil}}\hspace{-0.3 cm}u_k\nonumber\vspace{0.1cm}\\
=1-2\prod\limits_{k\in \mathcal{C}_{\lceil \lambda n\rceil}}\hspace{-0.3 cm}u_k-u_{i-1}^2+\left(\vec{\sum}_{k=i+1}^{i-1}\vec{\prod}_{l=k}^{i-1}u_l^2- \vec{\sum}_{k=i+1}^{i-2}\vec{\prod}_{l=k}^{i-1}u_l^2\right)\nonumber\vspace{0.1cm}\\
\hspace{0.4 cm}+\left(\vec{\sum}_{k=i+1}^{i-1}\vec{\prod}_{l=k}^{i}u_l^2- \vec{\sum}_{k=i+2}^{i-1}\vec{\prod}_{l=k}^{i}u_l^2\right)\nonumber\vspace{0.1cm}\\
=1-2\prod\limits_{k\in \mathcal{C}_{\lceil \lambda n\rceil}}\hspace{-0.3 cm}u_k-u_{i-1}^2+u_{i-1}^2+\vec{\prod}_{l=i+1}^{i}u_l^2\nonumber\\
=\left(1-\prod\limits_{k\in \mathcal{C}_{\lceil \lambda n\rceil}}\hspace{-0.3 cm}u_k\right)^2.\nonumber
\end{array}$$
Thus, we get that $$(R_u^{(\lambda,n)}L_u^{(\lambda,n)})(i,i)=1.$$ Now, let us look at $(R_u^{(\lambda,n)}L_u^{(\lambda,n)})(i,i+1)$.
By definition of $R_{(\lambda,n)}$,
\begin{align}
(R_u^{(\lambda,n)}L_u^{(\lambda,n)})(i,i+1)&\nonumber\\
&\hspace{-2.5 cm}=(1+u_i^2)L_u^{(\lambda,n)}(i,i+1)-u_{i-1}L_u^{(\lambda,n)}(i-1,i+1)-u_iL_u^{(\lambda,n)}(i+1,i+1).\label{calcul2}
\end{align}
As previously, we multiply identity \eqref{calcul2} by $\left(\prod\limits_{k\in \mathcal{C}_{\lceil \lambda n\rceil}}\hspace{-0.3 cm}u_k -1\right)^2$. Then, using the definition of the coefficients of $L_u^{(\lambda,n)}$, we get
$$\begin{array}{ll}
(1+u_i^2)\times\left(u_i\times \left(1+\vec{\sum}_{k=i+2}^{i-1}\vec{\prod}_{l=k}^{i-1}u_l^2\right)+\vec{\prod}_{k=i+1}^{i-1}u_k\right)&\vspace{0.15cm}\\
\hspace{0.4 cm}-u_{i-1}\times\left( u_{i-1}u_i\left(1+\vec{\sum}_{k=i+2}^{i-2}\vec{\prod}_{l=k}^{i-2}u_l^2 \right)+\vec{\prod}_{k=i+1}^{i-2}u_k\times(1+u_i^2)\right)&\vspace{0.15cm}\\
\hspace{0.4 cm}
-u_i\times\left(1+\vec{\sum}_{k=i+2}^{i}\vec{\prod}_{l=k}^{i}u_l^2 \right)&\vspace{0.15cm}\\
=u_i+u_i^3+u_i\vec{\sum}_{k=i+2}^{i-1}\vec{\prod}_{l=k}^iu_l^2+u_i\vec{\sum}_{k=i+2}^{i-1}\vec{\prod}_{l=k}^{i-1}u_l^2+(1+u_i^2)\vec{\prod}_{k=i+1}^{i-1}u_k&\vspace{0.15cm}\\
\hspace{0,4 cm}-u_iu_{i-1}^2-u_i\vec{\sum}_{k=i+2}^{i-2}\vec{\prod}_{l=k}^{i-1}u_l^2-(1+u_i^2)\vec{\prod}_{k=i+1}^{i-1}u_k-u_i-u_i\vec{\sum}_{k=i+2}^i\vec{\prod}_{l=k}^iu_l^2&\vspace{0.15cm}\\
=u_i^3-u_iu_{i-1}^2+u_i\left(\vec{\sum}_{k=i+2}^{i-1}\vec{\prod}_{l=k}^{i-1}u_l^2- \vec{\sum}_{k=i+2}^{i-2}\vec{\prod}_{l=k}^{i-1}u_l^2\right)&\vspace{0.15cm}\\
\hspace{0.4 cm}+u_i\left(\vec{\sum}_{k=i+2}^{i-1}\vec{\prod}_{l=k}^iu_l^2- \vec{\sum}_{k=i+2}^i\vec{\prod}_{l=k}^iu_l^2\right)\vspace{0.15cm}\\
=u_i^3-u_iu_{i-1}^2+u_iu_{i-1}^2-u_i^3\vspace{0.15cm}\\
=0.
\end{array}$$
This concludes the proof.
\end{proof}
\section{Convergence of the discretized Green function}
Before proving Theorem \ref{thm:convergence}, we need to prove a first lemma.
For every $t\in[-\lambda,\lambda]$ and for every $n\in\N^*$, let us define a continuous random fonction $t\mapsto X_t^{(n)}$ such that if $j/n\leq t<(j+1)/n$,
$$X_t^{(n)}=\prod\limits_{i=-\lceil \lambda n\rceil}^{ j}u_i +n(t-j/n)\left(\prod\limits_{i=-\lceil \lambda n\rceil}^{ j+1}u_i-\prod\limits_{i=-\lceil \lambda n\rceil}^{ j}u_i \right)$$
where for every $i\in \mathcal{C}_{\lceil \lambda n\rceil}$, $u_i=\sqrt{A_iA_{i+1}}$ where $(A_i)_{i\in\mathcal{C}_{\lceil \lambda n\rceil}}$ is a family of independent Inverse Gaussian random variables with parameters $(1,n)$. Then we have the following lemma:
\begin{lem} \label{cvchap5}For the topology of uniform convergence on $[-\lambda,\lambda]$,

\begin{align}
(X^{(n)}_t)_{t\in[-\lambda,\lambda]}\xrightarrow[n\rightarrow+\infty]{law}(X_t)_{t\in[-\lambda,\lambda]}:=\left(e^{B_{t+\lambda}-(t+\lambda)/2}\right)_{t\in[-\lambda,\lambda]}.\nonumber
\end{align}
\end{lem}
\begin{proof}[Proof of lemma \ref{cvchap5}]
First, remark that for every $t\in [-\lambda,\lambda]$ and for every $n\in\N^*$,
\begin{align}
X_t^{(n)}=\tilde{Y}_t^{(n)}+E_t^{(n)}
\end{align}
where $t\mapsto \tilde{Y}_t^{(n)}$ is a continuous function such that  if $j/n\leq t<(j+1)/n)$,
$$\tilde{Y}_t^{(n)}=\prod\limits_{i=-\lceil \lambda n\rceil}^{ j}A_i +n(t-j/n)\left(\prod\limits_{i=-\lceil \lambda n\rceil}^{ j+1}A_i-\prod\limits_{i=-\lceil \lambda n\rceil}^{ j}A_i \right)$$ and $t\mapsto E_t^{(n)}$ is a random error function. By Lemma \ref{lem:supremum}, $\underset{t\in[-\lambda,\lambda]}\sup\hspace{0.1 cm} E_t^{(n)}$ goes to zero in probability as $n$ goes to infinity. Consequently, we only have to focus on $\tilde{Y}^{(n)}$. This random function converges in law toward a geometric Brownian motion according to Lemma \ref{lem:convergenceprocessus}.
\end{proof}
Now, we are ready to prove Theorem \ref{thm:convergence}.
\begin{proof}[Proof of Theorem \ref{thm:convergence}]
 Let us define a rescaled bilinear continuous interpolation $(\tilde{I}_u^{(\lambda,n)})$ of $(R_u^{(\lambda,n)})^{-1} $ exactly as in \eqref{interpo}.  Besides, recall that by Lemma \ref{cvchap5} 
\begin{align}
(X^{(n)}_t)_{t\in[-\lambda,\lambda]}\xrightarrow[n\rightarrow+\infty]{law}(X_t)_{t\in[-\lambda,\lambda]}\label{cvrap}
\end{align}
where the convergence holds for the topology of uniform convergence.
Now, the idea of the proof is to write $\frac{1}{n}\tilde{I}_u^{(\lambda,n)}$ as a function of $X^{(n)}$.
Remark that
\begin{itemize}
\item If $j>i$ (for the usual order in $\{-\lceil \lambda n \rceil,\cdots,\lceil \lambda n\rceil\}$) then $$\vec{\prod}_{l=j}^iu_l=X_{i/n}^{(n)}\frac{X_{\lceil \lambda n \rceil/n}^{(n)}}{X_{(j-1)/n}^{(n)}}.$$
\item If $j<i$, then
$$\vec{\prod}_{l=j}^iu_l=\frac{X_{i/n}^{(n)}}{X_{(j-1)/n}^{(n)}}.$$
\end{itemize}
Therefore, by Proposition \ref{inversemo}, it holds that:
\begin{itemize}
\item From \eqref{formule4}, if $t\in[-\lambda,\lambda]$,
\begin{align}
\frac{1}{n}\tilde{I}_u^{(\lambda,n)}(t,t)=\frac{\displaystyle\int_t^{\lambda}\left(\frac{X_{\lambda}^{(n)}X_t^{(n)}}{X_s^{(n)}}\right)^2ds+\int_{-\lambda}^t\left(\frac{X_t^{(n)}}{X_s^{(n)}}\right)^2ds}{(X_{\lambda}^{(n)}-1)^2} +o_n^{\P}(1)\label{identitn1}
\end{align}
where $o_n^{\P}(1)$ is a random variable which goes to zero in probability uniformly in $t$. (The uniformity in $t$ comes from Lemma \ref{lem:supremum}.)
\item From \eqref{formule1}, if $t,t'\in[-\lambda,\lambda]$ with $t<t'$, 
\begin{align}
\frac{1}{n}\tilde{I}_u^{(\lambda,n)}(t,t')=&\nonumber\\
&\hspace{-3.5 cm}\frac{\displaystyle \frac{X_{t'}^{(n)}}{X_t^{(n)}}\left(\int_{t'}^{\lambda} \left(\frac{X_{\lambda}^{(n)}X_t^{(n)}}{X_s^{(n)}} \right)^2ds+\int_{-\lambda}^t\left(\frac{X_t^{(n)}}{X_s^{(n)}}\right)^2ds\right)+\frac{X_{\lambda}^{(n)}X_t^{(n)}}{X_{t'}^{(n)}}\int_t^{t'}\left( \frac{X_{t'}^{(n)}}{X_s^{(n)}}\right)^2ds}{(X_{\lambda}^{(n)}-1)^2} +o_n^{\P}(1)\label{identitn2}
\end{align}
where $o_n^{\P}(1)$ goes to $0$ in probability uniformly in $t$ and $t'$ thanks to Lemma \ref{lem:supremum}.
\end{itemize}
Remark that \eqref{identitn1} is just a special case of \eqref{identitn2}. Therefore, in the rest of the proof, we will only focus on \eqref{identitn2}.
From \eqref{identitn2} and \eqref{cvrap} and the fact that the Lebesgue integral is a continuous functional for the topology of uniform convergence, we obtain that $\frac{1}{n}\tilde{I}_u^{(\lambda,n)}(\cdot,\cdot)$ converges in law for the topology of uniform convergence toward some symmetric random kernel $(\mathcal{R}^{(\lambda)}(t,t'))_{(t,t')\in(\mathcal{C}^{(\lambda)})^2}$ which is defined for $t,t'\in[-\lambda,\lambda]$ such that $t\leq t'$ by the formula
\begin{align}
\mathcal{R}^{(\lambda)}(t,t')&=\frac{\displaystyle \frac{X_{t'}}{X_t}\left(\int_{t'}^{\lambda} \left(\frac{X_{\lambda}X_t}{X_s} \right)^2ds+\int_{-\lambda}^t\left(\frac{X_t}{X_s}\right)^2ds\right)+\frac{X_{\lambda}X_t}{X_{t'}}\int_t^{t'}\left( \frac{X_{t'}}{X_s}\right)^2ds}{(X_{\lambda}-1)^2}\nonumber\\
&=\frac{X_{t'}X_t}{(X_{\lambda}-1)^2}\left(X_{\lambda}^2\int_{t'}^{\lambda}\frac{ds}{X_s^2}+X_{\lambda}\int_t^{t'}\frac{ds}{X_s^2}+\int_{-\lambda}^t\frac{ds}{X_s^2} \right).\label{identitn3}
\end{align}
Now, let us remark that $(X_t/X_0)_{t\in[-\lambda,\lambda]}$ is distributed as $(M_t)_{t\in[-\lambda,\lambda]}$ where $M$ was defined in subsection \ref{contimod}. Moreover, by \eqref{identitn3}, for every $t,t'\in[-\lambda,\lambda]$ such that $t\leq t'$
\begin{align*}
\mathcal{R}^{(\lambda)}(t,t')\\
&\hspace{-1 cm}=\frac{(X_{t'}/X_0)(X_t/X_0)}{(X_{\lambda}/X_0-1/X_0)^2}\left((X_{\lambda}/X_0)^2\int_{t'}^{\lambda}\frac{X_0^2}{X_s^2}ds+(X_{\lambda}/X_0)X_0^{-1}\int_t^{t'}\frac{X_0^2}{X_s^2}ds+X_0^{-2}\int_{-\lambda}^t\frac{X_0^2}{X_s^2}ds \right).
\end{align*}
Consequently, it holds that
\begin{align*}
\left(\mathcal{R}^{(\lambda)}(t,t')\right)_{-\lambda\leq t\leq t'\leq \lambda}\\
&\hspace{-2.6 cm}\overset{law}=\left(\frac{M_{t'}M_t}{(M_{\lambda}-M_{-\lambda})^2}\left(M_{\lambda}^2\int_{t'}^{\lambda}\frac{ds}{M_s^2} +M_{\lambda}M_{-\lambda}\int_t^{t'}\frac{ds}{M_s^2}+M_{-\lambda}^2\int_{-\lambda}^t\frac{ds}{M_s^2}\right)\right)_{-\lambda\leq t\leq t'\leq \lambda}.
\end{align*}
Therefore, $\mathcal{R}^{(\lambda)}$ has the same distribution as $\mathcal{G}^{(\lambda)}$ which is introduced in subsection \ref{contimod}. It implies that $\frac{1}{n}\tilde{I}_u^{(\lambda,n)}(\cdot,\cdot)$ converges in law toward $\mathcal{G}^{(\lambda)}$.
In order to conclude the proof, we only have to justify that $\frac{1}{n}\tilde{I}_u^{(\lambda,n)}(\cdot,\cdot)$ has the same limit as $\tilde{G}_{\beta}^{(\lambda,n)}(\cdot,\cdot)$. Recall that, by \eqref{matrixprod},
\begin{align}
G_{\beta}^{(\lambda,n)}=\frac{1}{n}D(R_u^{(\lambda,n)})^{-1}D\nonumber
\end{align}
where $D$ is a diagonal matrix whose entries are $(\sqrt{A_i})_{i\in\mathcal{C}_{\lceil \lambda n\rceil}}$ and $(A_i)_{i\in\mathcal{C}_{\lceil \lambda n\rceil}}$ are independent Inverse Gaussian random variables with parameters $(1,n)$. However by Lemma \ref{lem:supremum}, $\sup\hspace{0.1 cm}\{|\ln(A_i)|,i\in\mathcal{C}_{\lceil \lambda n\rceil}\}$ goes to $0$ in probability as $n$ goes to infinity. Therefore, $D$ goes to the identity matrix in probability  uniformly in its coefficients as $n$ goes to infinity. Consequently, the limits in law of $n^{-1}\tilde{I}_u^{(\lambda,n)}$ and $\tilde{G}_{\beta}^{(\lambda,n)}$ are the same, that is, $\mathcal{G}^{(\lambda)}$. This concludes the proof.
\end{proof}
\section{Study of $\mathcal{G}^{(\lambda)}$}
\begin{proof}[Proof of Proposition \ref{formeg}]
Remark that for every $t<t'$, $\mathcal{G}^{(\lambda)}(t,t')$ can be divided into three parts as follows
\begin{align}
(M_{\lambda}-M_{-\lambda})^2\frac{\mathcal{G}^{(\lambda)}(t,t')}{M_tM_{t'}}=M_{\lambda}^2A(t,t')+M_{\lambda}M_{-\lambda}B(t,t')+M_{-\lambda}^2C(t,t')\label{decoupage}
\end{align}
with
$$A(t,t')=\int_{t'}^{\lambda}\frac{ds}{M_s^2},\hspace{0.2 cm}B(t,t')=\int_t^{t'}\frac{ds}{M_s^2} \text{ and }C(t,t')=\int_{-\lambda}^t\frac{ds}{M_s^2}.$$
Therefore, we have three symmetric kernels $A$, $B$ and $C$. Let $f\in L^2([-\lambda,\lambda])$.
We can observe that
\begin{align}
\int_{-\lambda}^{\lambda}\int_{-\lambda}^{\lambda}A(t,t')f(t)\bar{f}(t')dtdt'&=\int_{-\lambda}^{\lambda}\int_{-\lambda}^{t'}\int_{t'}^{\lambda}\frac{ds}{M_s^2}f(t)\bar{f}(t')dtdt'+\int_{-\lambda}^{\lambda}\int_{t'}^{\lambda}\int_t^{\lambda}\frac{ds}{M_s^2}f(t)\bar{f}(t')dtdt'.\nonumber
\end{align}
By Fubini's theorem, this implies that
\begin{align}
\int_{-\lambda}^{\lambda}\int_{-\lambda}^{\lambda}\hspace{-0.2 cm}A(t,t')f(t)\bar{f}(t')dtdt'&=\int_{-\lambda}^{\lambda}\frac{1}{M_s^2}\int_{-\lambda}^s\int_{-\lambda}^{t'}f(t)\bar{f}(t')dtdt'ds+\int_{-\lambda}^{\lambda}\frac{1}{M_s^2}\int_{-\lambda}^s\int_{t'}^sf(t)\bar{f}(t')dtdt'ds\nonumber\\
&=\int_{- \lambda}^{\lambda}\frac{1}{M_s^2}\left|\int_{-\lambda}^sf(t)dt\right|^2ds.\label{fubini1}
\end{align}
In the same way, one can show that,
\begin{align}
\int_{-\lambda}^{\lambda}\int_{-\lambda}^{\lambda}C(t,t')f(t)\bar{f}(t')dtdt'&=\int_{-\lambda}^{\lambda}\frac{1}{M_s^2}\left|\int_s^{\lambda}f(t)dt\right|^2ds.\label{fubini2}
\end{align}
By Fubini's theorem again, we get
\begin{align}
\int_{-\lambda}^{\lambda}\int_{-\lambda}^{\lambda}B(t,t')f(t)\bar{f}(t')dtdt'&=\int_{-\lambda}^{\lambda}\int_{-\lambda}^{t'}\int_t^{t'}\frac{ds}{M_s^2}f(t)\bar{f}(t')dtdt'+\int_{-\lambda}^{\lambda}\int_{t'}^{\lambda}\int_{t'}^{t}\frac{ds}{M_s^2}f(t)\bar{f}(t')dtdt'\nonumber\\
&=\int_{-\lambda}^{\lambda}\frac{1}{M_s^2}\int_{-\lambda}^sf(t)dt\int_s^{\lambda}\bar{f}(t)dtds+\int_{-\lambda}^{\lambda}\frac{1}{M_s^2}\int_{-\lambda}^s\bar{f}(t)dt\int_s^{\lambda}f(t)dtds.\label{fubini3}
\end{align}
These identities hold for every $f\in L^2([-\lambda,\lambda])$. Let $f\in L^2([-\lambda,\lambda])$.
Combining identities \eqref{fubini1}, \eqref{fubini2}, \eqref{fubini3} and \eqref{decoupage}, we get
$$
\begin{array}{ll}
\displaystyle \left(M_{\lambda}-M_{-\lambda}\right)^2\int_{-\lambda}^{\lambda}\int_{-\lambda}^{\lambda}\mathcal{G}^{(\lambda)}(t,t')f(t)\bar{f}(t')dtdt'&\\
=\displaystyle M_{\lambda}^2\int_{- \lambda}^{\lambda}\frac{1}{M_s^2}\left|\int_{-\lambda}^sf(t)M_tdt\right|^2ds+M_{\lambda}M_{-\lambda}\int_{-\lambda}^{\lambda}\frac{1}{M_s^2}\int_{-\lambda}^sf(t)M_tdt\int_s^{\lambda}\bar{f}(t)M_tdtds\\
\displaystyle \hspace{0.4 cm}+M_{\lambda}M_{-\lambda}\int_{-\lambda}^{\lambda}\frac{1}{M_s^2}\int_{-\lambda}^s\bar{f}(t)M_tdt\int_s^{\lambda}f(t)M_tdtds+M_{-\lambda}^2\int_{-\lambda}^{\lambda}\frac{1}{M_s^2}\left|\int_s^{\lambda}f(t)M_tdt\right|^2ds\\
\displaystyle=\int_{-\lambda}^{\lambda}\frac{1}{M_s^2}\left|M_{\lambda}\int_{-\lambda}^sf(t)M_tdt+M_{-\lambda}\int_s^{\lambda}f(t)M_tdt\right|^2ds.
\end{array}
$$
This proves that $\mathcal{G}^{\lambda}$ is non-negative. Now, let us check that it is positive. Let $f\in L^2([-\lambda,\lambda])$ such that
$$\int_{-\lambda}^{\lambda}\frac{1}{M_s^2}\left|M_{\lambda}\int_{-\lambda}^sf(t)M_tdt+M_{-\lambda}\int_s^{\lambda}f(t)M_tdt\right|^2ds=0.$$
Then, for almost every $s\in[-\lambda,\lambda]$,
$$M_{\lambda}\int_{-\lambda}^sf(t)M_tdt+M_{-\lambda}\int_s^{\lambda}f(t)M_tdt=0.$$
Thus, by Lebesgue differentiation theorem, we get
that for almost every $s\in[-\lambda,\lambda]$,
$$ M_{\lambda}f(s)M_s-M_{-\lambda} f(s)M_s=0.$$
Moreover $M_{\lambda}-M_{-\lambda}\neq0$ almost surely. Therefore, almost surely, $f$ is zero almost everywhere. This concludes the proof.

\end{proof}
\section{Proof of identities in law}
\begin{proof}[Proof of Proposition \ref{idenloipoint}]
Let $t\in\mathcal{C}^{(\lambda)}$.
By Theorem \ref{thm:convergence},
$$G_{\beta}^{(\lambda,n)}(\lceil tn\rceil,\lceil tn\rceil)\xrightarrow[n\longrightarrow+\infty]{law}\mathcal{G}^{(\lambda)}(t,t).$$
However, by Theorem 3 in \cite{SZT}, for every $n\in\N^*$, $G_{\beta}^{(\lambda,n)}(\lceil tn\rceil,\lceil tn\rceil)$ is distributed as $1/(2\gamma)$ where $\gamma$ is a Gamma distribution with parameters $(1/2,1)$. Therefore, we get that
$$\frac{M_t^2}{(M_{\lambda}-M_{-\lambda})^2}\left(M_{\lambda}^2\int_{t}^{\lambda}\frac{ds}{M_s^2} +M_{-\lambda}^2\int_{-\lambda}^t\frac{ds}{M_s^2}\right)=\mathcal{G}^{(\lambda)}(t,t)\overset{law}=\frac{1}{2\gamma}.$$
\end{proof}
\begin{proof}[Proof of Corollary \ref{identityaladufresne}]
\textbf{First proof}: Let us use Proposition \ref{idenloipoint} with $t=-\lambda$. This gives that
$$\frac{M_{-\lambda}^2M_{\lambda}^2}{(M_{\lambda}-M_{-\lambda})^2}\int_{-\lambda}^{\lambda}\frac{ds}{M_s^2}\overset{law}=\frac{1}{2\gamma}$$
where $\gamma$ is a Gamma distribution with parameters $(1/2,1)$.
However, we can rewrite the left-hand side as
\begin{align*}
\frac{M_{-\lambda}^2M_{\lambda}^2}{(M_{\lambda}-M_{-\lambda})^2}\int_{-\lambda}^{\lambda}\frac{ds}{M_s^2}&=\frac{1}{\left(\frac{M_{\lambda}}{M_{-\lambda}}-1\right)^2 }\int_{0}^{2\lambda}\frac{M_{\lambda}^2}{M_{s-\lambda}^2}ds\\
&=\frac{1}{\left(\frac{M_{\lambda}}{M_{-\lambda}}-1\right)^2 }\int_{0}^{2\lambda}\frac{M_{\lambda}^2}{M_{\lambda-s}^2}ds.
\end{align*}
However, recall that for every $s\in[-\lambda,\lambda]$, $M_s=e^{B_s-s/2}$ where $B$ is a Brownian motion such that $B(0)=0$. Consequently, we get that
\begin{align*}
\frac{M_{-\lambda}^2M_{\lambda}^2}{(M_{\lambda}-M_{-\lambda})^2}\int_{-\lambda}^{\lambda}\frac{ds}{M_s^2}&=\frac{1}{\left(e^{B_{\lambda}-B_{-\lambda}-\lambda}-1\right)^2 }\int_{0}^{2\lambda}e^{2(B_{\lambda}-B_{\lambda -s})-s}ds.
\end{align*}
Remark that $\tilde{B}:=(B_{\lambda}-B_{\lambda-s})_{s\geq 0}$ is a standard Brownian motion such that $\tilde{B}(0)=0$. This gives exactly the formula in Corollary \ref{identityaladufresne} with $\tilde{B}$ and $2\lambda$.
This first proof of Corollary \ref{identityaladufresne} uses directly the new tools developped in this paper. However, this is also possible to prove it thanks to the Matsumoto-Yor properties whose a new proof is given in this paper.\\
\textbf{Second proof:} Let $\lambda>0$. Let us consider a Brownian motion $\alpha$ on $\R_+$ such that $\alpha_0=0$. For every $t\geq 0$, we define $e_t=e^{\alpha_t-t/2}$, $T_t=\int_0^t e_s^2ds$ and $Z_t=T_t/e_t$.
By (iii) in Theorem \ref{thm:matsumotoyor}, the law of $e_t$ conditionally on $Z_t$ is an Inverse Gaussian distribution with parameters $(1,1/Z_t)$.
Therefore, conditionally on $Z_t=z$,
$$\frac{(e_t-1)^2}{T_t}=\frac{(e_t-1)^2}{z e_t}$$ is distributed as 
$$\frac{(X-1)^2}{zX}$$ where $X$ is an Inverse Gaussian distribution with parameters $(1,1/z)$. However, it is true generally that if $Y$ is an Inverse Gaussian distribution with parameters $(\mu,r)$, then
$$r\frac{(Y-\mu)^2}{\mu^2Y}\overset{law}=2\gamma
$$
where $\gamma$ is a Gamma distribution with parameters $(1/2,1)$. (This stems from a direct computation involving the density of an Inverse Gaussian distribution.) Consequently, $\frac{(e_t-1)^2}{T_t}$ is distributed like $2\gamma$. This is exacly what we wanted to prove.

\end{proof}
Now, let us prove functional identities in law.
\begin{proof}[Proof of Proposition \ref{idenloifonc}]
Let $f$ be a deterministic continuous non-negative function on $\mathcal{C}^{(\lambda)}$.
By Lemma 8.1 in \cite{TG}, for every $n\in\N^*$, for every $\eta\in\R_+^{\mathcal{C}_{\lceil \lambda n\rceil}}$,
$$\sum\limits_{i\in \mathcal{C}_{\lceil \lambda n\rceil}}\sum\limits_{j\in\mathcal{C}_{\lceil \lambda n\rceil}}G_{\beta}^{(\lambda,n)}(i,j)\eta_i\eta_j\overset{law}=\frac{\left(\sum\limits_{i\in\mathcal{C}_{\lceil \lambda n\rceil}}\eta_i\right)^2}{2\gamma}$$
where $\gamma$ is distributed like a gamma distribution with parameters $(1/2,1)$.
Now, let us apply this fact with $\eta_i=n^{-1}f(i/n)$. Then we obtain that
\begin{align*}
\left(\frac{1}{n}\sum\limits_{i\in\mathcal{C}_{\lceil \lambda n\rceil}}f(i/n)\right)^2\frac{1}{2\gamma}&\overset{law}=\frac{1}{n^2}\sum\limits_{i\in \mathcal{C}_{\lceil \lambda n\rceil}}\sum\limits_{j\in\mathcal{C}_{\lceil \lambda n\rceil}}G_{\beta}^{(\lambda,n)}(i,j)f(i/n)f(j/n).
\end{align*}
The left-hand side converges in law toward $\left(\int_{-\lambda}^{\lambda}f(x)dx\right)^2\frac{1}{2\gamma}$ because $f$ is assumed to be continuous. Now, let us focus on the right-hand side. By Lemma \ref{lem:supremum}, it holds that,
\begin{align}
\frac{1}{n^2}\sum\limits_{i\in \mathcal{C}_{\lceil \lambda n\rceil}}\sum\limits_{j\in\mathcal{C}_{\lceil \lambda n\rceil}}G_{\beta}^{(\lambda,n)}(i,j)f(i/n)f(j/n)&\nonumber\\
&\hspace{-6 cm}=\sum\limits_{i\in \mathcal{C}_{\lceil \lambda n\rceil}}\sum\limits_{j\in\mathcal{C}_{\lceil \lambda n\rceil}}\int_{i/n}^{(i+1)/n}\int_{j/n}^{(j+1)/n}\tilde{G}_{\beta}^{(\lambda,n)}(t,t')dtdt'f(i/n)f(j/n)+o_n^{\P}(1)\label{idenfonc1}
\end{align}
where $o_n^{\P}(1)$ goes to $0$ in probability.
Moreover, by \eqref{idenfonc1} and the continuity of $f$, we get that
\begin{align}
\frac{1}{n^2}\sum\limits_{i\in \mathcal{C}_{\lceil \lambda n\rceil}}\sum\limits_{j\in\mathcal{C}_{\lceil \lambda n\rceil}}G_{\beta}^{(\lambda,n)}(i,j)f(i/n)f(j/n)&=\int_{-\lambda}^{\lambda}\int_{-\lambda}^{\lambda}\tilde{G}_{\beta}^{(\lambda,n)}(t,t')f(t)f(t')dtdt'+o_n^{\P}(1). \label{idenfonc2}
\end{align}
We know that $\tilde{G}_{\beta}^{(\lambda,n)}$ converges in law toward $\mathcal{G}^{(\lambda)}$. Besides, if $W(\mathcal{C}^{(\lambda)}\times\mathcal{C}^{(\lambda)})$ is the normed vector space of continuous functions on $\mathcal{C}^{(\lambda)}\times\mathcal{C}^{(\lambda)}$ with the $L^{\infty}$-norm, then $$H\mapsto \int_{-\lambda}^{\lambda}\int_{-\lambda}^{\lambda}H(t,t')f(t)f(t')dtdt'$$
is a continuous function on $W(\mathcal{C}^{(\lambda)}\times\mathcal{C}^{(\lambda)})$. Together with \eqref{idenfonc2}, this implies that
$$\frac{1}{n^2}\sum\limits_{i\in \mathcal{C}_{\lceil \lambda n\rceil}}\sum\limits_{j\in\mathcal{C}_{\lceil \lambda n\rceil}}G_{\beta}^{(\lambda,n)}(i,j)f(i/n)f(j/n)\xrightarrow[n\longrightarrow+\infty]{law}\int_{-\lambda}^{\lambda}\int_{-\lambda}^{\lambda}\mathcal{G}^{(\lambda)}(t,t')f(t)f(t')dtdt'.$$
Moreover, by Proposition \ref{formeg},
\begin{align*}\int_{-\lambda}^{\lambda}\int_{-\lambda}^{\lambda} \mathcal{G}^{(\lambda)}(t,t')f(t)f(t')dxdy\\
&\hspace{-4 cm}=\frac{1}{(M_{\lambda}-M_{-\lambda})^2} \int_{-\lambda}^{\lambda}\frac{1}{M_u^2}\left(M_{-\lambda}\int_u^{\lambda}f(t)M_tdt+M_{\lambda}\int_{-\lambda}^uM_tf(t)dt \right)^2du.
\end{align*}
It concludes the proof.
\end{proof}

\section{Study of $\mathcal{H}^{(\lambda)}$}
\begin{proof}[Proof of Theorem \ref{thm:definitionofH}]
\textbf{Step 1:} First let us show that the range of $\mathcal{G}^{(\lambda)}$ is included in $\mathcal{D}\left(\mathcal{H}^{(\lambda)} \right)$. Let $f\in L^2([-\lambda,\lambda])$. For every $x\in\mathcal{C}^{(\lambda)}$, it holds that
\begin{align}
(M_{\lambda}-M_{-\lambda})^2\mathcal{G}^{(\lambda)}f(x)&=(M_{\lambda}-M_{-\lambda})^2\left(\int_{-\lambda}^{x}f(t)\mathcal{G}^{(\lambda)}(x,t)dt+\int_{x}^{\lambda}f(t)\mathcal{G}^{(\lambda)}(x,t)dt\right)\nonumber\\
&\hspace{-2 cm}=\int_{-\lambda}^xf(t)M_xM_t\left(M_{\lambda}^2\int_x^{\lambda}\frac{ds}{M_s^2}+M_{\lambda}M_{-\lambda}\int_t^x\frac{ds}{M_s^2}+M_{-\lambda}^2\int_{-\lambda}^t\frac{ds}{M_s^2} \right)dt\nonumber\\
&\hspace{-1.6 cm}+\int_x^{\lambda}f(t)M_xM_t\left( M_{\lambda}^2\int_t^{\lambda}\frac{ds}{M_s^2}+M_{\lambda}M_{-\lambda}\int_x^t\frac{ds}{M_s^2}+M_{-\lambda}^2\int_{-\lambda}^x\frac{ds}{M_s^2} \right)dt.\label{developpement1}
\end{align}
Therefore, $\mathcal{G}^{(\lambda)}f$ is continuous. In particular it is in $L^2([-\lambda,\lambda])$.
Besides, by looking at \eqref{developpement1}, one can remark that $\mathcal{G}^{(\lambda)}f(-\lambda)$ and $\mathcal{G}^{(\lambda)}f(\lambda)$ are both equal to
$$\frac{1}{(M_{\lambda}-M_{-\lambda})^2}\int_{-\lambda}^{\lambda}f(t)M_t\left(M_{\lambda}^2M_{-\lambda}\int_t^{\lambda}\frac{ds}{M_s^2}+M_{\lambda}M_{-\lambda}^2\int_{-\lambda}^t\frac{ds}{M_s^2} \right)dt.$$
Now, we have to look at the derivative of $\left(\frac{\mathcal{G}^{(\lambda)}}{M}\right)$. By differentiating \eqref{developpement1}, we get that for every $x\in\mathcal{C}^{(\lambda)}$
\begin{align}
(M_{\lambda}-M_{-\lambda})^2\left(\frac{\mathcal{G}^{(\lambda)}f}{M}\right)'(x)&\nonumber\\
&\hspace{-4 cm}=\frac{1}{M_x^2}\left( M_{\lambda}(M_{-\lambda}-M_{\lambda})\int_{-\lambda}^xf(t)M_tdt+M_{-\lambda}(M_{-\lambda}-M_{\lambda})\int_x^{\lambda}f(t)M_tdt\right).\label{developpement2}
\end{align}
By \eqref{developpement2}, It is clear that $\left(\frac{\mathcal{G}^{(\lambda)}f}{M}\right)'$ is continuous. In particular, this is in $L^2([-\lambda,\lambda])$. Moreover, by looking at \eqref{developpement2}, one can remark that $M_{-\lambda}\left(\frac{\mathcal{G}^{(\lambda}f}{M}\right)'(-\lambda)$ and $M_{\lambda}\left(\frac{\mathcal{G}^{(\lambda}f}{M}\right)'(\lambda)$ are both equal to 
$$\frac{1}{M_{-\lambda}-M_{\lambda}}\int_{-\lambda}^{\lambda}f(t)M_tdt. $$
Finally, one can differentiate also \eqref{developpement2}. This gives that for almost every $x\in\mathcal{C}^{(\lambda)}$,
\begin{align}
\left(M^2\left(\frac{\mathcal{G}^{(\lambda)}f}{M}\right)'\right)'(x)=-f(x)M_x.\label{developpement3}
\end{align}
Recall that $f\in L^2([-\lambda,\lambda])$. Together with \eqref{developpement3}, this implies that $\left(M^2\left(\frac{\mathcal{G}^{(\lambda)}f}{M}\right)'\right)'\in L^2([-\lambda,\lambda]).$ Therefore, $\mathcal{G}^{(\lambda)}f\in\mathcal{D}\left(\mathcal{H}^{(\lambda)} \right)$.

\textbf{Step 2:} Now, we have to show that $\mathcal{D}\left(\mathcal{H}^{(\lambda)} \right)$ is included in the range of $\mathcal{G}^{(\lambda)}$. Let $g\in\mathcal{D}\left(\mathcal{H}^{(\lambda)} \right)$. Let us define $$f=-\frac{1}{M}\left(M^2\left(\frac{g}{M}\right)'\right)'.$$
Let us show that $g=\mathcal{G}^{(\lambda)}f$. Exactly as in \eqref{developpement1}, it holds that for every $x\in\mathcal{C}^{(\lambda)}$,

\begin{align}
(M_{\lambda}-M_{-\lambda})^2\mathcal{G}^{(\lambda)}f(x)&=(M_{\lambda}-M_{-\lambda})^2\left(\int_{-\lambda}^{x}f(t)\mathcal{G}^{(\lambda)}(x,t)dt+\int_{x}^{\lambda}f(t)\mathcal{G}^{(\lambda)}(x,t)dt\right)\nonumber\\
&\hspace{-3.6 cm}=\int_{-\lambda}^x-\frac{1}{M_t}\left(M^2\left(\frac{g}{M}\right)'\right)'(t)M_xM_t\left(M_{\lambda}^2\int_x^{\lambda}\frac{ds}{M_s^2}+M_{\lambda}M_{-\lambda}\int_t^x\frac{ds}{M_s^2}+M_{-\lambda}^2\int_{-\lambda}^t\frac{ds}{M_s^2} \right)dt\nonumber\\
&\hspace{-3.2 cm}+\int_x^{\lambda}-\frac{1}{M_t}\left(M^2\left(\frac{g}{M}\right)'\right)'(t)M_xM_t\left( M_{\lambda}^2\int_t^{\lambda}\frac{ds}{M_s^2}+M_{\lambda}M_{-\lambda}\int_x^t\frac{ds}{M_s^2}+M_{-\lambda}^2\int_{-\lambda}^x\frac{ds}{M_s^2} \right)dt\nonumber\\
&\hspace{-3.6 cm}=-\int_{-\lambda}^x\left(M^2\left(\frac{g}{M}\right)'\right)'(t)M_x\left(M_{\lambda}^2\int_x^{\lambda}\frac{ds}{M_s^2}+M_{\lambda}M_{-\lambda}\int_t^x\frac{ds}{M_s^2}+M_{-\lambda}^2\int_{-\lambda}^t\frac{ds}{M_s^2} \right)dt\nonumber\\
&\hspace{-3.2 cm}-\int_x^{\lambda}\left(M^2\left(\frac{g}{M}\right)'\right)'(t)M_x\left( M_{\lambda}^2\int_t^{\lambda}\frac{ds}{M_s^2}+M_{\lambda}M_{-\lambda}\int_x^t\frac{ds}{M_s^2}+M_{-\lambda}^2\int_{-\lambda}^x\frac{ds}{M_s^2} \right)dt.\label{developpement4}
\end{align}
Thanks to integration by parts, on can check that for every $x\in\mathcal{C}^{(\lambda)}$,
\begin{align}
\frac{(M_{\lambda}-M_{-\lambda})^2}{M_x}\mathcal{G}^{(\lambda)}f(x)&=-M_{\lambda}^2\int_x^{\lambda}\frac{ds}{M_s^2}\left(M_x^2\left(\frac{g}{M}\right)'(x)-M_{-\lambda}^2\left(\frac{g}{M}\right)'(-\lambda) \right)\nonumber\\
&\hspace{0.5 cm}-M_{\lambda}M_{-\lambda}\int_{-\lambda}^x\left(\frac{g}{M}\right)'(t)dt+M_{\lambda}M_{-\lambda}^3\left( \frac{g}{M}\right)'(-\lambda)\int_{-\lambda}^x\frac{ds}{M_s^2}\nonumber\\
&\hspace{0.5 cm}+M_{-\lambda}^2\int_{-\lambda}^x\left(\frac{g}{M}\right)'(t)dt-M_{-\lambda}^2M_x^2\left(\frac{g}{M}\right)'(x)\int_{-\lambda}^x\frac{ds}{M_s^2}\nonumber\\
&\hspace{0.5 cm}-M_{\lambda}^2\int_x^{\lambda}\left(\frac{g}{M}\right)'(t)dt+M_{\lambda}^2M_x^2\left(\frac{g}{M}\right)'(x) \int_x^{\lambda}\frac{ds}{M_s^2}\nonumber\\
&\hspace{0.5 cm}+M_{\lambda}M_{-\lambda}\int_x^{\lambda}\left(\frac{g}{M}\right)'(t)dt-M_{\lambda}^3M_{-\lambda}\left(\frac{g}{M} \right)'(\lambda)\int_x^{\lambda}\frac{ds}{M_s^2}\nonumber\\
&\hspace{0.5 cm}-M_{-\lambda}^2\int_{-\lambda}^x\frac{ds}{M_s^2}\left(M_{\lambda}^2\left(\frac{g}{M}\right)'(\lambda)-M_x^2\left(\frac{g}{M}\right)'(x) \right).\label{developpement5}
\end{align}
This expression seems to be quite awful. However, recall from the definition of $\mathcal{D}\left(\mathcal{H}^{(\lambda)}\right)$ that $M_{-\lambda}\left(\frac{g}{M}\right)'(-\lambda)=M_{\lambda}\left(\frac{g}{M}\right)'(\lambda)$. Therefore, we can simplify many terms in \eqref{developpement5} and we obtain that for every $x\in\mathcal{C}^{(\lambda)}$,
\begin{align}
\frac{(M_{\lambda}-M_{-\lambda})^2}{M_x}\mathcal{G}^{(\lambda)}f(x)&=-M_{\lambda}M_{-\lambda}\int_{-\lambda}^x\left(\frac{g}{M}\right)'(t)dt+M_{-\lambda}^2\int_{-\lambda}^x\left(\frac{g}{M}\right)'(t)dt\nonumber\\
&\hspace{0.5 cm}-M_{\lambda}^2\int_x^{\lambda}\left(\frac{g}{M}\right)'(t)dt+M_{\lambda}M_{-\lambda}\int_x^{\lambda}\left(\frac{g}{M}\right)'(t)dt\nonumber\\
&=M_{\lambda}\left(g(\lambda)-g(-\lambda)\right)+M_{-\lambda}\left(g(\lambda)-g(-\lambda)\right)\nonumber\\
&\hspace{0.5 cm}+\frac{g(x)}{M_x}\left(-M_{\lambda}M_{-\lambda}+M_{-\lambda}^2-M_{\lambda}M_{-\lambda}+M_{\lambda}^2 \right)\nonumber\\
&=\frac{g(x)}{M_x}(M_{\lambda}-M_{-\lambda})^2\label{developpment6}
\end{align}
where in the last equality, we used the fact that $g(\lambda)=g(-\lambda)$. Finally, we proved that for every $x\in\mathcal{C}^{(\lambda)}$,
$$\mathcal{G}^{(\lambda)}f(x)=g(x).$$ Therefore, $g$ is in the range of $\mathcal{G}^{(\lambda)}$.
This concludes the proof of the fact that $\mathcal{D}\left(\mathcal{H}^{(\lambda)}\right)$ is exactly the image of $\mathcal{G}^{(\lambda)}$.

\textbf{Step 3:}
We know that $\mathcal{G}^{(\lambda)}$ is a surjection from $L^2([-\lambda,\lambda])$ onto $\mathcal{D}\left(\mathcal{H}^{(\lambda)} \right)$. Moreover, by Proposition \ref{formeg}, $\mathcal{G}^{(\lambda)}$ is positive definite. In particular it is injective. Therefore, $\mathcal{G}^{(\lambda)}$ is a one-to-one mapping from $L^2([-\lambda,\lambda])$ onto $\mathcal{D}\left(\mathcal{H}^{(\lambda)} \right)$. Let us denote its inverse by $\mathcal{H}^{(\lambda)}$. The computation in the two first steps implies directly that for every $g\in\mathcal{D}\left(\mathcal{H}^{(\lambda)}\right)$,
$$H^{(\lambda)}g=- \frac{1}{M}\left(M^2\left(\frac{g}{M}\right)'\right)'.$$
Moreover, remark that $\mathcal{G}^{(\lambda)}$ is self-adjoint because it is a bounded symmetric operator. Moreover, we have seen that it is injective. Thus, $\mathcal{H}^{(\lambda)}=\left(\mathcal{G}^{(\lambda)}\right)^{-1}$ is self-adjoint according to Corollary 2.5 in \cite{Tes}.
Furthermore, $\mathcal{H}^{(\lambda)}$ is positive definite because it is the inverse of $\mathcal{G}^{(\lambda)}$ which is itself positive definite by Proposition \ref{formeg}.
\end{proof}

Now, let us show that the spectrum of $\mathcal{H}^{(\lambda)}$ consists in a sequence of increasing positive eigenvalues.
\begin{proof}[Proof of Proposition \ref{prop:spectrum}]
Recall that the operator $\mathcal{G}^{(\lambda)}$ on $L^2([-\lambda,\lambda])$ comes from a continuous kernel on the compact set $\mathcal{C}^{(\lambda)}\times \mathcal{C}^{(\lambda)}$. Then, a classical proof using Arzelà-Ascoli Theorem shows that $\mathcal{G}^{(\lambda)}$ is a compact operator on $L^2([-\lambda,\lambda])$. (See the beginning of section VI.5 in \cite{reedsimon1}.) Moreover, we know that $\mathcal{G}^{(\lambda)}$ is  self-adjoint on $L^2([-\lambda,\lambda])$. Therefore, by the theorem of diagonalization of self-adjoint compact operators (see Theorem VI.16 in \cite{reedsimon1}), the set of non-zero spectral values of $\mathcal{G}^{(\lambda)}$ consists in a random sequence $(A_k(\lambda))_{k\geq 0}$ which represents the eigenvalues of $\mathcal{G}^{(\lambda)}$ which are counted with multiplicity. The sequence $(A_k(\lambda))_{k\geq 0}$ is bounded, decreasing and converges to 0 when $k$ goes to infinity. Moreover,
$$\sigma(\mathcal{G}^{(\lambda)})=\{0\}\cup \{(A_k(\lambda))_{k\geq 0}\}. $$
Indeed, $0$ is also a spectral value of $\mathcal{G}^{(\lambda)}$ but not an eigenvalue because we know by Theorem \ref{thm:definitionofH} that $\mathcal{G}^{(\lambda)}$ is injective but not surjective because its image is $\mathcal{D}\left(\mathcal{H}^{(\lambda)}\right)\subsetneqq L^2([-\lambda,\lambda])$. Now, remark that $E$ is a spectral value of $\mathcal{H}^{(\lambda)}$ if and only if $Id-E\mathcal{G}^{(\lambda)}$ is invertible. Therefore, 
$$\sigma\left(\mathcal{H}^{(\lambda)}\right)\cap\R^*=\{(1/A_k(\lambda))_{k\geq 0}\}:=\{(E_k(\lambda))_{k\geq 0}\}.$$
Moreover, 0 is not a spectral value of $\mathcal{H}^{(\lambda)}$ because it is a bijection from its domain onto $L^2([-\lambda,\lambda])$ by Theorem \ref{thm:definitionofH}.
\end{proof}
\section{Proof of asymptotic results on $\mathcal{H}^{(\lambda)}$}

Now, let us prove Proposition \ref{volumeinfini}.
\begin{proof}[Proof of Proposition \ref{volumeinfini}]
 By subsection \ref{contimod} for every $\lambda>0$ and for every $t,t'\in[-\lambda,\lambda]$ such that $t\geq t'$,
\begin{align}\mathcal{G}^{(\lambda)}(t,t')&=\frac{M_{t'}M_t}{(M_{\lambda}-M_{-\lambda})^2}\left(M_{\lambda}^2\int_{t'}^{\lambda}\frac{ds}{M_s^2} +M_{\lambda}M_{-\lambda}\int_t^{t'}\frac{ds}{M_s^2}+M_{-\lambda}^2\int_{-\lambda}^t\frac{ds}{M_s^2}\right)\nonumber\\
&=\frac{M_{t'}M_t}{(M_{\lambda}/M_{-\lambda}-1)^2}\left(\frac{M_{\lambda}^2}{M_{-\lambda}^2}\int_{t'}^{\lambda}\frac{ds}{M_s^2} +\frac{M_{\lambda}}{M_{-\lambda}}\int_t^{t'}\frac{ds}{M_s^2}+\int_{-\lambda}^t\frac{ds}{M_s^2}\right).\label{limite1}
\end{align}
We know that $M_{\lambda}/M_{-\lambda}$ goes to 0 when $\lambda$ goes to infinity. Therefore, we can replace the term $\frac{1}{(M_{\lambda}/M_{-\lambda}-1)^2}$ by 1 when $\lambda$ goes to infinity. Moreover, for every $\lambda>0$, for every $t'\in [-\lambda,\lambda]$, it holds almost surely that
\begin{align*}
\frac{M_{\lambda}^2}{M_{-\lambda}^2}\int_{t'}^{\lambda}\frac{ds}{M_s^2}&=e^{2B_{\lambda}-2B_{-\lambda}-2\lambda}\int_{t'}^{\lambda}e^{-2B_s+s}ds=O(e^{-\lambda})
\end{align*}
because $|B_t|\underset{t\rightarrow+\infty}o\hspace{-0.3 cm}( t^{3/4})$. In the same way, one can show that
\begin{align*}
\frac{M_{\lambda}}{M_{-\lambda}}\int_t^{t'}\frac{ds}{M_s^2}\xrightarrow[\lambda \longrightarrow+\infty]{a.s}0.
\end{align*}
Using this in \eqref{limite1} implies that for every $T>0$,
\begin{align}
\underset{-T\leq t\leq t'\leq T}\sup\hspace{0,1 cm}\left|\mathcal{G}^{(\lambda)}(t,t')-M_{t'}M_t\int_{-\infty}^{t}\frac{ds}{M_s^2}\right|\xrightarrow[\lambda\longrightarrow+\infty]{a.s}0.\label{limite2}
\end{align}
It concludes the proof.
\end{proof}
Thanks to Proposition \ref{volumeinfini}, we can prove a surprising identity in law.
\begin{proof}[Proof of Corollary \ref{MBG}]
By Proposition \ref{volumeinfini},
\begin{align}
\left(\frac{\mathcal{G}^{(\infty)}(0,t)}{\mathcal{G}^{(\infty)}(0,0) }\right)_{t\geq 0}=\left(M_t \right)_{t\geq 0}\label{MBG1}
\end{align}
and 
\begin{align}
\left(\frac{\mathcal{G}^{(\infty)}(0,-t)}{\mathcal{G}^{(\infty)}(0,0) }\right)_{t\geq 0}=\left(M_{-t}\frac{\displaystyle\int_{-\infty}^{-t}\frac{ds}{M_s^2}}{\displaystyle\int_{-\infty}^0\frac{ds}{M_s^2}} \right)_{t\geq 0}.\label{MBG2}
\end{align}
Besides,
\begin{align}
\left(\frac{\mathcal{G}^{(\infty)}(0,t)}{\mathcal{G}^{(\infty)}(0,0) }\right)_{t\geq 0}\overset{law}=\left(\frac{\mathcal{G}^{(\infty)}(0,-t)}{\mathcal{G}^{(\infty)}(0,0) }\right)_{t\geq 0}.\label{MBG3}
\end{align}
Indeed, the law of the discrete process $\mathcal{G}^{(\lambda,n)}$ is clearly symmetric and this property remains true when we take the limit. Recall that $M_t=e^{B_t-t/2}$ for every $t\in\R$. Therefore, combining \eqref{MBG1}, \eqref{MBG2} and \eqref{MBG3}, we get that
$$\left(e^{B_{-t}+t/2}\frac{\displaystyle\int_t^{+\infty}e^{-2B_{-s}-s}ds}{\displaystyle\int_0^{+\infty}e^{-2B_{-s}-s}ds}\right)_{t\geq 0}$$
is a geometric Brownian motion starting from 1. Moreover, $(-B_{-t})_{t\geq 0}$ is a standard Brownian motion starting from 0. Therefore, we proved Corollary \ref{MBG}.
\end{proof}

Now, let us prove the asymptotic behaviour of the density of states of $\mathcal{H}^{(\lambda)}$. Our proof is inspired from a paper of Fukushima and Nakao (see \cite{Funak}) and it is based on the study of a Sturm-Liouville equation.\\
\textbf{Strategy of the proof:}
The equation $\mathcal{H}^{(\lambda)}\varphi_E=E\varphi_E$ is equivalent to the Sturm-Liouville equation
$$\left(M^2\varphi_E'\right)'+EM^2\varphi_E=0. $$
Some classical results on Sturm-Liouville equations imply that $\varphi_E$ can be written as $\varphi_E=R_E\sin(\theta_E)$ where $R_E$ never vanishes. Therefore, $\varphi_E$ vanishes $k$ times in $[-\lambda,\lambda]$ if and only if $\theta_E(\lambda)\in[k\pi,(k+1)\pi)$. Moreover, Sturm-Liouville oscillation theorem states that $\varphi_{E_k(\lambda)}$ vanishes approximately $k$ times. Consequently, if $N_{\lambda}(E)=k$, then $E_k(\lambda)\leq E<E_{k+1}(\lambda)$ which implies that 
$$N_{\lambda}(E)=k\simeq\frac{\theta_{E_k(\lambda)}(\lambda)}{\pi}\simeq \frac{\theta_E(\lambda)}{\pi}.$$
Therefore, for every fixed $E>0$, we only have to study the asymptotic behaviour of $\theta_E(\lambda)$ when $\lambda$ goes to infinity. However, according to the theory of Sturm-Liouville equations, $\theta_E$ is solution of the following ODE with random coefficients:
$$\theta_E'=M_t^{-2}\cos(\theta_E)^2+EM_t^2\sin(\theta_E)^2.$$
Moreover a very surprising fact is that $\zeta_E:=-\frac{\cot(\theta_E)}{M^2}$ is a Markov process with explosions. Actually explosions of $\zeta_E$ occur precisely when $\theta_E$ is a multiple of $\pi$. Therefore, we only have to count the explosions of $\zeta_E$. However, by the Markov property, the explosion times of $\zeta_E$ are i.i.d random variables. As a consequence, the number of explosions of $\zeta_E$ before time $t$, that is the number of times where $\theta_E$ is a multiple of $\pi$ before $t$ is a renewal process which can be studied thanks to classical tools. However, one problem of the Sturm-Liouville equation is that $M$ is not differentiable. Consequently, we can not apply directly the Sturm-Liouville theory. However, we can apply it with a regularized version of $M$. Let us replace $M$ by $M^{(n)}$ which is a $C^2$ random function such that $M^{(n)}$ converges uniformly to $M$ almost surely. (For example, $M^{(n)}$ can be taken as a polynomial interpolation of $M$.) Let us prove the following lemma about the effects of this approximation on the spectrum.
\begin{lem}\label{approxspectrum}
Let $\lambda>0$. Let $(M^{(n)})_{n\in\N^*}$ be a sequence of $C^2$ functions which converges almost surely uniformly to $M$ on $[-\lambda,\lambda]$. 
Let $\mathcal{H}_n^{(\lambda)}$ be a random self-adjoint positive operator defined by
$$\mathcal{H}_n^{(\lambda)}g=-\frac{1}{M^{(n)}}\left((M^{(n)})^2\left(\frac{g}{M^{(n)}}\right)'\right)'$$
where $g$ is in the domain $\mathcal{D}\left(\mathcal{H}_n^{(\lambda)}\right)$ which is defined by
$$\begin{Bmatrix}
g\in L^2([-\lambda,\lambda]),&\left(\frac{g}{M^{(n)}}\right)'\in L^2([-\lambda,\lambda]),& \left((M^{(n)})^2\left(\frac{g}{M^{(n)}}\right)'\right)'\in L^2([-\lambda,\lambda]),\\
g(-\lambda)=g(\lambda),& M_{-\lambda}^{(n)}\left(\frac{g}{M^{(n)}}\right)'(-\lambda)=M_{\lambda}^{(n)}\left(\frac{g}{M^{(n)}}\right)'(\lambda)\end{Bmatrix}.$$
Then, the spectrum of $\mathcal{H}_n^{(\lambda)}$ is a random increasing positive sequence $(E_{n,k}(\lambda))_{k\geq 0}$ which diverges to infinity. Moreover, these spectral values of $\mathcal{H}_n^{(\lambda)}$ are eigenvalues which are counted with multiplicity. Furthermore, for every $k\in\N$,
$$ E_{n,k}(\lambda)\xrightarrow[n\rightarrow+\infty]{a.s}E_{k}(\lambda).$$
In particular, if $N_{n,\lambda}(E)$ is the number of eigenvalues of $\mathcal{H}_n^{(\lambda)}$ which are lower than $E$, then for every $E>0$
$$N_{n,\lambda}(E)\xrightarrow[n\rightarrow+\infty]{a.s}N_{\lambda}(E).$$
\end{lem}
\begin{proof}[Proof of Lemma \ref{approxspectrum}]
The fact that $\mathcal{H}_n^{(\lambda)}$ is positive and self-adjoint can be proved exactly as for $\mathcal{H}^{(\lambda)}$. Moreover, exactly as for $\mathcal{H}^{(\lambda)}$ in the proof of Proposition \ref{prop:spectrum}, the eigenvalues of $\mathcal{H}_n^{(\lambda)}$ are the inverse of the eigenvalues of $\mathcal{G}_n^{(\lambda)}$ where $\mathcal{G}_n^{(\lambda)}(t,t')$ is defined by $$\frac{M_{t'}^{(n)}M_t^{(n)}}{(M_{\lambda}^{(n)}-M_{-\lambda}^{(n)})^2}\left((M_{\lambda}^{(n)})^2\int_{t'}^{\lambda}\frac{ds}{(M_s^{(n)})^2} +M_{\lambda}^{(n)}M_{-\lambda}^{(n)}\int_t^{t'}\frac{ds}{(M_s^{(n)})^2}+(M_{-\lambda}^{(n)})^2\int_{-\lambda}^t\frac{ds}{(M_s^{(n)})^2}\right)$$ for every $t\leq t'\in \mathcal{C}^{(\lambda)}$.
By the expression above, this is clear that $(t,t')\mapsto\mathcal{G}_n^{(\lambda)}(t,t')$ converges uniformly toward $(t,t')\mapsto \mathcal{G}^{(\lambda)}(t,t')$. This implies that
$\mathcal{G}_n^{(\lambda)}f$ converges almost surely in $L^2([-\lambda,\lambda])$ toward $\mathcal{G}^{(\lambda)}f$ uniformly in $f\in L^2([-\lambda,\lambda])$ such that $||f||_2=1$. Furthermore, by the min-max principle,
$$\frac{1}{E_{n,k}(\lambda)}=\underset{V_k}\max\hspace{0.1 cm}\underset{f\in V_k, ||f||_2=1}\min\hspace{0.1 cm}\int_{-\lambda}^{\lambda}\left|\mathcal{G}_n^{(\lambda)}f(t)\right|^2dt$$
where $V_k$ is in the set of vector spaces of dimension $k$ in $L^2([-\lambda,\lambda])$.
This implies directly that $E_{n,k}(\lambda)$ converges toward $E_k( \lambda)$ when $n$ goes to infinity.

\end{proof}
\begin{proof}[Proof of Theorem \ref{DOS}]
\textbf {Step 1: Link with the Sturm-Liouville equation.} Let $\lambda>0$. Let $M^{(n)}$ be a sequence of $C^2$ random functions which converges almost surely uniformly toward $M$ and let $\mathcal{H}_n^{(\lambda)}$ be the operator which is associated with $M^{(n)}$. Let $E_{n,k}(\lambda)$ be some eigenvalue of $\mathcal{H}_n^{(\lambda)}$. There exists an eigenvector $\psi_{E_{n,k}(\lambda)}\in \mathcal{D}\left(\mathcal{H}_n^{(\lambda)} \right)$ such that 
\begin{align}
\left((M^{(n)})^2\left(\frac{\psi_{E_{n,k}(\lambda)}}{M^{(n)}}\right)'\right)'+E_{n,k}(\lambda)M^{(n)}\psi_{E_{n,k}(\lambda)}=0.\label{Dos1}
\end{align}
Then, let us define $\varphi_{n,k}=\psi_{E_{n,k}(\lambda)}/M^{(n)}$. We omit the dependence in $\lambda$ for sake of convenience. Then, $\varphi_{n,k}$ belongs to the set
$$\begin{Bmatrix}
g\in L^2([-\lambda,\lambda]),&g'\in L^2([-\lambda,\lambda]),& \left((M^{(n)})^2g'\right)'\in L^2([-\lambda,\lambda]),\\
M_{-\lambda}^{(n)}g(-\lambda)=M_{\lambda}^{(n)}g(\lambda),& M_{-\lambda}^{(n)}g'(-\lambda)=M_{\lambda}^{(n)}g'(\lambda)\end{Bmatrix}.$$
Moreover, $\varphi_{n,k}$ is solution of
\begin{align}
\left((M^{(n)})^2\varphi_{n,k}'\right)'+E_{n,k}(\lambda)(M^{(n)})^2\varphi_{n,k} =0.\label{Dos2}
\end{align}
Remark that the ODE in \eqref{Dos2} is a Sturm-Liouville equation. By Theorem 2.1 in Chapter 8 of \cite{codlev}, there exists an increasing sequence of real numbers $(\mu_{n,k}(\lambda))_{k\geq 0}$ which goes to infinity such that for every $k\in\N$, there exists a non trivial $C^2$ function $\xi_{n,k}$ such that
$$\left((M^{(n)})^2\xi_{n,k}'\right)'+ \mu_{n,k}(\lambda)(M^{(n)})^2\xi_{n,k}=0$$
and $\xi_{n,k}(-\lambda)=\xi_{n,k}(\lambda)=0$.
These numbers $(\mu_{n,k}(\lambda))_{k\geq 0}$ are called the Dirichlet eigenvalues of our Sturm Liouville equation. $(\xi_{n,k})_{k\in\N}$ are the Dirichlet eigenstates which are associated with $(\mu_{n,k}(\lambda))_{k\geq 0}$. For every $E>0$, let us define $\tilde{N}_{n,\lambda}(E)$ be the number of Dirichlet eigenvalues which are lower than $E$.
Moreover, by the Liouville transformation (see 4.3 in \cite{Everitt1982}) the eigenvalues of the Sturm-Liouville problem
$$(I)\left\{
\begin{array}{c}
((M^{(n)})^2\varphi')'+E(M^{(n)})^2\varphi=0\\
M_{-\lambda}^{(n)}\varphi(-\lambda)=M_{\lambda}^{(n)}\varphi(\lambda)\\
M_{-\lambda}^{(n)}\varphi'(-\lambda)=M_{\lambda}^{(n)}\varphi'(\lambda)

\end{array}
\right. $$
are the same as the eigenvalues of the Sturm-Liouville problem
$$(II)\left\{
\begin{array}{c}
\Phi''+Q\Phi+E\Phi=0\\
\Phi(-\lambda)=\Phi(\lambda)\\
\Phi'(-\lambda)=\Phi'(\lambda)+\left(M^{(n)}_{\lambda}\left(\frac{1}{M^{(n)}}\right)'_{\lambda}-M^{(n)}_{-\lambda}\left(\frac{1}{M^{(n)}}\right)'_{-\lambda}\right)\Phi(\lambda)

\end{array}
\right.$$
where $Q=\frac{1}{M^{(n)}}\left((M^{(n)})^2\left( \frac{1}{M^{(n)}}\right)'\right)'$.
Besides, by Theorem 1.3 in \cite{Plaksina}, the eigenvalues of $(II)$ are interlaced with the eigenvalues of the following Sturm-Liouville problem:
$$(III)\left\{
\begin{array}{c}
\Phi''+Q\Phi+E\Phi=0\\
\Phi(-\lambda)=\Phi(\lambda)=0
\end{array}
\right. .$$
Furthermore, by the Liouville tranformation again, the eigenvalues of $(III)$ are the same as the eigenvalues of the following Sturm-Liouville problem:
$$(IV)\left\{
\begin{array}{c}
((M^{(n)})^2\varphi')'+E(M^{(n)})^2\varphi=0\\
\varphi(-\lambda)=\varphi(\lambda)=0\\
\end{array}
\right. .$$
Therefore, the eigenvalues of $(I)$ and $(IV)$ are interlaced. In particular, almost surely, for every $E>0$
\begin{align}
|N_{n,\lambda}(E)-\tilde{N}_{n,\lambda}(E)|\leq 1
.\label{dirichlet}
\end{align}
Now, let $\mu_{n,k}(\lambda)$ be an eigenvalue of the Sturm Liouville equation with Dirichlet condition. Let $\xi_{n,k}$ be the associated eigenstate.
According to chapter 8 in \cite{codlev}, there exists two functions $R_{n,\mu_{n,k}(\lambda)}$ and $\theta_{n,\mu_{n,k}(\lambda)}$ such that
\begin{align}
\xi_{n,k}=R_{n,\mu_{n,k}(\lambda)}\sin\left(\theta_{n,\mu_{n,k}(\lambda)} \right)&\label{Dos3}
\end{align}
\begin{align}
R_{n,\mu_{n,k}(\lambda)}^2=((M^{(n)})^2\xi_{n,k}')^2+\xi_{n,k}^2&\label{Dos4}
\end{align}
\begin{align}
\theta_{n,\mu_{n,k}(\lambda)}'=\frac{1}{(M^{(n)})^2}\cos(\theta_{n,\mu_{n,k}(\lambda)})^2+\mu_{n,k}(\lambda)(M^{(n)})^2\sin(\theta_{n,\mu_{n,k}(\lambda)})^2&\label{Dos5}
\end{align}
where $\theta_{n,\mu_{n,k}(\lambda)}(-\lambda)=0$. Remark that for any $\mu>0$, the function $\theta_{n,\mu}$ depends also on $\lambda$ but we omit this dependence in the notation for sake of convenience.
By \eqref{Dos4}, $R_{n,\mu_{n,k}(\lambda)}$ can never be zero at any point. Otherwise, this would imply that there exists some point $t_0$ such that $\xi_{n,k}'(t_0)=\xi_{n,k}(t_0)=0$. By Cauchy-Lipschitz theorem, this would imply that $\xi_{n,k}$ is zero everywhere which is false because this is an eigenvector. Therefore, as $R_{n,\mu_{n,k}(\lambda)}$ never vanishes, \eqref{Dos3} implies that $\xi_{n,k}$ vanishes when $\theta_{n,\mu_{n,k}(\lambda)}$ is a multiple of $\pi$. Moreover, by Theorem 2.1 in chapter 8 of \cite{codlev}, the number of zeros of $\xi_{n,k}$ in $[-\lambda,\lambda]$ is always $k+2$.  Now, let $E\in\R_+^*$. There exists some $k\in\N$ such that $\tilde{N}_{n,\lambda}(E)=k$, that is, $\mu_{n,k-1}(\lambda)\leq E< \mu_{n,k}(\lambda)$ with the convention stating that $\mu_{n,-1}(\lambda)=0$.
We said that $\xi_{n,k-1}$ vanishes exactly $k+1$ times in $[-\lambda,\lambda]$. Therefore, $$\theta_{n,\mu_{n,k-1}(\lambda)}(\lambda)= k\pi.$$
Consequently,
$$\tilde{N}_{n,\lambda}(E)= \frac{\theta_{n,\mu_{n,k-1}(\lambda)}(\lambda)}{\pi}\leq \frac{\theta_{n,E}(\lambda)}{\pi}$$
where the inequality stems from the increasingness of $E\mapsto \theta_{n,E}(\lambda)$
where $\theta_{n,E}$ is solution of
$$\theta_{n,E}'=\frac{1}{(M^{(n)})^2}\cos(\theta_{n,E})^2+E\times(M^{(n)})^2\sin(\theta_{n,E})^2.$$
with $\theta_{n,E}(-\lambda)=0$.
We can obtain a lower bound in the same way which yields
\begin{align}
-1+\frac{\theta_{n,E}(\lambda)}{\pi}\leq \tilde{N}_{n,\lambda}(E)\leq \frac{\theta_{n,E}(\lambda)}{\pi}.\label{Dos6}
\end{align}
Together with \eqref{dirichlet}, this yields
\begin{align}
-2+\frac{\theta_{n,E}(\lambda)}{\pi}\leq N_{n,\lambda}(E)\leq 1+\frac{\theta_{n,E}(\lambda)}{\pi}.\label{Dos7}
\end{align}
As $M^{(n)}$ converges uniformly to $M$ on $[-\lambda,\lambda]$, $\theta_{n,E}$ converges uniformly on $[-\lambda,\lambda]$ toward $\theta_E$ which is the solution of the random ODE
$$\theta_E'=\frac{1}{M^2}\cos(\theta_E)^2+EM^2\sin(\theta_E)^2$$
with $\theta_E(-\lambda)=0$. This solution is well-defined on $[-\lambda,\lambda]$ because $(t,\theta)\mapsto \frac{1}{M_t^2}\cos(\theta)^2+EM_t^2\sin(\theta)^2$ is continuous and globally lipischitz in the second variable on any set of the type $[-r,r]\times\R$. Now, let us assume that $E$ and $\lambda$ are fixed. By Lemma \ref{approxspectrum}, $N_{n,\lambda}(E)$ converges toward $N_{\lambda}(E)$ almost surely. Moreover, $\theta_{n,E}$ converges uniformly toward $\theta_E$ almost surely. As a consequence, we can take the limit in \eqref{Dos7} which implies that for every $E>0$ and every $\lambda>0$, almost surely,
\begin{align}
-2+\frac{\theta_E(\lambda)}{\pi}\leq N_{\lambda}(E)\leq 1+\frac{\theta_E(\lambda)}{\pi}.\label{Dos8}
\end{align}
\textbf{Step 2: Study of $\theta_{E}$.} 
Let us shift $\theta_E$ in order to start from 0. For every $t\geq 0$, let us define $\tilde{\theta}_E$ such that for every $t\geq 0$, $\tilde{\theta}_E(t)=\theta_E(t-\lambda)$. This implies that $\tilde{\theta}_E$ is a solution of the random ODE
\begin{align}\tilde{\theta}_E'(t)=\frac{1}{M_{-\lambda}^2\tilde{M}_t^2}\cos(\tilde{\theta}_E)^2+EM_{-\lambda}^2\tilde{M}_t^2\sin(\tilde{\theta}_E)^2\label{eqdos1}
\end{align}
where $M_{-\lambda}=e^{B_{-\lambda}+\lambda/2}$ and $\tilde{M}_t=e^{B_{t-\lambda}-B_{-\lambda}-t/2}$ which is a geometric Brownian motion which is equal to 1 at $0$. Moreover, $\tilde{M}$ is independent of $M_{-\lambda}$.  Of course, $\tilde{\theta}_E$ depends on $\lambda$ (as well as $\theta_E$ actually) but we omit this in the notation for sake of convenience. For every $t\geq 0$, we define $\mathcal{F}_t$ as the right-continuous completion of $\sigma(M_{-\lambda},\tilde{M}_s,s\leq t)$. This is clear that $\tilde{\theta}_E$ is adapted with the filtration $(\mathcal{F}_t)_{t\geq 0}$. By \eqref{eqdos1}, remark that $\tilde{\theta}_E$ is stricly increasing. Therefore, it goes to infinity or it has a finite random limit $\Theta^*$. As $(1/\tilde{M}_t)_{t\geq 0}$ goes to infinity at exponential speed almost surely, \eqref{eqdos1} implies that $\Theta^*$ must be of the form $\pi/2+K^*\pi$ with $K^*$ which is random and possibly infinite. For every $k\in\N$, let us define the stopping times
$$T_k=\inf\{t\geq 0, \tilde{\theta}_E(t)=k\pi\}\text{ and }\tau_k=\inf\{t\geq 0,\tilde{\theta}_E(t)=\pi/2+k\pi\}. $$
By definition we have a sequence of inequalities
$$0=T_0<\tau_0<T_1<\tau_1<\cdots<T_{K^*}<\tau_{K^*}=+\infty.$$
Now, let us prove that $\tilde{\theta}_E$ goes to infinity almost surely. In order to do that, it is enough to show that for every $k\in\N^*$, $\E\left(\tau_k \right)<+\infty.$
Further, for technical purposes, we need to introduce other stopping times: for every $k\in\N$ and for every $n\in\N^*$, let us define
$$T_{k,n}^+=\inf\{t\geq 0, \tilde{\theta}_E(t)=k\pi+1/n\}\text{ and }T_{k,n}^-=\inf\{t\geq 0,\tilde{\theta}_E(t)=k\pi-1/n\}. $$
Now, we need a lemma whose proof is postponed to the end of this section.
\begin{lem}\label{lem:stoppingtime}
For every $n\in\N^*$,
$\E\left[T_{0,n}^+\right]<+\infty$.
\end{lem}
Thanks to this lemma, let us show that $\E\left[\tau_0\right]<+\infty$. Let us introduce $j_E:=-\cot(\tilde{\theta}_E)$. Observe that $j_E(t)$ explodes when $t$ is some stopping time $T_k$ for any $k$ and vanishes when $t$ is some stopping time $\tau_k$ for any $k$. Moreover, on each interval of type $]T_k, T_{k+1}[$, $j_E$ is solution of the following Riccati equation:
\begin{align}
j_E'=\frac{1}{M_{-\lambda}^2\tilde{M}^2}j_E^2+EM_{-\lambda}^2\tilde{M}^2. \label{Dos9}
\end{align}
Now, consider $n\in\N^*$ and let us work on the interval $]T_{0,n}^+,\tau_0[$. (For now, we do not know whether $\tau_0$ is finite or not.) Then, we define the stochastic process $\zeta_E:=\frac{j_E}{M_{-\lambda}^2\tilde{M}^2}$. This stochastic process explodes at some times. However, it is a continuous locally bounded process on $]T_{0,n}^+,\tau_0[$ almost surely. Therefore, we can use stochastic calculus on this interval.
 Recall that $\tilde{M}$ is a geometric Brownian motion on $[0,2\lambda]$ which is $1$ at time $0$. In particular, $\tilde{M}$ satisfies the SDE
\begin{align}
d\tilde{M}_t=\tilde{M}_tdB_t.\label{Dos10}
\end{align}
Therefore, on $]T_{0,n}^+,\tau_0[$, it holds that
\begin{align*}
d\zeta_E(t)=\frac{1}{M_{-\lambda}^2}\left(j_E\cdot d\left(\tilde{M}_t^{-2} \right)+\frac{j_E'(t)}{\tilde{M}_t^2}dt\right).
\end{align*}
Then, using Ito's formula together with \eqref{Dos9} and \eqref{Dos10}, we get
\begin{align*}
d\zeta_E(t)&=\frac{1}{M_{-\lambda}^2}\left(-\frac{2j_E(t)}{\tilde{M}_t^3}d\tilde{M}_t+\frac{3j_E(t)}{\tilde{M}_t^4}d\langle \tilde{M}\rangle_t+\frac{1}{\tilde{M}_t^2}\left(\frac{j_E(t)^2}{M_{-\lambda}^2\tilde{M}_t^2}+EM_{-\lambda}^2\tilde{M}_t^2 \right)dt\right)\\
&=-2\frac{j_E(t)}{M_{-\lambda}^2\tilde{M}_t^2}dB_t+3\frac{j_E(t)}{M_{-\lambda}^2\tilde{M}_t^2}dt+\frac{j_E(t)^2}{M_{-\lambda}^4\tilde{M}_t^4}dt+Edt\\
&=-2\zeta_E(t)dB_t+(\zeta_E(t)^2+3\zeta_E(t)+E)dt.
\end{align*}
Finally, on $]T_{0,n}^+,\tau_0[$, $\zeta_E$ is solution of the SDE:
\begin{align}
d\zeta_E(t)=-2\zeta_E(t)dB_t+(\zeta_E(t)^2+3\zeta_E(t)+E)dt.\label{Dos11}
\end{align}
Therefore, until its explosion, this is a Markov process with generator,
$$\mathcal{L}f=2z^2f''(z)+(z^2+3z+E)f'(z)$$
for any regular function $f$. Now, let us define a function $f_-$ on $\R_-$ such that for every $z\in\R_-$,
\begin{align}
f_-(z)=-\int_z^0\frac{e^{-u/2+E/(2u)}}{2|u|^{3/2}}\int_{-\infty}^u\frac{1}{|t|^{1/2}}e^{t/2-E/(2t)}dtdu.\label{fm}
\end{align}
One can check that $f_-$ is well-defined (the singularity at $0$ in the integral is not a real one) and smooth on $\R_-^*$, that  $f_-$ is bounded and has a finite limit at $-\infty$. Besides $f_-$ satisfies $\mathcal{L}f_-=1$ on $\R_-^*$. Then, by Propositions 2.6 and 2.2 in Chapter VII of \cite{Revuz_Yor}, conditionally on $\mathcal{F}_{T_{0,n}^+}$
, it holds that
$$\left(f_-\left(\zeta_E((T_{0,n}^++t)\wedge \tau_0)\right)-(T_{0,n}^++t)\wedge \tau_0\right)_{t\geq 0}$$
is a martingale with respect to the filtration $\left(\mathcal{F}_{T_{0,n}^++t}\right)_{t\geq 0}$. (The integrability of this martingale comes from the boundedness of $f_-$ and Lemma \ref{lem:stoppingtime}.) This implies that
\begin{align}
\E\left[(T_{0,n}^++t)\wedge \tau_0 \right]=\E\left[ T_{0,n}^+\right]+\E\left[f_-\left(\zeta_E((T_{0,n}^++t)\wedge \tau_0)\right) \right]-\E\left[f_-\left(\zeta_E(T_{0,n}^+)\right)\right].\label{Dos12}
\end{align}
By monotone convergence and dominated convergence (recall that $f_-$ is bounded), we can make $t$ go toward infinity in \eqref{Dos12} which implies that
\begin{align*}
\E\left[ \tau_0 \right]=\E\left[ T_{0,n}^+\right]+\E\left[f_-\left(\zeta_E( \tau_0)\right) \right]-\E\left[f_-\left(\zeta_E(T_{0,n}^+)\right)\right].
\end{align*}
Thanks to Lemma \ref{lem:stoppingtime} and the boundedness of $f_-$, this implies that $\E\left[\tau_0\right]<+\infty$. Actually, this is not difficult to iterate this method in order to prove that $\E\left[\tau_k\right]<+\infty$ for every $k\in\N^*$. In particular, this yields

\begin{align}
\underset{t\rightarrow+\infty}\lim \tilde{\theta}_E(t)=+\infty\hspace{1 cm}a.s.\label{Dos13}
\end{align}
\textbf{Step 3: Link with a renewal process.} On any interval of the form $[T_{k,n}^+,T_{k+1,n}^-]$, $\zeta_E$ is a continuous stochastic process which is the solution of the SDE \eqref{Dos11}. This is very tempting to say that $\zeta_E$ is a Markov diffusion process on $\R$ whose values are in $\R\cup\{\infty\}$ which would imply that the lengths of time $T_{k+1}-T_{k}$ are i.i.d as the succesive hitting times of $\infty$ by a Markov process. However, to our knowledge, the case of a process whose values can be $\infty$ is not contained in the theory of Markov diffusions. In order to avoid this theoritical problem, let us make a change of variable.
We define $\Gamma_E:=\frac{\zeta_E-i}{\zeta_E+i}$. The idea under this transformation is to map $\R\cup\{\infty\}$ on the unit circle. Remark that $\Gamma_E=1$ when $\zeta_E=\infty$. Moreover, $\zeta_E$ is continuous on intervals  of the form $]T_k,T_{k+1}[$. Furthermore, for every $k\in\N$, $\zeta_E(t)$ goes to $+\infty$ when $t$ goes to $T_k$ on the left and $\zeta_E(t)$ goes to $-\infty$ when $t$ goes to $T_k$ on the right. Therefore, $\Gamma_E$ is continuous on $\R_+$. Further, remark that $\zeta_E=-i\frac{\Gamma_E+1}{\Gamma_E-1}.$ Then, on each interval of the form $[T_{k,n}^+,T_{k+1,n}^-]$, by \eqref{Dos11} and Ito's formula,
\begin{align}
d\Gamma_E(t)&=\frac{2i}{(\zeta_E(t)+i)^2}d\zeta_E(t)-\frac{2i}{(\zeta_E(t)+i)^3}d\langle \zeta_E\rangle_t\nonumber\\
&=4i\frac{\zeta_E(t)}{(\zeta_E(t)+i)^2}dB_t+2i\left(\frac{\zeta_E(t)^2+3\zeta_E(t)+E}{(\zeta_E(t)+i)^2}-4\frac{\zeta_E(t)^2}{(\zeta_E(t)+i)^3} \right)dt.\label{Dos14}
\end{align}
Moreover, by a straightforward computation:
\begin{align}
\frac{1}{(\zeta_E+i)^2}=-\frac{(1-\Gamma_E)^2}{4},\hspace{0.5 cm}\frac{\zeta_E}{(\zeta_E+i)^2}=-\frac{i}{4}(1-\Gamma_E^2),\hspace{0.5 cm}\frac{\zeta_E^2}{(\zeta_E+i)^2}=\frac{(1+\Gamma_E)^2}{4}\nonumber
\end{align}
and
\begin{align}
\frac{\zeta_E^2}{(\zeta_E+i)^3}=-\frac{i}{8}(1+\Gamma_E)^2(1-\Gamma_E).\nonumber
\end{align}
Together with \eqref{Dos14}, this implies that on each interval of the form $]T_{k,n}^+,T_{k+1,n}^-[$,
\begin{align}
d\Gamma_E(t)\nonumber\\
&\hspace{-1 cm}=(1-\Gamma_E^2)dB_t+2i\left(\frac{1}{4}(1+\Gamma_E)^2-\frac{3i}{4}(1-\Gamma_E^2)-\frac{E}{4}(1-\Gamma_E)^2+\frac{i}{2}(1+\Gamma_E)^2(1-\Gamma_E) \right)dt.\label{Dos15}
\end{align}
Besides, $\Gamma_E$ is a continuous and uniformly bounded stochastic process. Therefore, we can prove that $\Gamma_E$ satisfies the SDE \eqref{Dos15} on the whole set $\R_+$ and not only on intervals of the form $[T_{k,n}^+,T_{k+1,n}^-]$ by using the continuity of $\Gamma_E$ and the dominated convergence theorem for the stochastic integral. (See Theorem 2.12 in \cite{Revuz_Yor}.) 
Now, we want to apply Definition 7.1.1 of \cite{Oksendal} in order to say that $\Gamma_E$ is an Ito diffusion. However, Definition 7.1.1 in \cite{Oksendal} requires that the functions inside the SDE are globally lipschitz. Here, this is not true at first sight. Nevertheless, $|\Gamma_E|$ is bounded by $1$. Consequently, we can consider compactly supported functions which coincide with $\gamma\mapsto 1-\gamma^2, \gamma\mapsto(1+\gamma)^2, \gamma\mapsto(1-\gamma)^2$ and $\gamma\mapsto(1+\gamma)^2(1-\gamma)$ on the unit circle. Then, these compactly supported functions are lipschitz and we can consider a modified version of the SDE \eqref{Dos15} with lipschitz coefficients which is satisfied by $\Gamma_E$. Therefore, according to Definition 7.1.1 in \cite{Oksendal}, $\Gamma_E$ is a complex Itô diffusion. Now, recall that the stopping times $(T_k)_{k\geq 0}$ are the successive hitting times of $k\pi$ by $\tilde{\theta}_E$. This implies that the stopping times $(T_k)_{k\geq 0}$ are the successive hitting times of $\infty$ by $\zeta_E$, that is, the successive hitting times of $1$ by $\Gamma_E$. As $\Gamma_E$ is an Itô diffusion, by Theorem 7.2.4 in \cite{Oksendal} it satisfies the strong Markov property. Then, by classical arguments, this implies that $(T_{k+1}-T_k)_{k\in\N}$ is an i.i.d sequence of random variables. Therefore
\begin{align}
\left(\left\lfloor\frac{\tilde{\theta}_E(t)}{\pi}\right\rfloor\right)_{t\geq 0}=\left(\#\{k\in\N^*, T_k\leq t\}\right)_{t\geq 0}=:(R_t)_{t\geq 0}\label{Dos155}
\end{align}
where $(R_t)_{t\geq 0}$ is a renewal process whose $k$-th interarrival time is $T_{k}-T_{k-1}$ for every $k\in\N^*$. One can refer to chapter 10 in \cite{Gristir} for more information on renewal processes.
 Combining \eqref{Dos155} with \eqref{Dos8} implies that for every $E>0$,
\begin{align}
\frac{N_{\lambda}(E)}{2\lambda}=\frac{R_{2\lambda}}{2\lambda}+o^{\P}_{\lambda}(1)\label{Dos16}
\end{align}
where $o^{\P}_{\lambda}(1)$ is a random variable which goes to 0 in probability when $\lambda$ goes to infinity.
Therefore, by Theorem 10.2(1) and \eqref{Dos16}, we know that
\begin{align*}
\frac{N_{\lambda}(E)}{2\lambda}\xrightarrow[\lambda\rightarrow+\infty]{\P}\frac{1}{\E\left[T_1\right]}.
\end{align*}
Finally, we only have to compute explicitely $\E\left[T_1\right]$.\\
\textbf{Step 4: Computation of $\E\left[T_1\right]$.} 
Let $n\in\N^*$. By \eqref{Dos11}, on the interval $[T_{0,n}^+,T_{1,n}^-]$, recall that $\zeta_E$ is a solution of the SDE
\begin{align}
d\zeta_E(t)=-2\zeta_E(t)dB_t+(\zeta_E(t)^2+3\zeta_E(t)+E)dt.\nonumber
\end{align}
As before, the generator associated with this SDE is given by
$$\mathcal{L}f=2z^2f''(z)+(z^2+3z+E)f'(z)$$
for any function $f$ which is enough regular.
Now, for every $z\in\R_-$, we define as before
$$f^*(z)=f_-(z)=-\int_z^0\frac{e^{-u/2+E/(2u)}}{2|u|^{3/2}}\int_{-\infty}^u\frac{1}{|t|^{1/2}}e^{t/2-E/(2t)}dtdu.$$
Moreover, for any $z\in\R_+$, let us define 
$$f^*(z):=\int_0^z\frac{e^{-u/2+E/(2u)}}{2u^{3/2}}\int_0^u\frac{1}{\sqrt{t}}e^{t/2-E/(2t)}dtdu.$$
This is not difficult to check that $f^*$ is well-defined, bounded, smooth everywhere, excepted at $0$ where it may be only $C^1$. Moreover, $\mathcal{L}f^*=1$ on $\R^*$. Consequently, conditionally on $\mathcal{F}_{T_{0,n}^+}$,
$$\left(f^*\left(\zeta_E((T_{0,n}^++t)\wedge T_{1,n}^-)\right)-(T_{0,n}^++t)\wedge T_{1,n}^-\right)_{t\geq 0}$$
is a martingale with respect to the filtration $\left(\mathcal{F}_{T_{0,n}^++t}\right)_{t\geq 0}$.
In particular, for every $n\in\N^*$ and for every $t\geq 0$,
\begin{align}
\E\left[f^*\left(\zeta_E((T_{0,n}^++t)\wedge T_{1,n}^-)\right)\right]-\E\left[f^*\left(\zeta_E(T_{0,n}^+)\right)\right]&=\E\left[(T_{0,n}^++t)\wedge T_{1,n}^- \right]-\E\left[T_{0,n}^+ \right].\label{Dos17}
\end{align}
By monotone convergence and dominated convergence theorem (recall that $f^*$ is bounded), we can make $t$ go to infinity in \eqref{Dos17}. This yields for every $n\in\N^*$
\begin{align}
\E\left[f^*\left(\zeta_E(T_{1,n}^-) \right)\right]-\E\left[f^*\left(\zeta_E(T_{0,n}^+)\right)\right]&=\E\left[T_{1,n}^- \right]-\E\left[T_{0,n}^+ \right].\label{Dos18}
\end{align}
Then, by monotone convergence and by dominated convergence again, we can make $n$ go to infinity in \eqref{Dos18} which yields
\begin{align*}
\int_{-\infty}^0\frac{e^{-u/2+E/(2u)}}{2|u|^{3/2}}\int_{-\infty}^u\frac{1}{|t|^{1/2}}e^{t/2-E/(2t)}dtdu+\int_{0}^{\infty}\frac{e^{-u/2+E/(2u)}}{2u^{3/2}}\int_0^u\frac{1}{\sqrt{t}}e^{t/2-E/(2t)}dtdu&\\
&\hspace{-14 cm}=\underset{n\rightarrow+\infty}\lim\hspace{0.1 cm}\E\left[f^*\left(\zeta_E((T_{1,n}^-) \right)\right]-\E\left[f^*\left(\zeta_E(T_{0,n}^+)\right)\right]\\
&\hspace{-14 cm}=\underset{n\rightarrow+\infty}\lim\hspace{0.1 cm}\E\left[T_{1,n}^- \right]-\E\left[T_{0,n}^+ \right]\\
&\hspace{-14 cm}=\E\left[T_1\right].
\end{align*}
Moreover, by using the change of variable $(u',t')=(\pm1/u,\pm1/t)$, then by using Fubini's theorem and finally by making a simple translation change of variables, one can check that for every $E>0$,
\begin{align}
\int_{-\infty}^0\frac{e^{-u/2+E/(2u)}}{2|u|^{3/2}}\int_{-\infty}^u\frac{1}{|t|^{1/2}}e^{t/2-E/(2t)}dtdu+\int_{0}^{\infty}\frac{e^{-u/2+E/(2u)}}{2u^{3/2}}\int_0^u\frac{1}{\sqrt{t}}e^{t/2-E/(2t)}dtdu&\nonumber\\
&\hspace{-14 cm}=\frac{1}{2}\int_0^{+\infty}\int_{0}^{+\infty}\frac{e^{-\frac{t}{2u(t+u)}-Et/2}}{\sqrt{u}\sqrt{t+u}}\times\frac{2u+t}{u(t+u)} dudt.\label{Dos19}
\end{align}
Now, we are going to compute explicitely the integral of \eqref{Dos19}. Let us make the change of variables 
$$t=\frac{s}{E},\hspace{0.4 cm} \frac{t}{u(t+u)}=\frac{s}{v^2}$$
which maps $(\R_{+}^*)^2$ onto itself.
This is equivalent to
$$(t,u):=\phi(v,s)=\left(\frac{s}{E},\frac{1}{2E}\left(-s+\sqrt{s^2+4Ev^2} \right)\right).$$
We can compute the Jacobian of $\phi$ which is 
$$Jac(\phi)=\frac{2v}{E\sqrt{s^2+4Ev^2}}.$$
Moreover, we can remark that
$$u(u+t)=\frac{v^2}{E}\text{ and }2u+t=\frac{1}{E}\sqrt{s^2+4Ev^2}.$$
Therefore,
\begin{align*}
\frac{1}{2}\int_0^{+\infty}\int_{0}^{+\infty}\frac{e^{-\frac{t}{2u(t+u)}-Et/2}}{\sqrt{u}\sqrt{t+u}}\times\frac{2u+t}{u(t+u)}dudt&\nonumber\\
&\hspace{-7 cm}=\frac{1}{2}\int_0^{+\infty}\int_0^{+\infty}e^{-s/(2v^2)-s/2}\frac{1}{(v^2/E)^{3/2}}\times\frac{1}{E}\sqrt{s^2+4Ev^2}\times \frac{1}{E}\frac{2v}{\sqrt{s^2+4Ev^2}}ds dv\nonumber\\
&\hspace{-7 cm}=\frac{1}{\sqrt{E}}\int_0^{+\infty}\int_0^{+\infty}\frac{e^{-s/(2v^2)-s/2}}{v^2}ds dv\\
&\hspace{-7 cm}=\frac{1}{\sqrt{E}}\int_0^{+\infty}\frac{2}{1+v^2}dv\\
&\hspace{-7 cm}=\frac{\pi}{\sqrt{E}}.
\end{align*}
This completes the proof of Theorem \ref{DOS}.
\end{proof}
\begin{flushleft}
Now, let us prove Lemma \ref{lem:stoppingtime}.
\end{flushleft}
\begin{proof}[Proof of lemma \ref{lem:stoppingtime}.]
Let $n\in\N^*$. By \eqref{eqdos1},
\begin{align}
\tilde{\theta}_E(T_{0,n}^+)-\tilde{\theta}_E(0)\geq\frac{1}{M_{-\lambda}^2}\int_0^{T_{0,n}^+}\frac{\cos(\tilde{\theta}_E(s))^2}{\tilde{M}_s^2}ds.\nonumber
\end{align}
Therefore,
\begin{align}
\frac{M_{-\lambda}^2}{n}&\geq\cos(1/n)^2\int_0^{T_{0,n}^+}e^{-2\tilde{B}_s+s}ds\nonumber\\
&\geq \cos(1/n)^2\underset{s\in[0,+\infty]}\inf\hspace{0.1 cm}e^{-2\tilde{B}_s+s/2}\times2\left(e^{T_{0,n}^+/2}-1 \right).\label{eqlem1}
\end{align}
Then, \eqref{eqlem1} and the independence between $M_{-\lambda}$ and $\tilde{B}$ yield
\begin{align}
\E\left[\left(e^{T_{0,n}^+/2}-1\right)^{1/8}\right]&\leq \frac{\E\left[M_{-\lambda}^{1/4}\right]}{(2n)^{1/8}\cos(1/n)^{1/4}}\E\left[\underset{s\geq 0}\sup \hspace{0.1 cm}e^{\tilde{B}_s/4-s/16} \right]\nonumber\\
&=\frac{\E\left[M_{-\lambda}^{1/4}\right]}{(2n)^{1/8}\cos(1/n)^{1/4}}\E\left[\underset{s\geq 0}\sup \hspace{0.1 cm}e^{\tilde{B}_s-s} \right]\label{eqlem2}
\end{align}
where in the equality, we used a change of time for the Brownian motion. Then, by Girsanov's theorem (see Theorem VIII.1.7 in \cite{Revuz_Yor}), it holds that for every $t\geq 0$,
\begin{align}
\E\left[\underset{s\in[0,t]}\sup \hspace{0.1 cm}e^{\tilde{B}_s-s} \right]&=e^{-t/2}\E\left[e^{ \underset{s\in [0,t]}\sup\hspace{0.1 cm} \tilde{B}_s-\tilde{B}_t}\right]\nonumber\\
&=e^{-t/2}\E\left[e^{|\tilde{B}_t|} \right]\nonumber\\
&\leq e^{-t/2}\E\left[e^{B_t}+e^{-B_t} \right]\nonumber\\
&\leq 2\label{eqlem3}
\end{align}
where in the second equality we used Theorem 2.23 in \cite{MoPer}. The combination of \eqref{eqlem2} and \eqref{eqlem3} concludes the proof of Lemma \ref{lem:stoppingtime}.
\end{proof}
\section{Acknowledgements}
We would like to thank the ANR grant LOCAL ANR-22-CE40-0012-02. We also acknowledge Cyril Labbé which made some very interesting remarks regarding some preliminary version of this paper. His pieces of advice enabled us to improve our paper significantly.
\bibliographystyle{alpha}
\bibliography{Bibli}
\end{document}